\documentclass[a4paper,11pt]{amsart}
\newfont{\cyr}{wncyr10 scaled 1100}

\usepackage{graphicx} 
\usepackage[left=2.7cm,right=2.7cm,top=3.5cm,bottom=3cm]{geometry}
\usepackage{amsthm,amssymb,amsmath,amsfonts,mathrsfs,amscd,yfonts}
\usepackage{turnstile}
\usepackage{comment}
\usepackage[latin1]{inputenc}
\usepackage[all]{xy}
\usepackage{latexsym}
\usepackage{float}
\usepackage{longtable}
\usepackage{mathrsfs}
\usepackage[usenames,dvipsnames]{color}
\usepackage{hyperref}
\usepackage{color} 
\usepackage{devanagari}
\usepackage{hieroglf}
\DeclareFontFamily{U}{BOONDOX-calo}{\skewchar\font=45 }
\DeclareFontShape{U}{BOONDOX-calo}{m}{n}{
  <-> s*[1.05] BOONDOX-r-calo}{}
\DeclareFontShape{U}{BOONDOX-calo}{b}{n}{
  <-> s*[1.05] BOONDOX-b-calo}{}
\DeclareMathAlphabet{\mathcalboondox}{U}{BOONDOX-calo}{m}{n}
\SetMathAlphabet{\mathcalboondox}{bold}{U}{BOONDOX-calo}{b}{n}
\DeclareMathAlphabet{\mathbcalboondox}{U}{BOONDOX-calo}{b}{n}
\DeclareMathAlphabet{\altmathcal}{OMS}{cmsy}{m}{n}
\usepackage[dvipsnames]{xcolor}
  \xdefinecolor{darkgreen}{rgb}{0,0.35,0}
\theoremstyle{theorem}
\newtheorem{theorem}{Theorem}[section]
\newtheorem{corollary}[theorem]{Corollary}
\newtheorem{lemma}[theorem]{Lemma}
\newtheorem{proposition}[theorem]{Proposition}
\newtheorem{conjecture}[theorem]{Conjecture}

\theoremstyle{definition}
\newtheorem{definition}[theorem]{Definition}

\newtheorem{examplewr}[theorem]{Example}

\theoremstyle{remark}
\newtheorem{obswr}[theorem]{Observation}
\newtheorem{remarkwr}[theorem]{Remark}

\newenvironment{remark}{\begin{remarkwr}\begin{upshape}}{\end{upshape}\end{remarkwr}}
\newenvironment{example}{\begin{examplewr}\begin{upshape}}{\end{upshape}\end{examplewr}}
\newcommand\mnote[1]{\marginpar{\tiny #1}}

\newcommand{\Bet}{\mathrm{B}}
\newcommand{\bE}{{\mathbb{E}}}
\newcommand{\bF}{{\mathbb{F}}}
\newcommand{\fO}{\mathfrak{o}}

\newcommand{\calH}{{\mathcal H}}

\DeclareMathOperator{\red}{red}

\DeclareMathOperator{\et}{et}

\DeclareMathOperator{\Ad}{Ad}
\DeclareMathOperator{\Spec}{Spec}
\DeclareMathOperator{\Log}{Log}

\newcommand{\fa}{\mathfrak{a}}

\newcommand{\TT}{{\mathbb T}}

\newcommand{\EN}{E_2^{^{_{(N)}}}\!}
\newcommand{\sigmaN}{\sigma_1^{^{_{_{_{_{\!(N)}}}}}}\!\!}

\newcommand{\dnem}{d}

\newcommand{\fq}{\mathfrak{q}}

\newcommand{\cC}{\mathcal C}

\newcommand{\cE}{\mathcal E}
\newcommand{\cU}{\mathcal U}
\newcommand{\fU}{\mathfrak U}

\newcommand{\cA}{\mathcal A}

\newcommand{\Q}{\mathbb{Q}}
\newcommand{\Z}{\mathbb{Z}}

\newcommand{\F}{\mathbb{F}}

\newcommand{\C}{\mathbb{C}}

\newcommand{\res}{\mathrm{res}}
\newcommand{\PP}{\mathbb{P}}

\newcommand{\Gal}{\mathrm{Gal\,}}

\newcommand{\GL}{\mathrm{GL}}

\newcommand{\Div}{\mathrm{Div}}

\newcommand{\End}{\mathrm{End}}

\newcommand{\Aut}{\mathrm{Aut}}

\newfont{\gotip}{eufb10 at 12pt}

\newcommand{\cO}{{\mathcal O}}
\newcommand{\cI}{{\mathscr I}}
\newcommand{\cM}{{\mathcal M}}
\newcommand{\cN}{{\mathfrak N}}

\newcommand{\cS}{{\mathcal S}}

\newcommand{\ra}{\rightarrow}
\newcommand{\lra}{\longrightarrow}

\newcommand{\Norm}{\mathrm{N}}

\newcommand{\SL}{{\mathrm {SL}}}
\newcommand{\SO}{{\mathrm {SO}}}

\newcommand{\dnm}{{\text{\dn{m}}}}

\newcommand{\Pic}{{\mathrm{Pic}}}

\newcommand{\R}{{\mathbb R}}
\newcommand{\M}{{\mathrm{M}}}

\newcommand{\cX}{\mathcal X}

\newcommand{\DD}{{\mathbf D}}

\DeclareMathOperator{\Hom}{Hom}

\newcommand{\Tp}{\overline{\mathbb{T}}}

\newcommand{\fm}{{\mathfrak m}}

\newcommand{\mat}[4]{\left(\begin{array}{cc}#1&#2\\#3&#4\end{array}\right)}
\newcommand{\smallmat}[4]{\bigl(\begin{smallmatrix}#1&#2\\#3&#4\end{smallmatrix}\bigr)}

\dedicatory{To Gopal Prasad, on his 75th birthday}

\thanks{ The first author was supported by an NSERC Discovery grant. The  second and fourth author were supported by the National Science Foundation under Grant No. DMS-1440140 while  they were in residence at the Mathematical Sciences Research Institute in Berkeley, California, during the Spring 2019 semester.   The second author was also supported by the National Science Foundation under Grants No. DMS-1701651 and DMS-2001369. The third author was supported by an ERC consolidator grant and Icrea Academia.  The last-named author   was supported by  NSF grant DMS-1931087.
 This project has received funding from the European Research Council (ERC) under the European Union's Horizon 2020 research and innovation programme (grant agreement No 682152). }

\include{thebibliography}

\begin{document}


\title[Derived Hecke algebras for dihedral weight one forms]{The derived Hecke algebra \\ for dihedral weight one forms}
\author{Henri Darmon, Michael Harris, Victor Rotger, Akshay Venkatesh}

\begin{abstract}
We study the action of the derived Hecke algebra in the setting of dihedral weight one forms, and   
 prove  a conjecture of the second- and fourth- named authors 
relating this action to certain Stark units associated to the symmetric square $L$-function.
The proof exploits the theta correspondence between various Hecke modules as well as ideas
of Merel and Lecouturier on higher Eisenstein elements. 
%
\end{abstract}


 \address{H. D.: Department of Mathematics and Statistics, McGill University, Montreal, Canada}
 \email{darmon@math.mcgill.ca}
 \address{K. H.: Department of Mathematics, Columbia University, New York, U.S. }
 \email{harris@math.columbia.edu}
 \address{V. R.: IMTech, UPC and Centre de Recerca Matem\`{a}tiques, C. Jordi Girona 1-3, 08034 Barcelona, Spain}
 \email{victor.rotger@upc.edu}
 \address{A. V.: School of Mathematics, Institute for Advanced Study, Princeton, U.S. }
 \email{akshay@ias.edu}


\maketitle
\setcounter{tocdepth}{1}

\tableofcontents

\section{Introduction}   

In the theory of modular forms, the case of weight one  is exceptional in several ways.  
The space of weight one forms, which can be interpreted as 
 the global sections of the Hodge line bundle $\omega$
  on a modular curve $X$,  does not admit a simple dimension formula.
This occurs precisely because the higher cohomology group $H^1(X, \omega)$
  can be nontrivial --- that is to say, the space of weight one forms
   manifests itself in two different cohomological degrees.

 A conjecture proposed in \cite{PV, GV, V2, HV}
 asserts that, in situations where spaces of automorphic forms occur across multiple cohomological degrees,
 the different degrees are related by means of a hidden action of a motivic cohomology group.
 The last mentioned paper \cite{HV}, in particular, formulates this story
 in the context of weight $1$ forms for the modular curve, and
translates the general conjectures into a numerically testable statement.
This statement, which is summarised below,
 is the main topic of this paper.

\subsection{The Shimura class}
While   the general definition of derived Hecke operators 
shall not be recalled here, 
one crucial ingredient in their construction is to take the
cup product with a certain distinguished  class in coherent cohomology, 
the so-called {\em Shimura class}.

As explained in detail in \cite[\S 3.1]{HV},  the Shimura class
attached to a prime $N\geq 5$
arises from the covering $X_1(N) \rightarrow X_0(N)$ of classical modular curves,
which (at least away from elliptic points) 
is {\'e}tale with deck group $(\Z/N\Z)^{\times}$
 and thus  furnishes 
an element  
$$\mathfrak{S}^\times\in H^1_{\rm et}(X_0(N),(\Z/N\Z)^\times\otimes \Z[1/6]).$$
(Here, and in what follows, modular curves  will 
be  regarded as  schemes over
the ring $Z=\Z[\frac{1}{6N}]$ to avoid any technical issues.)

Let $p>3$ be a prime, 
let $p^t$ be the highest power of $p$ dividing $N-1$, assume $t\geq 1$ and  fix a surjective {\em discrete logarithm} 
\begin{equation}
\label{eqn:discrete-log}
\log: (\Z/N\Z)^{\times} \ra \Z/p^t\Z.
\end{equation}
This choice determines a class $\mathfrak{S} := \log(\mathfrak{S}^\times) \in H^1_{\mathrm{et}}(X_0(N), \Z/p^t)$.
Restricting to  the fiber product    of $X_0(N)$ over $\Spec(Z)$
with $\Spec(\Z/p^t\Z)$, denoted $\bar{X}=X_0(N)_{/\Z/p^t\Z}$,  the resulting class can be pushed into Zariski cohomology,
using the inclusion of $\Z/p^t\Z$ into the {\'e}tale sheaf represented by $\mathbb{G}_a$:
in this way $\mathfrak{S}$ can be viewed as  a class in coherent cohomology. It is called the {\em Shimura class},  denoted
(by a slight abuse of notation)
$$ \mathfrak{S} \in H^1(\bar{X}, \mathcal{O}_{\bar{X}}) = \Hom(S_2(N),  \Z/p^t\Z),$$
where the last identification is provided by Serre duality, and $S_2(N)$ is the space of weight $N$
cusp forms (with $q$-expansions integral at $p$). 
  Note that $\mathfrak{S} $ depends on $N$, on $p$, and on the choice of  discrete logarithm.

 \subsection{The   main result}
   
Let  $g \in H^0(X_1(\dnem),\omega)$ be a Hecke new cusp form 
      of weight $1$,   level $\dnem$ and nebentype $\chi$, and let $g^*  \in H^0(X_1(\dnem),\omega)$ be the dual newform,  
   whose  Fourier expansion   is related to that of $g$  by complex conjugation, and whose automorphic representation is obtained from that of $g$ by twisting by $\chi^{-1}$. 
   Assume for simplicity that the primes $N$ and $p$
 do not divide $6 \dnem$.
 
 Let  $\rho_g: G_{\Q} \lra \GL_2(L) \simeq {\rm Aut}(V_g)$ be the odd $2$-dimensional Artin representation attached by Deligne and Serre  to $g$,
 acting on a two-dimensional
$L$-vector space $V_g$, for a suitable finite extension $L$ of $\Q$ (containing the fourier coefficients of $g$, and contained in a cyclotomic field).  Let  $\Ad(\rho_g)$ denote the $3$-dimensional subrepresentation of $\End_L(V_g)$
 consisting of $L$-linear endomorphisms of  $V_g$ of trace zero, equipped with the natural action of $G_\Q$ by conjugation.

Let $R$ be 
    the  ring of   integers of $L$, with $6N$ inverted. The product $g(z) g^*(Nz)$ is a weight $2$ cuspidal modular form of level $N \dnem$ with trivial Nebentypus character and coefficients in $R$, and can thus be viewed as
an element of  the space $S_2(N \dnem) = H^0(X_0(N \dnem), \Omega^1)$ of global regular differential forms. Let
$$G(z)  :=  \mathsf{Tr}^{N\dnem}_{N} (g(z) g^*(Nz))  \in S_2(N;R) = H^0(X_0(N)_{/R}, \Omega^1)$$
denote the trace of $g(z) g^*(Nz)$ to the space of modular forms of weight $2$ and level $N$.   
 
    The pairing between  $G$ and the Shimura class $\mathfrak{S}$  arising from Serre duality gives rise to a numerical invariant
$$ \langle G, \mathfrak{S} \rangle \in R/p^t,$$
see \S \ref{notn} for details. 
The conjecture of \cite{HV} relates this quantity  to the discrete logarithm
 of a suitable Stark unit attached to $g$, which we now proceed to describe.

The image of the integral group ring $R[G_\Q]$ in $\Ad(\rho_g)$ 
endows this space with a   Galois-stable $R$-sublattice, which is denoted
$\Ad(\rho_g)^\circ$, and whose $R$-linear dual is denoted $\Ad^*(\rho_g)^\circ$. 

 Let $H$ denote the finite extension of $\Q$  which is cut out
by $\Ad(\rho_g)$. Because complex conjugation acts with eigenvalues $1$, $-1$ and $-1$ on  this representation, 
Dirichlet's unit theorem asserts that
the  $R$-module
$$ U_g := (\cO_H^\times \otimes \Ad^*(\rho_g)^\circ)^{G_\Q}$$
is of rank one  (cf.~ Lemma 2.7 of \cite{HV}).
The choice of a prime $\mathcal{N}$  of $H$ above $N$ gives rise to a frobenius element $\sigma_N$,  whose  image under $\rho_g$  is a natural element of $\Ad(\rho_g)^\circ$ 
which is invariant under the conjugation action of $\sigma_N$.  Evaluation at $\sigma_N$ thus gives rise to a 
homomorphism from $\Ad^*(\rho_g)^\circ $ to $R$ which is $\sigma_N$-equivariant (for the trivial $\sigma_N$ action on $R$). 
Combining this evaluation   with the reduction modulo  $\mathcal{N}$
gives a  
``mod $N$ reduction map"
$$ \red_N: U_g  := ( \cO_H^\times \otimes \Ad^*(\rho_g)^\circ)^{G_\Q}  \lra  ((\cO_H/\mathcal{N})^\times \otimes  R)^{\sigma_N=1}  = 
   (\Z/N\Z)^\times \otimes R.$$

\noindent
A version of main conjecture of \cite{HV} (Conjecture 3.1 in loc.\,cit.) may be phrased as follows.
  \begin{conjecture} 
  \label{conj:HV}
There exists an integer $m =m_g \geq 1$ and
  $u_g \in U_g$ such that,      for all  primes $N$ and $p$ as above, 
$$ m \cdot \langle G, \mathfrak{S} \rangle = \log( \mathrm{red}_N(u_g)).$$
 \end{conjecture}

 Note that both sides of this conjectured identity   belong to $R/p^t$, and  that both
depend linearly  on the choice of  discrete logarithm  made in \eqref{eqn:discrete-log}. The 
 validity of  Conjecture \ref{conj:HV}   is thus independent of this choice.  
 Similarly, both sides of the conjecture are independent of the choice of $\mathcal{N}$.
 (In loc. cit. the conjecture  
was formulated differently,  and 
was slightly more precise about the primes dividing $m$; 
the version above is more explicit and is what we will prove in certain cases.)

This article presents  a  proof of Conjecture 
\ref{conj:HV}
when $g$ is {\em dihedral}  under certain simplifying assumptions
on ramification.
Recall that $g$ is said to be dihedral if the Galois representation
  $\rho_g$ is induced from a ray class character $\psi_1$ of the Galois group of an (imaginary or real) quadratic extension $K$ of $\Q$. 
   In that case $g= \theta_{\psi_1}$ is Hecke's classical theta series associated to $\psi_1$,
i.e. the modular form whose $L$-series is given by $L(K, \psi_1)$. We assume throughout that $\psi_1^2 \ne 1$, as this implies that $ \theta_{\psi_1}$ is cuspidal.

 Let $D$ denote the discriminant of $K$ and $\delta=(\sqrt{D})$ its different.
 \begin{theorem}
 \label{thm:main} 
 If $K$ is imaginary, assume that $D$ is an odd prime and that
 $\psi_1$ is unramified. If $K$ is real assume that $D$ is odd and that $\psi_1$  has conductor  dividing $\delta$. 
 Then Conjecture \ref{conj:HV} is true
for $g = \theta_{\psi_1}$.
\end{theorem}
\begin{remark}
The proof  of Theorem \ref{thm:main} described
in  Chapter~\ref{sec:cm-forms}
shows that the integer  $m$ of Conjecture \ref{conj:HV}
divides $24$ in the real case, and, in the imaginary case, that  it divides  $6$  unless the order of $\psi_1^2$ is a   power of a prime $\ell$, in which
case $m$ divides $6\ell$. 
No claim is made that 
these bounds for $m$ are  optimal; they are merely what comes directly out of the proofs.  
\end{remark}
 
 The key idea in the proof of Theorem \ref{thm:main}
  is to  express $G$ as the theta lift of an appropriate 
  {\em Heegner cycle},
 and to compute the image of $\mathfrak{S}$ under the adjoint of the theta lift as a combination of {\em higher Eisenstein elements}. While the latter computation is performed in full generality,  the expression for $G$ 
in terms of Heegner cycles has only been worked out 
 in a non-trivial simple scenario.

In particular, the ramification conditions force the following simplifying feature.
Let $\psi_1'$ denote the $\Gal(K/\Q)$-conjugate of $\psi_1$ and set $\psi= \psi_1/\psi_1'$.  Then
\begin{equation}
\label{eqn:key-decomp}
 \rho_g \otimes \rho_{g^*} = \mathrm{Ind}_{\Q}^K(1) \oplus \mathrm{Ind}_{\Q}^K(\psi)  \end{equation}
 decomposes as the direct sum of  the induced representations of two characters of
 $G_K$,  the trivial character $1$ and an 
 {\em unramified} character $\psi$. 
R.\,Zhang's forthcoming Ph.D thesis \cite{Zh}  will contain
 a proof under less restrictive ramification conditions for $K$ imaginary.  When $K$ is real, we envisage a method for calculating $G$ invoking Kudla-Millson theory but in order to cover the general case one needs to solve some issues related to the regularization of the theta lift from the split orthogonal group; cf.\,\S \ref{outline} for more details.

  \subsection{Trivial cases} 
  \label{sec:trivial-cases}
   If $K$ is imaginary quadratic and $N = \mathfrak N \cdot  \mathfrak N'$ splits in $K$, then $u_g$ belongs to $(\cO_H^\times \otimes \mathrm{Ind}_{\Q}^K(\psi))^{G_\Q}$, while 
  $\sigma_N$ belongs to $\mathrm{Ind}_{\Q}^K(1)$. 
  If $K$ is real quadratic and 
  $N$ is inert in $K$,  then 
  $u_g$ belongs to the unit group $\mathfrak{o}^\times$ of $K$,  on which $\sigma_N$  acts as $-1$.
   In both cases the regulator $ \mathrm{red}_N(u_g)$ 
  of Conjecture \ref{conj:HV} vanishes trivially. 
  
  This is consistent with the fact that the modular form
  $G$ is identically zero in these two scenarios.
  Indeed,
  the main theorem of  \cite{HaKu} asserts that, for all newforms $f$ of weight two on $\Gamma_0(N)$, 
\begin{equation}\label{Haiku}
\langle G,f\rangle^2  = C\cdot L(f,g, g^*,1),
\end{equation}
where $C$ is a product of local automorphic terms and $L(f,g, g^*,s)$ is the triple product $L$-series associated to $f$, $g$ and $g^*$. 
 The Artin formalism  applied to \eqref{eqn:key-decomp} implies that
\begin{equation}
\label{factorisation}
L(f,g, g^*,s) = L(f/K,s) \cdot L(f/K,\psi,s),
\end{equation}
where the two $L$-functions appearing on the right-hand side are the ones associated to the base change of $f$ to $K$, twisted by suitable characters. 
Since $(d,N)=1$,  both $ L(f/K,s)$ and $L(f/K,\psi,s)$ satisfy a functional equation with $s=1$ as center of symmetry and global sign $(\frac{-N}{K})$. This sign is $-1$, and hence
$$L(f/K,1) = L(f/K,\psi,1)=0.$$
 It follows that $\langle G,f\rangle=0$ for all   $f$, and hence that $ G=0$.

\medskip
 \subsection{Outline of the paper}\label{outline}
 The interesting cases of Theorem \ref{thm:main}
occur when
$(\frac{-N}{K}) = 1$, i.e., when
\begin{itemize}
\item $K$ is imaginary quadratic, and $N$ is inert in $K$;
\item $K$ is real quadratic, and $N$ is split in $K$.
\end{itemize}
The body of the article is devoted to the proof of
Theorem \ref{thm:main} in
these non-trivial cases, referred to as the {\em definite} and {\em indefinite} cases respectively.
The main idea is to transfer the computation to a suitable (definite, resp.\, indefinite) quaternion algebra  $B$ over $\Q$  by means of a theta lift $\Theta$:
$$ \Theta:  \mbox{modular forms on $B$} \rightarrow S_2(N).$$
This allows the identity in Conjecture \ref{conj:HV} to be recast on $B$. Indeed, 
 $G$ and $\mathfrak{S}$  are obtained via $\Theta$ from objects arising (respectively) from
  \begin{itemize}
  \item[(i)]  CM points or real quadratic closed geodesics; 
   \item[(ii)] Siegel units.
 \end{itemize}
Let us examine these two key ingredients in further detail.

Ingredient (i), in general,  takes the form  
\begin{equation}
\label{eqn:Heegner-cycles-intro}
 G \, =_{\mathfrak{S}} \, \Theta(Z_{K,\psi}),
\end{equation}
where $Z_{K,\Psi}$ is a suitable Heegner cycle and the symbol "$=_{\mathfrak{S}}$"
 means equality up to modular forms that pair to $0$ with the Shimura class (see \S \ref{notn} for some details). 
 More precisely, 
 $Z_{K,\Psi}$ is a formal linear combination of supersingular points in characteristic $N$,  obtained as a weighted combination of the mod $N$ reductions of CM elliptic curves in the definite case (Theorem \ref{thm:waldspurger-garrett-CM-general}), and  a linear combination of real quadratic geodesics  in the  homology of $X_0(N)$ in the indefinite case (Theorem   \ref{thm:RM-general}).  The equality \eqref{eqn:Heegner-cycles-intro}
 can be interpreted as coming from a certain see-saw  (Remark \ref{seesaw}) although we give a rather direct proof.
 
The content of (ii) is the computation of the image of the Shimura class under the  {\em adjoint}  (dual, in other words) of the theta lift. The outcome is an expression
\begin{equation}
\label{eqn:shimura-theta-intro}
 \Theta^*(\mathfrak{S})= \mbox{explicit {\it higher Eisenstein element} $\fU_N$}.
\end{equation}
 In   the definite case, $\fU_N$ is obtained by restricting a suitable Siegel unit to the supersingular locus in characteristic $N$. In the indefinite case  $\fU_N$ is built out of the  modular symbol arising from the (logarithmic derivative of the) same Siegel unit. 
 In all cases, the basic idea of proof is that $\Theta^*(\mathfrak{S})$ is {\em uniquely characterized by its behavior
 with respect to Hecke operators} and so can be proved to equal $\fU_N$, up to an irrelevant ambiguity,
 by a purely Hecke-theoretic computation. 
Most of the work for the proof of
\eqref{eqn:shimura-theta-intro}
 is given in Chapter \ref{sec:hee}, building on  ideas of Mazur, Merel and Lecouturier on (higher) Eisenstein elements and the classical theory of modular units and modular symbols.

\medskip
Combining \eqref{eqn:Heegner-cycles-intro} and \eqref{eqn:shimura-theta-intro} leads to an identity of 
 the form
 \begin{equation}\label{main-idea}
 \langle G, \mathfrak{S} \rangle= \langle \Theta(Z_{K,\psi}),  \mathfrak{S}\rangle =  
  \langle Z_{K,\psi}, \Theta^*(\mathfrak{S})\rangle = \langle Z_{K,\psi}, \fU_N \rangle.
\end{equation}
In the definite setting, the right-hand quantity can be interpreted as the discrete logarithm of an elliptic unit,  obtained by evaluating the Siegel unit attached to $\fU_N$ on the CM divisor attached to
$Z_{K,\psi}$. 
In the indefinite setting,  the regulator   
involves only   the logarithm of the fundamental unit of   $K$, and this 
 fundamental unit emerges in   $\langle Z_{K,\psi}, \fU_N \rangle$  
from the eigenvalues of certain hyperbolic matrices in $\Gamma_0(N)$.
  The details of these calculations,
  concluding with   the proof of  Theorem  \ref{thm:main},
 are supplied in 
 Chapter \ref{sec:cm-forms}.  
 
 \medskip
It is worth insisting on a crucial feature of the dihedral case, namely,  that  the desired  units can be constructed
  explicitly,  as 
 CM values of modular units in the definite case, or as eigenvalues of suitable matrices in $\SL_2(\Z)$ in the indefinite case. This is what accounts for  Stark's conjecture being  known
 for the adjoint $L$-functions of dihedral forms. It remains open, however,  
 for the adjoint $L$-functions of  so-called {\em exotic} 
 weight one forms with non-dihedral projective image. Although the existence and essential uniqueness
 of the predicted  unit   is still guaranteed by Dirichlet's unit theorem,  no analogue of the Kronecker limit formula 
 relating  it to $L$-functions attached to $g$  is available.
  The numerical 
 experiments described in \cite{HV}  test Conjecture
 \ref{conj:HV} numerically, but only in CM dihedral cases 
 that  now fall under the purview of Theorem \ref{thm:main}.
 The article \cite{marcil} provides numerical evidence for 
 Conjecture \ref{conj:HV} in several more interesting 
 instances where $g$ is exotic.

\begin{remark}\label{seesaw}
It may be helpful to indicate the see-saw that underlies the 
crucial computation \eqref{eqn:Heegner-cycles-intro}. We  emphasize, however,  that  the  proofs in \S \ref{thetaproducts} and
\S \ref{main-id-RM}  do not use this in any explicit way.\footnote{  However, the representation-theoretic perspective appears to be 
indispensable to treat cases in which the Hecke characters have more general ramification.} Here we will proceed purely formally. 

Set $$G(L(2) \times L(2))= \{ (g_1,g_2) \in \GL(2) \times \GL(2), \det(g_1) = \det(g_2)\}.$$
We examine the following see-saw:    
{\small  \begin{equation*}
\label{Rome}
 \xymatrix{
  \theta_{\psi_1^{-1}}^{(N)} \boxtimes \theta_{\psi_2^{-1}} \in & [ G(L(2) \times L(2)) ]
 &   [\mathrm{GO}(B) ]\sim (B^{\times}  \times B^{\times})/\GL(1) &   \ni \Theta(\mathfrak{S})  \\
\mathfrak{S} \in &   [\GL(2) ]\ar[u] \ar[ur]  &    \left[  \mathrm{G}(\mathrm{O}(W_1)\times \mathrm{O}(W_2))\right]     \ar[ul] 
  \ar@{^{(}->}[u]^{j} & \ni \psi_{_{12}} \times \psi_{_{12'}}
}
\end{equation*} }

  The arrow  $\Theta^*$ from lower-left to upper-right is a realization of the Jacquet-Langlands correspondence
  and we denote its formal adjoint simply by $\Theta$. 
  Of course $\mathfrak{S}$ is not in fact a characteristic zero modular form, but let us proceed as if it were;
  in the end the proof uses integral normalizations of the $\theta$-correspondence to get around this. 
  The see-saw principle  and adjointness respectively give
    $$ \langle\theta_{\psi_1^{-1}}^{(N)} \theta_{\psi_2^{-1}} ,  \mathfrak{S} \rangle = \langle Z_{K, \psi},  \Theta^*(\mathfrak{S}) \rangle = \langle \Theta(Z_{K, \psi}),  \mathfrak{S} \rangle,$$
  where $Z_{K, \psi}$ arises from pushing forward 
 $\psi_{_{12}} \times \psi_{_{12'}}$  under $j$.   In the special case $\psi_2=\psi_1^{-1}$ this recovers \eqref{eqn:Heegner-cycles-intro} from the point of view of the see-saw formalism. 
   
\end{remark}

 \subsection{Notation} 
 \label{notn}
 
 We will fix here some notation that is used throughout the paper.  This notation
 will also be introduced where we use it; we have gathered some of it
 here as a convenient reference.
 
 Throughout the paper, 
 $K$ will denote a quadratic field, with ring of integers $\mathfrak{o}$ and discriminant $D$. In \S \ref{thetaproducts}
 this field is imaginary, and in   \S \ref{main-id-RM} it is real. 
 The symbol $x \mapsto x'$   denotes the nontrivial automorphism of $K$,
 and we will allow ourselves to apply it to various associated constructions (elements of $K$, ideals, 
 characters, etc.)
 
The narrow class group of $K$  (i.e., the usual class group in the imaginary case) is denoted
by $\mathcal{C}$.
 In \S \ref{main-id-RM},  $\mathcal{C}_D$ will denote the ray class group of $K$
 allowing level $\delta$, the different ideal of $K$.
The symbols $\psi_1$ and $\psi_2$ denote
 characters of $\mathcal{C}_D$ with inverse central characters,  and we put
 \begin{equation} \psi_{_{12}} := \psi_1 \psi_2, \qquad \psi_{_{12'}}:=\psi_{1}\psi_2', \end{equation}
 which   in all cases descend to characters of $\mathcal{C}$. 
 In the special case  $\psi_2=\psi_1^{-1}$ which is  germane to the proof of Theorem \ref{thm:main},
 the above definitions simplify to
\begin{equation} 
\label{psidef} 
 \psi_{_{12}} = 1, \quad \psi_{_{12'}} = \psi_1/\psi_1' := \psi, \quad \mbox{ say}.\end{equation}
 
 We will use $L$ for a coefficient field for characters $\psi$ as above, i.e.
 $L$ is a number field containing the values of 
 $\psi: \mathcal{C} \mbox{ or } \mathcal{C}_D \rightarrow L^{\times}$.
 In this context, $R$ will denote a suitable ring of integers of $L$ (possibly with
 denominators at some primes). 
 
 We will often denote by $g$ (respectively $h$)
 the dihedral forms associated to $\psi_1$ (respectively $\psi_2$)
 with associated Galois representation $\rho_g: G_{\Q} \rightarrow \GL(V_g)$. 
In this situation, we will often denote
$$ G:= \mbox{trace to level $\Gamma_0(N)$ of  $\theta_{\psi_1^{-1}}(Nz) \theta_{\psi_2^{-1}}(z)$,}$$
which in particular becomes the trace of $\theta_{\psi_1}(z) \theta_{\psi_1^{-1}}(Nz)$ in the case $\psi_2=\psi_1^{-1}$.

The integer $N>3$   always denotes a prime, 
and $Z$ denotes the ring $\Z[\frac{1}{6N}]$. 
The modular curves $X_0(N)$ and $X_1(N)$
are understood\footnote{
Note that the distinction between ``stack'' or the associated coarse moduli scheme will make very little difference for our purposes;
the cover $X_1(N) \rightarrow X_0(N)$ is {\'e}tale only when considering the stacks,
but in any case we are interested only in the $(\Z/p^t)$-subcover which is also
{\'e}tale over the scheme.}
to be schemes over $Z$.
 We   write
  $H^1_{\et}(X_0(N))$ and $H^1_{\Bet}(X_0(N))$ to denote, 
 respectively, the {\'e}tale cohomology of $X_0(N)$ as a scheme over $Z$,
 and the Betti cohomology of the Riemann surface $X_0(N)(\C)$.  
 
  We define the space of cusp forms $S_2(N)$
 and the space of modular forms $M_2(N)$ as (free) $Z$-modules accordingly and define
 $$S_2(N)^{\vee} = \Hom(S_2(N), Z)$$
 as their $Z$-linear duals. 
 For $R$ a $Z$-algebra, 
  $S_2(N; R) := S_2(N) \otimes_{Z} R$ is similarly defined
 as the space of cusp forms with coefficients in $R$, and likewise for $M_2(N;R)$. 
  For an element $f \in M_2(N; R)$ we will denote by $f^{(d)} \in M_2(Nd; R)$
  the modular form with $q$-expansion $f(q^{d})$. 
  
  Let $p$ be an odd prime $\geq 5$ and $p^t$ the largest power of $p$ dividing $N-1$.
Fix
 a surjective  ``discrete logarithm"
$$\log:    (\Z/N\Z)^\times \longrightarrow \Z/p^t \Z,$$
where $p^t$ is the largest power of $p$ dividing $N-1$. 
Note that this logarithm  factors through the quotient  $G_N$ of \eqref{GNdef}
\begin{equation} \label{GNdef} G_N:= (\Z/N\Z)^\times/\langle \pm 1\rangle, \end{equation}
since $p$ is assumed to be odd.  This logarithm also extends uniquely
to the multiplicative group of the quadratic extension $\mathbb{F}_{N^2}$
and this extension will also be denoted by $\log$. All   formulas will be independent of the choice of logarithm:
both sides will scale the same way if one alters it.

  \begin{remark}\label{brokensymmetry} In the indefinite case the prime $N$ splits in $K$ and the correct definition of the discrete logarithm entails the choice of one of 
the two prime divisors of $N$.  This choice is denoted $\mathfrak{N}$ and the need to pin down a choice introduces a ``breaking of symmetry"
in the final formula Proposition \ref{prop:almost-final-rm} as well as in the intermediate calculations.    The choice intervenes at the beginning of \S \ref{sec:trace-indef-statement}.
\end{remark}

Given two  modular forms $F$ and $G$ of level $N$,  the notation
 $$F =_{\mathfrak{S}} G$$
 means that ``$F$ and $G$ have the same pairing with the Shimura class.'' 
(Strictly, the prime $p$ should have been included in the notation, but the choice of $p$ is understood to be fixed.)
More precisely, $F =_{\mathfrak{S}} G$ means that:
\begin{enumerate}
\item $F,G$ lie inside $M_2(N; R)$ for $R$ the ring of $p$-integers
in some algebraic number field, and 
\item the reductions $\bar{F}, \bar{G} \in M_2(N; R/p^t R) = H^0(X_0(N)_{R}, \Omega^1)$
have the same pairing with $\mathfrak{S}$ under the Serre duality pairing
\begin{equation} \label{gdp} H^0(X_0(N)_{R/p^t}, \Omega^1) \otimes H^1(X_0(N)_{R/p^t}, \mathcal{O}) \rightarrow R/p^t\end{equation}
obtained by taking the cup product to $H^1(X_0(N)_{R/p^t}, \Omega^1)$ and using the ``trace'' map on the latter. \footnote{In the current
setting, if $t > 1$, this can be defined using Grothendieck duality   for the   structural morphism
 $X_0(N)_{R/p^t} \rightarrow \mathrm{Spec} \ R/p^t$, after e.g. adding auxiliary level structure to remove any ``stacky'' structure. This identifies $H^1(\Omega^1)$
 with $\Hom(\mathrm{R}\pi_* \mathcal{O}, R/p^t)$, homomorphisms in the derived category of $R/p^t$ modules. In particular,
 each element of $H^1(\Omega^1)$ induces (by passage to $H^0$) a map $R/p^t \rightarrow R/p^t$,
 i.e. an element of $R/p^t$.}
 \end{enumerate}
This notion is readily seen to be  independent of $R$, i.e,
compatible with extension of scalars in the obvious sense.

 \subsection{Acknowledgements} 
 The authors are grateful to Jan Vonk for the valuable insights which guided their 
 approach   to proving the main theorem of Chapter  \ref{main-id-RM}. 
 They also  thank Frank Calegari, H\'el\`ene Esnault,  and Alice Pozzi  for stimulating discussions surrounding the topics of this paper.
 
 The anonymous referee made a very careful reading of the paper and made several valuable corrections and suggestions.
 We thank her or him  for their substantial effort, which has substantially improved the paper.

The authors are happy to express their appreciation to Gopal Prasad by dedicating this article to him.  
The second- and fourth-named authors would like to take this opportunity to add a few personal words:

M.H.: `` I vividly remember
Gopal's patient and enthusiastic explanation of his own work on arithmetic groups, 
and his sincere interest in my work, when we first met, just one year after my
Ph.D.   And I am  grateful for Gopal's generosity and friendship at all our subsequent meetings, 
on three continents, over the following decades."
 
 A.V.:  ``The clarity and beauty of Gopal's work on $p$-adic groups speaks for itself, and has influenced my work on too many occasions to readily enumerate.  
And it is with much warmth that  I   recall the kindness 
 that Gopal has showed  throughout my career; I still remember clearly that when I, as a graduate student, visited the University of Michigan,
 Gopal took the time    to speak with me and encourage my work.
 It is therefore with the greatest pleasure that I   dedicate this paper to him, with admiration for a great
mathematical career and the best wishes for the future.''

\section{A trace identity for definite theta series}
\label{thetaproducts}

In  \cite{baby-gz} and \cite{gross-zagier}, Gross and Zagier proved a formula for the central critical value (resp.\,derivative) of the  $L$-function attached to the convolution of a cusp form $f$ of weight $2$ and a theta series $g$ of weight $1$ associated to a character of an imaginary quadratic field $K$. A substantial step in the proof of both formulas is the computation, for a given prime $N$, of the trace of the product $g(z)E(Nz)$ of $g$ and a suitable Eisenstein series $E$ to the space of modular forms of level $N$.

In this note we need to carry the computation of the trace of the product $g(z)h(Nz)$ of two cuspidal theta series attached to ray class characters of $K$.
We did not attempt to adapt the computations of \cite[\S 7,8,9]{baby-gz} to the present setting, but rather follow a different method invoking the Weil representation of  $\SL_2(\mathbb{A}_f)$ (where $\mathbb{A}_f$ denotes the ring of finite ad\`eles) on the space of Schwartz functions on the ad\'elic points of the underlying quadratic spaces. 


\subsection{Setup on Heegner points} \label{Heegnersetup}

The computation will be carried out in slightly greater generality than in  Theorem  \ref{thm:main} 
 of the introduction. 
Let $K$ be an imaginary quadratic field of odd discriminant $D$ with maximal order $\mathfrak{o}$ and let $\mathcal{C} = \Pic(\mathfrak{o})$ denote the class group. For $I$ an ideal, note that the image $I'$ by conjugation defines the same class in $\mathcal{C}$ as $I^{-1}$. 
Denote by $a$ the number of distinct prime factors of $D$. 

Let $N$ be an odd prime with the property that
{\em $-N$ is a square modulo $D$.} When $D$ is {\em prime}, as assumed in the introduction,  this
condition is equivalent to $N$ being inert in $K$; in general it always implies that $N$ remains inert but
is a stricter condition.

Fix an algebraic closure $\overline{\mathbb{F}_N}$, and  let  $\mathbb{F}_{N^2}$ be the subfield of size $N^2$. 





Choose an auxiliary odd prime $q$ such that $q \equiv -N \, (\mathrm{mod}\, D)$.   An elementary computation of quadratic symbols shows that $q$ is split in $K$. Assume throughout that $q$ is such that the ideals $\mathfrak{q}$, $\overline{\mathfrak{q}}$ in $K$ above $q$ are {\em principal}. The existence of such $q$ is guaranteed by Cebotarev density theorem.


A  calculation with Hilbert symbols (cf.\,\cite[\S 2.1]{Vig1}) shows that
\begin{equation}\label{BKj}
B \simeq K + Kj, \quad \mbox{ with }
j^2 = -q N \quad\mbox{and} \quad  zj = j z', \ \ \mbox{ for all } z\in K.
\end{equation}
is the definite quaternion algebra over $\Q$ of discriminant $N$. Let $b \mapsto b'$ denote the canonical anti-involution on $B$; it coincides with complex conjugation when restricted to $K$.
Let $n(b) = b b'$ denote the reduced norm on $B$.

An {\em orientation} on a maximal order $\cM$ in $B$ is a choice of homomorphism 
$$\mathbf{o}: \cM \ra \F_{N^2} $$ 
onto $\mathbb{F}_{N^2}$.  Note that $\cM$ admits exactly two possible orientations. Two oriented maximal orders $\vec{\cM}_1 =  (\cM_1,\mathbf{o}_1)$, $\vec{\cM}_2 =(\cM_2,\mathbf{o}_2)$ are equivalent if there exists an isomorphism $i:  \cM_1  \ra \cM_2$ satisfying $\mathbf{o}_1 = \mathbf{o}_2 \circ i$. 

 Write $\Pic(B)$ for the set of equivalence classes of oriented maximal orders. 
 By a classical result of Deuring (cf.\,\cite[\S 42.3]{Vo}), $\Pic(B)$ is in bijection with the
set $\cE$ of isomorphism classes of supersingular elliptic curves over $\overline{\F}_N$
as follows: we associate to an elliptic curve $E$ the order $\mathrm{End}(E)$, which
acquires an orientation by considering its action on the tangent space.

Fix a basepoint $\vec{\cM} \in \Pic(B)$ containing $\mathfrak{o} \oplus \mathfrak{o} j$.  Define the map
\begin{equation} \label{iotadef0}
\iota: \Pic(\mathfrak{o})  \lra \Pic(B)
\end{equation}
that takes an ideal class $I$  
 to the oriented maximal order $ \iota(I) = I^{-1} \vec{\cM} I$.   


\subsection{Statement of the trace identity} \label{tracedefstatement}

  Define $\Div(\cE)$ to be the module of $\Z$-valued functions on $\Pic(B)$, equipped with its natural action of the Hecke algebra $\TT$ as described e.g.\,in \cite[\S 4]{baby-gz}. 
If $\vec{\cM} \in \Pic(B)$, let $\cM$ denote the underlying unoriented order and set $w_{\vec{\cM}} = \frac{1}{2} |\cM^\times|$. 
Denote by $e_x \in \Div(\cE)$  the characteristic function of $x \in \Pic(B)$ and set as in the introduction $\Sigma_0 = \sum_x \frac{e_x}{w_x} \in \Div(\cE)\otimes \Q$. 
 The space $\Div(\cE)$ is endowed with a natural   symmetric  bilinear form
\begin{equation} \label{divpair} \langle \ , \ \rangle: \Div(\cE) \times \Div(\cE) \rightarrow \Z, \qquad
\langle e_x, e_y \rangle :=  w_x \delta_{x y}, \end{equation}
relative to which the Hecke operators $T_\ell$  are 
self-adjoint {\em for all $\ell$}, including for $\ell=N$.

The Jacquet-Langlands correspondence identifies $\Div(\cE)$
and $M_2(\Gamma_0(N))$ as Hecke modules.  
This identification can be described explicitly by means of the $\Theta$-correspondence, which is the Hecke-equivariant map 
\begin{equation}\label{Theta-cor}
 \Theta: \Div(\cE) \otimes_{\mathbb{T}} \Div(\cE) \rightarrow M_2(\Gamma_0(N))
\end{equation}
given by (cf.\,e.g.\,\cite{Em} and  \cite[Prop. 5.6]{baby-gz})
\begin{equation}\label{Theta-def}
 \Theta(\phi_1 \otimes \phi_2) = \frac{1}{2}  \langle \phi_1, \Sigma_0 \rangle \langle \phi_2, \Sigma_0 \rangle + \sum_{m \geq 1} \langle \phi_1, T_m \phi_2 \rangle q^m.
 \end{equation}

 \begin{remark}
 \label{switch}  
 Formula \eqref{Theta-def} makes it clear that 
 $$ \Theta(\phi_1 \otimes \phi_2)  =  \Theta(\phi_2 \otimes \phi_1),$$
 because each $T_m$ is self-adjoint, including for $m=N$. 
 \end{remark}
  
Given a character $\psi: \Pic(\mathfrak{o}) \ra L^\times$ with values in some finite field  extension $L/\Q$, define 
\begin{equation} \label{psibrackdef} [\psi] := \iota_*(\psi)  = \sum_{I \in \Pic(\mathfrak{o})}
\psi(I)  \iota(I)
\in \Div(\cE)\otimes L.
\end{equation}



\label{sec:trace-def-statement}
The main result of this section is the following. Let $\theta_{\psi}$ denote the theta series associated to $\psi$ as recalled in \eqref{classicalO2} below. Note also that $\theta_{\psi} = \theta_{\psi^{-1}}$ because characters of $ \Pic(\mathfrak{o})$ are anticyclotomic in the sense that $\psi' = \psi^{-1}$;
this accounts for the discrepancy in phrasing between the statement below, and the analogous 
Theorem \ref{thm:RM-general} in the RM scenario.

 \begin{theorem}
\label{thm:waldspurger-garrett-CM-general}   
Let $\psi_1$ and $\psi_2$ be 
characters of $\mathcal{C}$ and 
 let $\theta_{\psi_i}$ be the newforms associated to $\psi_i$ (equivalently: to $\psi_i^{-1}$).  Put 
  $\psi_{_{12}} = \psi_1 \psi_2, \psi_{_{12'}} =\psi_{1}\psi_2'$.
Then  there exists $p_0$ such that, for any $N$ and any $p \geq p_0$ with $p\mid N-1$:
\begin{equation}\label{eqn:Heegnercycles} 
 \mathsf{Tr}^{N D}_{N} (\theta_{\psi_1}(Nz) \theta_{\psi_2} (z))  \,  =_{\mathfrak{S}} \,  4 \cdot \Theta([\psi_{_{12}}] \otimes [\psi_{_{12'}}]).
\end{equation}
where, as in \eqref{eqn:Heegner-cycles-intro}, the notation ``$f =_{\mathfrak{S}} g$" means  that both modular forms have the same pairing with the Shimura class of level $N$. 
If $D$ is prime, \eqref{eqn:Heegnercycles} is a strict identity (not just up to $\mathfrak S$) for all primes $p$. 
 \end{theorem} 
 
 
We expect a similar trace identity to hold for general ray class characters $\psi_1$, $\psi_2$ of  $K$ with opposite central character, which amounts to allowing the ring class characters $\psi_{_{12}}$, $\psi_{_{12'}}$ to have arbitrary conductor $c\geq 1$. In such generality however we do not expect the constant to be as simple as $C=4$ and \eqref{eqn:Heegnercycles}  should hold up to a suitable constant $C=C(\psi_1,\psi_2)$ that depends on $\psi_1$, $\psi_2$ but not on $N$. The reader is referred to R. Zhang's forthcoming Ph.D thesis \cite{Zh} for the proof in greater generality in the adelic language.

Theorem  \ref{thm:waldspurger-garrett-CM-general} above and Theorem \ref{thm:RM-general} below cover the simplest non-trivial settings in both definite and indefinite cases.  The proofs are different because we chose to be as direct as possible in each case and avoid repetition, but the approaches in \S \ref{thetaproducts} and \S \ref{main-id-RM} are in some ways complementary.  

%
%
%
\subsection{Summary of the proof}\label{localCM}

  Theorem \ref{thm:waldspurger-garrett-CM-general} will be proved 
  subject to three Propositions  given below;
these will be proved in the remaining subsections.

Define $\cO \subset B$ via
\begin{equation}\label{Ocq}
\cO = \cO(q) :=  \mathfrak{o} \oplus  \mathfrak{o} j.
\end{equation}
An elementary computation using \cite[1.4.7]{Vig1} shows that $\cO = \cO(q)$ has square-free discriminant $D N q$, and 
therefore (cf.\,\cite[3.5.3]{Vig1}) $\cO$ is an Eichler order, that is to say, the intersection of two maximal orders. 
Let us fix now and for the rest of this section  
   a maximal order $\cM \supset \cO$ as well as an orientation on it. 
All other orders containing $\cO$ can be obtained from $\cM$ as 
  \begin{equation}
   \label{MDdef} 
   \cM_d := \mathfrak{d}^{-1} \cM \mathfrak{d},\end{equation} 
  where $d$ ranges over positive divisors $d\mid Dq$  and $\mathfrak{d}$ is an ideal in $\mathfrak{o}$ of norm $d$. This is because locally at every prime $\ell$ dividing $Dq$ there are exactly two local maximal orders containing $\cO \otimes \Z_{\ell}$: one is obtained from the other by conjugating by any element of norm $\ell$ normalizing $\cO\otimes \Z_{\ell}$ (cf.\,e.g.\,\cite[\S 3.5]{Vig1}).

If $I_1, I_2$ are ideal classes for $\mathfrak{o}$ we can form   
$$ I_1 \cO I_2 = I_1 I_2 \oplus I_1 I_2' j \subset B.$$
By definition, the left hand side means the additive subgroup of $B$ generated by  all threefold products $i_1 \cdot o \cdot i_2$. 


We regard $K$ and $B$ as quadratic spaces by means of the norm and reduced norm respectively.  
For every ideal $I$ in either $V=K$ or $B$, let 
\begin{equation}\label{thetaL}
\theta_I= \theta(I) = \sum_{a\in I} q^{\frac{n(a)}{n(I)}}
\end{equation}
denote the theta series associated to $I$; here $n(I)$ stands for the single positive generator of the ideal of $\Q$ spanned by the norms of all elements in $I$.  The theta series $\theta_I$ is a modular form of weight $[V:\Q]/2$. With this normalization, $\theta_I$ only depends, in the case $V=K$, on the class of $I$ up to principal ideals, since $\theta_I = \theta_{Ix}$ for any $x\in K^\times$. 
Moreover, for any character of $\cC$, 
\begin{equation}
\label{classicalO2}
\theta_{\psi} = \sum_{I\in \cC} \psi(I)^{-1} \theta_{I}
\end{equation}
is the new theta series associated to $\psi$, a classical modular newform of weight $1$, level $D$ and nebentype character $\chi_K$, the quadratic Dirichlet character associated to $K/\Q$.  
 (As mentioned above, in the current situation  one could omit the inverse on the right hand side, but the formula above  is valid under less restrictive ramification conditions and facilitates
 comparison with the RM case.)  

For any $d\geq 1$, recall that
 $\theta^{(d)}(q) := \theta(q^d)$.
We will include forward references to Propositions in the RM case that play a similar
role, although because of the slightly different setups the statements are not entirely parallel.

\begin{proposition}\label{traceNDN0}  (See \S \ref{NDN0proof}, cf. also
Prop.   \ref{prop:trace-theta}). 
For any pair of classes $I_1, I_2$ of $\cC$,
we have
\begin{equation}  \mathsf{Tr}^{DN}_N \theta(I_1 I_2) \theta^{(N)}(I_1 I_2') = \frac{1}{2} \sum_{d|Dq} \theta(I_1 \cM_d I_2),
\end{equation}
where, on the left, $\mathsf{Tr}$ is the trace from level $\Gamma_0(DN)$ to level $\Gamma_0(N)$,
and $\cM_d$ is as in \eqref{MDdef}. 
\end{proposition}

\begin{proposition} \label{IJprop}
(See \S \ref{IJpropproof}, cf. also Prop. \ref{prop:main-vonk}). 
For any pair of ideal classes $I,J$ we have
\begin{equation}\label{claim}
\theta(J' \cM I)= 2 \cdot \Theta(e_{I} \otimes e_{J}).
\end{equation}

Here $e_I, e_J$ are as in \S \ref{tracedefstatement}, where
we use $\iota$ of \eqref{iotadef0} to identify $I, J$ with elements of $\mathrm{Pic}(B)$. 
\end{proposition}

For every quadratic character $\chi$ of $\cC$, set 
\begin{equation}\label{GDN}
G_{DN}(\chi) = \theta_{\psi_1^{-1} \chi}^{(N)}  \, \cdot \, \theta_{\psi_2^{-1} \chi} \quad \mbox{ and } \quad G_{DN} := \sum_{\chi \in (\cC/\cC^2)^*}  G_{DN}(\chi).
\end{equation}
In the case when $D$ is prime, the only quadratic character is trivial, so 
$G_{DN}(\chi) = G_{DN}$, and 
 the following Proposition is vacuous: 
\begin{proposition} 
\label{Shimura} 
(See \S \ref{Shimuraproof}.) With the notation of \eqref{eqn:Heegner-cycles-intro}
$$
\mathsf{Tr}^{DN}_{N}(G_{DN}(\chi)) \, \sim_{ \mathfrak{S}} \,  \mathsf{Tr}^{DN}_{N}(G_{DN}(\chi')),
$$
for all quadratic characters $\chi, \chi'$ of $\cC$, so long as the prime $p$ is sufficiently large relative to $D$. 
\end{proposition} 

Let us see how these three results imply the theorem. 
%
%
  For every choice of quadratic character $\chi$ 
we have
\begin{equation}\label{eq1}
G_{DN}(\chi) 
=\sum_{I, J \in \cC} \chi(IJ) \psi_1(I)\psi_2(J) \theta_{I}^{(N)} \theta_{J} .
\end{equation}
But $\sum_{\chi} \chi(IJ)$ is zero unless
 $IJ$ is a square inside 
$\cC$; in that case, it equals $\# \cC[2]$. Moreover, any pair $(I,J)$ with $IJ\in \cC^2$ is of the form $(I_1I_2', I_1I_2)$
for precisely $\# \cC[2]$ pairs $(I_1, I_2)$ and then
$\psi_1(I) \psi_2(J) =  \psi_{12}(I_1) \psi_{12'}(I_2').$
 It follows that
\begin{equation} G_{DN} =
\sum_{I_1, I_2\in \cC} \psi_{_{12}}(I_1)\psi_{_{12'}}(I_2') \theta(I_1 I_2) \theta^{(N)}(I_1 I_2').
\end{equation}

 Using  Proposition \ref{traceNDN0} we get
\begin{equation}\label{G}
\mathsf{Tr}^{DN}_N(G_{DN}) = \frac{1}{2} \sum_{I_1, I_2 \in \cC}\psi_{_{12}}(I_1)\psi_{_{12'}}(I_2') \big( \sum_{d\mid Dq}  \theta(\mathfrak{d}^{-1} I_1 \cM I_2 \mathfrak{d}) \big).
\end{equation}

Note however  that all terms for every fixed $d$ in \eqref{G} are equal: this follows after reindexing $(I_1,I_2) \leftrightarrow (\mathfrak{d}^{-1} I_1,I_2 \mathfrak{d})$ and recalling that the class of the ideal $\mathfrak{d}$ has order $2$ in $\cC$ when $d\mid D$, while the class of $\mathfrak{q}$ is trivial.
The number of such terms is equal to $2^{a+1}$, with $a$ the number of prime factors of $D$. 
Hence  
\begin{eqnarray}
\label{GG}
\mathsf{Tr}^{DN}_N(G_{DN}) &=& 2^{a} \sum_{I_j} \psi_{_{12}}(I_1)\psi_{_{12'}}(I_2')  \theta( I_1 \cM I_2) \\
\nonumber
&=& 2^{a}  \sum_{I_j} \psi_{_{12}}(I_1')\psi_{_{12'}}(I_2')  \theta(I_1' \cM I_2).
\end{eqnarray}
 Proposition \ref{IJprop},
as well as the symmetry of $\Theta$ in its arguments, 
can be invoked 
 to transform the right hand side, 
to get  
$$\mathsf{Tr}^{DN}_N(G_{DN}) = 2^{a+1}  \sum \psi_{12}^{-1}(I_1) \psi_{12'}^{-1}(I_2) \Theta(e_{I_2} \otimes e_{I_1})=2^{a+1} \Theta([\psi_{_{12}}^{-1}] \otimes [\psi_{_{12'}}^{-1}]).$$

This directly yields Theorem \ref{thm:waldspurger-garrett-CM-general} when $D$ is prime, as in that case the order of $\cC$ is odd and hence $G_{DN}=G$.
When $D$ is composite, Proposition \ref{Shimura} shows that
 we can replace $G_{DN}$ on the left by $2^{a-1} G_{DN}(1)$
 if we are only interested in pairing with the Shimura class, 
since $\cC[2]$ has rank $a-1$ by genus theory (cf.\,e.g.\,\cite[\S 13]{Harvey}). Unwinding the notation
  $$ \mathsf{Tr}^{N D}_{N} (\theta_{\psi_1^{-1}}(Nz) \theta_{\psi_2^{-1}} (z))  \sim_{ \mathfrak{S}} \, 4 \cdot  \Theta([\psi^{-1}_{_{12}}] \otimes [\psi^{-1}_{_{12'}}]).$$
%
%
This proves Theorem \ref{thm:waldspurger-garrett-CM-general} after  recalling that $\theta_{\psi_i} = \theta_{\psi_i^{-1}}$. 

%
%
%

\subsection{Proof of Proposition \ref{traceNDN0}} \label{NDN0proof}
We must show that
 \begin{equation}  \mathsf{Tr}^{DN}_N \theta(I_1 I_2) \theta^{(N)}(I_1 I_2') = \frac{1}{2} \sum_{d|Dq} \theta(I_1 \cM_d I_2).
\end{equation}

 Let $T_q$ denote the Hecke operator at $q$ and $\mathsf{Tr}^{N_2}_{N_1}$ the trace map from modular forms of level $N_2$ to level $N_1$, for any $N_1\mid N_2$. Then
  $$\mathsf{Tr}^{DNq}_{DN}  \theta^{(Nq)}_J = T_q \cdot \theta^{(N)}_J = \theta^{(N)}_{J\mathfrak{q} }+
  \theta^{(N)}_{J \mathfrak{q}'} =2 \theta^{(N)}_J,$$
  where the second equality follows from e.g.\,\cite[\S 2]{Kani}, and the third
  since we are supposing that $\mathfrak{q}$ is principal. 
Hence 
\begin{equation*}  
\mathsf{Tr}^{DNq}_N \theta_{J_1}  \theta^{(Nq)}_{J_2} = 
\mathsf{Tr}^{DN}_N \theta_{J_1} \left( \mathsf{Tr}^{DNq}_{DN} \theta^{(Nq)}_{J_2}  \right)  =  
 2 \mathsf{Tr}^{DN}_N \theta_{J_1}    \theta^{(N)}_{J_2} . \end{equation*}

Taking $J_1 = I_1 I_2, J_2 = I_1 I_2'$ and switching sides we get 
%
\begin{equation}\label{DNq}  
\mathsf{Tr}^{DN}_N \theta_{I_1 I_2}    \theta^{(N)}_{I_1 I_2'}= \frac{1}{2} \mathsf{Tr}^{DNq}_N \theta_{I_1 \cO I_2}.
\end{equation}
 where we noted that
 $ \theta_{I_1 \cO I_2} =    \theta_{I_1 I_2 \oplus  I_1 I_2' j}   =  \theta_{I_1 I_2} \theta^{(Nq)}_{I_1 I_2'}$
since   $j^2= -qN$. So Proposition \ref{traceNDN0} reduces to
\begin{proposition}\label{traceNDN} 
For any pair of classes $I_1, I_2$ of $\Pic(\mathfrak{o})$,
\begin{equation} 
 \mathsf{Tr}^{DNq}_N (\theta_{I_1 \cO I_2}) = \sum_{d|Dq} \theta(I_1 \cM_d I_2).
\end{equation}
\end{proposition}

In order to prove Proposition \ref{traceNDN}, note  that  -- with $V=B$ or $K$ as before -- the rule that associates to every lattice $L$ a modular form $\theta_L$ of weight $\dim(V)/2$ may be extended to the space of Schwartz functions on $V\otimes \mathbb{A}_f$ where $ \mathbb{A}_f$ denotes the ring of finite ad\`eles of $\Q$ (cf.\,e.g.\,\cite{GH} for background).
Namely, such a function may be identified with a function $\Phi$ supported on some lattice $L \subset V$ and constant on the cosets of a sublattice of $L$. We can form the $\theta$-function
\begin{equation}\label{theta-Schwartz} \theta_\Phi := \sum_{z \in V}  \Phi(z) q^{Q(z)}.
\end{equation}
with $Q$ the norm form (or, as we will actually use, a rescaling of it). 
This is compatible with the previous definition in the sense that
$\theta_{[L]}=\theta_L$, where
 $[L]$ is the characteristic function of 
the closure of a lattice $L$ inside $V \otimes \mathbb{A}_f$.

The function $\Phi \mapsto \theta_\Phi$ is equivariant for the action of $\SL_2(\mathbb{A}_f)$
on Schwartz functions arising from the Weil representation on one side, and the natural action on the space of modular forms on the other; this is a straightforward consequence of the adelic interpretation
of $\theta$-series; we will
give a more detailed sketch of a similar equivariance in the more complicated RM setting in \S \ref{bthetadef}. 
This equivariance implies that the trace, from level $DNq$ to $N$, of 
 $\theta_{I_1 \cO I_2}$,
 can be computed by first computing
 the corresponding trace 
\begin{equation}
 \label{uggtr} 
 \mathsf{Tr}^{DNq}_N   \mbox{ of the Schwartz function $[I_1 \cO I_2]$}.
 \end{equation}
Proposition \ref{traceNDN} thus follows after
computing the trace of $[I_1 \cO I_2]$ 
with reference to the $\SL_2(\mathbb{A}_f)$-action on Schwartz functions, taking into account that we have an equality of rescaling factors
$ n(I_1 \cM_d I_2) = n(I_1 \cO I_2)$ (this can be readily deduced from the fact that $I_1, I_2$ are locally principal).

This Weil representation is a tensor product
of representations of $\SL_2(\mathbb{Q}_{\ell})$ on Schwartz functions on $V \otimes \Q_{\ell}$.
We review the formulas in  \S \ref{bthetadef} and will summarize them here. 
Given a prime $\ell$, let $\mu: \Q_{\ell} \rightarrow \C^\times$ 
be the restriction of the standard character of $\mathbb{A}/\mathbb{Q}$
which is given by $x \mapsto e^{2 \pi i x}$ on $\mathbb{R}$
and is trivial on each $\mathbb{Z}_{p}$. 
In particular $\mu$ is  trivial on $ \Z_{\ell}$ but not on $\ell^{-1} \Z_{\ell}$. For any $t \in \Q_\ell$, denote $m(t)= \smallmat 1t01$ and set $w=\smallmat 01{-1}0$. Given a Schwartz function $\Phi_\ell$ on $V \otimes \Q_\ell$:
\begin{eqnarray}\label{Weil}
  m(t)\cdot \Phi_\ell(x) = \mu(t  \langle x, x \rangle) \Phi_\ell(x) & \mbox{ for any }\, t \in \Q_\ell, \\  \nonumber
w \cdot \Phi_\ell(y) =   \gamma_{\ell} \int_{V\otimes \Q_\ell}
  \Phi_\ell(x) \mu(\langle y,x \rangle) dx. &
\end{eqnarray}
where, in particular, $\gamma_{\ell}=1$ for $\ell$ not dividing $N$.  Here $dx$
is taken to be the self-dual Haar measure.

The desired trace from \eqref{uggtr} can be calculated piecewise at every prime $\ell\mid Dq$ and then packaging together the local outputs.


\begin{lemma}\label{weiltrace}
Let $\ell$ be a prime divisor of $Dq$. 
Let $(W, \langle \, ,\rangle)$ be the quadratic space over $\Q_{\ell}$ given by $K\otimes \Q_{\ell}$ equipped with the norm form divided
by $N(I_1 I_2)$,
and let $L \subset W$ be a maximal integral lattice. Let $(W',  L',\langle \, ,\rangle')$ be obtained from $(W, L, \langle \, ,  \rangle )$
by   multiplying the form $\langle \, ,\rangle$ by $-q N$.
 
Then -- for the Weil representation action of $\SL_2(\Q_{\ell})$ on Schwartz functions on $W \oplus W'$ -- 
the characteristic function $1_{L \oplus L'}$ of $L \oplus L'$ is invariant by $\Gamma_0(\ell) \subset \SL_2(\Z_{\ell})$, and 
\begin{equation} 
\label{def_trace}
 \mathsf{Tr}^{\SL_2(\Z_{\ell})}_{\Gamma_0(\ell)} 1_{L \oplus L'} = 1_{\cM_+} + 1_{\cM_-},
\end{equation}
 where $\cM_{\pm}$ are the two self-dual integral lattices containing $(L \oplus L')$. 
\end{lemma}

Before we prove Lemma \ref{weiltrace} we explain why it implies Proposition \ref{traceNDN}. 
There is no loss of generality in choosing $I_1, I_2$ relatively prime to $Dq$. 
It follows from \eqref{BKj} that $W \oplus W'$ is isometric to 
$B \otimes \Q_{\ell}$ with its reduced norm form,
and this identification carries $L \oplus L'$ to the closure of $I_1 \mathcal{O} I_2$. 
Thus, combining together \eqref{trace1} at all primes $\ell \mid Dq$, Proposition \ref{traceNDN} follows after noticing that if we take the self-dual lattice $ \cM_+$  to be the localization at $\ell$ of the global maximal order $\cM_d$ for some $d$ with $\ell \nmid d$, then $\cM_- = \cM_{d\ell} \otimes \Z_\ell$ .


\begin{proof} (of Lemma \ref{weiltrace})
 Write
$\mathbf{e}$ for the characteristic function of $L \oplus L'$
and $\mathbf{e}^*$ for the characteristic function of the dual lattice $(L \oplus L')^*$.
%
  Note that we have inclusions
$$ (L \oplus L') \subset \cM_+, \cM_- \subset (L \oplus L')^*,$$
with both inclusions of index $\ell$,
and indeed the quotient $\frac{ (L \oplus L')^*}{L\oplus L'}$ is isomorphic to $(\Z/\ell\Z)^2$
where the induced  $\Q_{\ell}/\Z_{\ell}$-valued quadratic form takes the form $(x_1, x_2) \mapsto \ell^{-1} (x_1^2-x_2^2)$;
in these coordinates $\cM_{\pm}$ correspond to $x_1 = \pm x_2$. 
Invariance of $\mathbf{e}$  by $\Gamma_0(\ell)$ follows readily from the definitions. 
Now a set of coset representatives  for $\Gamma_0(\ell)$ in $\SL_2(\Z_{\ell})$ is 
 $$\{ w \} \cup \{w m(t) w: t\in \Z/\ell\Z \}.$$
Note that $w\mathbf{e} = \ell^{-1}  \mathbf{e}^*$: 
\begin{itemize}
\item[-] for $y\not \in (L\oplus L')^*$, $w\mathbf{e}(y)$ is the integral on $L\oplus L'$ of the character $\mu(\langle y,x \rangle)$, which vanishes since that character is not trivial;
\item[-]  for $y \in (L\oplus L')^*$ we have $w\mathbf{e}(y) = \ \mathrm{vol}(L)\mathrm{vol}(L') = \ell^{-1}$
(the self-dual Haar measure on $W \oplus W'$ assigns mass $\ell^{-1}$ to $L \oplus L'$). 
\end{itemize}
It thus follows that
$$ 
\left(  \sum_{t \in \Z/\ell\Z} m(t)   \right) w\mathbf{e}  = 1_S,
$$
where  $S = \{ x \in (L \oplus L')^*: \langle x, x \rangle \in \Z_{\ell}\}.$ But $S$ is just the  
 union of $\cM_+$ and $\cM_-$, 
and also $\cM_+ \cap \cM_- = L\oplus L'$. Thus 
$1_S = 1_{\cM_+} + 1_{\cM_-} - \mathbf{e}$, 
  and we deduce that
\begin{equation}\label{trace1}
\mathsf{Tr} \ \mathbf{e} \ \  = \ \   w\mathbf{e} + w 1_{\cM_+} + w 1_{\cM_-}\!\!-w\mathbf{e} \ \ = \ \  1_{\cM_+} + 1_{\cM_-}.
\end{equation}
\end{proof}

\subsection{Proof of Proposition \ref{IJprop}} \label{IJpropproof}

  Let $E_{I}$ denote the supersingular elliptic curve associated to $\iota(I)$ and $e_I$ for the corresponding element  in $\Div(\cE)$. Set $w_{I} = w_{\iota(I)}$.


In order to prove \eqref{claim}, it suffices to show that both sides have the same Fourier coefficients for all $m\geq 1$. The $m$-th Fourier coefficient of  the r.h.s.\,of \eqref{claim} is 
\begin{equation}\label{2am}
2 a_m(\Theta(e_{I} \otimes e_{J})) = 2 \langle T_m e_{I},  e_{J}\rangle  = 2 w_{J} B_{I,J}(m).
\end{equation}
Here $B(m)$ is the $m$-th Brandt matrix  and $B_{I,J}(m)$ is the entry in $B(m)$ associated to  $E_I$ and $E_J$ (cf.\,\cite[\S 1, \S 2]{baby-gz}). The equalities in \eqref{2am} follow from the definition of $\Theta$ in \eqref{Theta-def} and \cite[4.4, 4.5, 4.6]{baby-gz}.
Since $2 w_{J} = |\Aut(E_{J})|$, it follows from \cite[Prop.\,2.3]{baby-gz} that $2 w_{J} B_{I,J}(m)$  is also equal to the number of isogenies of degree $m$ from the supersingular elliptic curve  $E_{I}$
to $E_{J}$; see also the proof of  \cite[Prop.\,2.7 (6)]{baby-gz} in p.\,128 of loc.\,cit.
The $\Z$-module of such isogenies is identified with $J^{-1} \cM I$,
where the degree is identified with $z \mapsto n(z)n(J)/n(I)$ (cf.\,\cite[2.1]{baby-gz} combined with the definition of $M_{ij}$ in p.\,118 of loc.\,cit.). 
Consequently
\begin{equation*}
2 a_m(\Theta( e_{I} \otimes e_{J})) =    |\{z \in J^{-1} \cM I:   n(z)n(J)/n(I) = m\}|.
\end{equation*}
This in turn is the $m$-th Fourier coefficient of the l.h.s.\,of \eqref{claim}, as  $J^{-1} \cM I$ is homothetic  to $J' \cM I$.

\subsection{Proof of Proposition \ref{Shimura}}  \label{Shimuraproof}
 
 Recall that this proposition is used only for the case of $D$ composite, and thus is not 
 strictly necessary e.g. for the statement of Theorem \ref{thm:main}.
 
Each quadratic character $\chi$ cuts out an extension $H_{\chi}/K$ which is the composition of $K$ and a quadratic extension $\Q_{\chi}/\Q$ of discriminant dividing $D$. Hence $\chi$ may be regarded as the restriction to $G_K$ of the Dirichlet character of conductor dividing $D$ attached to $\Q_{\chi}/\Q$, that we still denote with the same symbol. As it is readily seen by comparing the associated Galois representations, $\theta_{\psi_1^{-1} \chi}$ is the twist of $g= \theta_{\psi_1^{-1}}$ by $\chi$, and $\theta_{\psi_2^{-1} \chi}$ is the twist of $h=\theta_{\psi_2^{-1}}$ by $\chi$. 

Let $\pi_1, \pi_2$ be the automorphic representations for $\GL_2$ associated to $g,h$. 
Let $K_0(D) \subset \GL_2(\mathbb{A}_f)$ be the standard
compact open subgroup and let $K_{1}(D)$ be the kernel
of the natural ``diagonal'' maps $K_0(D)  \rightarrow \left((\Z/D\Z)^{\times}\right)^2$.
Note that $K_1(D)$ in the $\GL_2$ context is sometimes
defined to only impose {\em one} constraint, but here we understand
that both the diagonal entries are congruent to $1$ modulo $D$.

Set $$X_1(D) = \GL_2(\Q) \backslash \mathcal{H}^\ast \times \GL_2(\mathbb{A}_f)/K_{1}(D),$$
whose set of connected components identified with
$$\Q^\times \backslash  \mathbb{A}_f^\times/\det(K_{1}(D)) = (\Z/D\Z)^\times.$$
There are embeddings
$$ \pi_{1,f}^{K_1(D)} \mbox{ and } \pi_{2,f}^{K_1(D)} \hookrightarrow  H^0(X_1(D), \omega_{X_1(D)})$$
carrying the new vectors to $g$ and $h$ respectively.

The new vectors are characterized,  uniquely up to scalar,
by the fact that  they transform under
$$ k = \left(  \begin{array}{cc} a & b \\ c & d \end{array} \right) \subset K_0(D)$$
by the character $k \mapsto \chi_K(a)$.  For each $\chi$ as above,
we can consider the ``pseudo-new'' vector 
$$ g^{\chi} \mbox{ or } h^{\chi} \in \pi_{1,f}^{K_1(D)} \mbox{ or } \pi_{2,f}^{K_1(D)}$$ uniquely characterized
up to scalar by similarly transforming by the character $k \mapsto \chi_K(a) \chi(ad)$.  (The uniqueness
of such a vector follows by applying the usual new vector theory
to the representation $\pi_{1,f} \otimes \chi$, which has the same conductor as $\pi_{1,f}$. 
Explicitly, we may construct a pseudo-new vector from a new vector
by multiplying the associated function in the Kirillov model by
the character $\chi$; this statement is the representation-theoretic manifestation
of the fact that twisting by $\chi$ multiplies coefficients of the $q$-expansion by $\chi$.
A nice short reference for  basic properties of the Kirillov model
and new vectors is the paper \cite{Sc} and  a more encyclopaedic treatment is \cite{JL}, in particular Theorem 2.13). 

  With this construction, we have the following properties:

\begin{itemize}
\item[(a)]
the standard newforms $g_{\chi}$ and $h_{\chi}$ in the twisted automorphic representation correspond to the cup products:
$$g_{\chi} = g^{\chi} \cdot \chi, \qquad h_{\chi} = h^{\chi} \cdot \chi,$$
where we pull back $\chi$ to a complex-valued function on $X_1(D)$ by means
of the map $X_1(D) \rightarrow (\Z/D\Z)^{\times}$
to the group of connected components.

\item[(b)]  $\langle g^{\chi}, h^{\chi} \rangle = \langle g, h \rangle$, with reference
to any nontrivial $\GL_2(\mathbb{A}_f)$-invariant pairing $\pi_{1,f} \times \pi_{2,f} \rightarrow \mathbb{C}$. 

\end{itemize}
 
 Let $X_{10}(D,N)$ be obtained from $X_1(D)$ by   imposing a further $K_0(N)$-level structure. 
Let $\mathfrak{S}_{DN} \in H^1(X_{10}(D,N),\omega) \otimes \Z/p^t\Z$ denote the pull-back of the Shimura class.
 let $\pi_1,\pi_2: X_{10}(D,N) \lra X_1(D)$ denote the two forgetful maps intertwined by the Atkin-Lehner involution at $N$. It follows that
$$
\langle \mathsf{Tr}^{DN}_{N} \ G_{DN}(\chi), \mathfrak{S}\rangle = \langle G_{DN}(\chi), \mathfrak{S}_{DN}\rangle = 
\int  \pi_1^*(g_\chi) \cup \pi_2^*(h_\chi)  \cup \mathfrak{S}_{DN}  $$  $$ =
\int   \pi_1^*(g^{\chi})  \cup \pi_2^*(h^{\chi}) \cup \mathfrak{S}_{DN},  
$$
where $\int: H^1(X_1(D)_{\Z/p^t \Z}, \omega) \rightarrow \Z/p^t \Z$ is the trace map.
 It remains to verify that 
\begin{equation} \label{RTV}
\int  \pi_1^*(g^{\chi})  \cup \pi_2^*(h^{\chi}) \cup \mathfrak{S}_{DN}   = \int \pi_1^*(g) \cup \pi_2^*(h) \cup \mathfrak{S}_{DN}.\end{equation}
Now (b) above implies that $g^{\chi} \otimes h^{\chi}$ and $g \otimes h$
have the same image in the diagonal coinvariants
on $\pi_{1,f} \otimes \pi_{2,f}$. 
That is to say, considered inside $\pi_{1,f} \otimes \pi_{2,f}$,
\begin{equation} \label{ooo} g^{\chi} \otimes h^{\chi} - g \otimes h 
 = \sum_{i \in I} c_i  \left[ s_i v_1 \otimes s_i v_2  - (v_1 \otimes v_2) \right],\end{equation}
where $c_i \in \mathbb{C}$ and $s_i \in \prod_{v|D} \GL_2(\Q_v)$. 
Moreover, a straightforward argument with rational structures shows
that we may even take $c_i$ to belong to the field $L = \Q(\psi_1, \psi_2)$,
and similarly $v_1$ and $v_2$ to be $L$-rational modular forms,
and for sufficiently large $p$ we can suppose $c_i$, $v_1$, $v_2$, and $D(D-1)$ to be $p$-integral. 
Then
$$ \int (s_i v_1) \cup (s_i v_2) \cup \mathfrak{S} = \int v_1 \cup v_2 \cup \mathfrak{S},$$
where $\mathfrak{S}$ is a Shimura class at a sufficiently deep level $N \cdot D^r$;
this follows from the invariance of the Shimura class under the ad\`ele group away from $N$ after pullback to a further cover. 
 Therefore \eqref{ooo} implies the desired \eqref{RTV}.

\newcommand{\p}{r}
\newcommand{\q}{s}
\newcommand{\Nrd}{\mathrm{Nrd}}

\section{A trace identity for indefinite theta series} 
  \label{main-id-RM}
  The goal of this chapter is to prove the counterpart of 
  Theorem \ref{thm:waldspurger-garrett-CM-general}  in the case where 
       $(g,h)=(\theta_{\psi_1^{-1}}, \theta_{\psi_2^{-1}})$ is   a pair of new weight one $\theta$-series associated to  ray class characters $\psi_1$ and $ \psi_2$ of
     a common {\em real quadratic} field  $K$, whose  central characters, denoted $\chi_1$ and $\chi_2$ respectively, 
 satisfy $\chi_1=\chi_2^{-1}$. 
 Let $D$ denote the discriminant of $K$, and let $\delta= (\sqrt{D})$ be
 its different. We will assume that  the discriminant 
 $D$ is odd. 
  
   Since $g$ and $h$ are holomorphic,
 the characters $\psi_1$ and $\psi_2$, whose induced representations
 are odd two-dimensional Artin representations,
  are necessarily of mixed signature at 
 $\infty$. This means that the hypotheses  of
 Section \ref{thetaproducts}, in which
 $\psi_1$ and $\psi_2$ were assumed to be unramified, are 
  restrictive to the point of being vacuous: indeed, 
  the presence of the unit $-1$ precludes the existence of 
unramified id\`ele class characters of $K$ of mixed signature.
  It will   therefore only 
 be assumed  that the conductors of $\psi_1$ and $\psi_2$
 divide the different $\delta := (\sqrt{D})$ of $K$, which means that   the levels of
 $$ g = \theta_{\psi_1^{-1}} , \qquad  h = \theta_{\psi_2^{-1}}$$
 divide $D^2$. In particular,  these forms belong to the spaces
$M_1(\Gamma_1(D^2),\chi^{-1})$ and
$M_1(\Gamma_1(D^2),\chi)$
respectively.

 Because  the $\theta$-series for $\psi_2$ and its Galois conjugate $\psi_2'$ coincide, 
 it is harmless to suppose that $\psi_1$ and $\psi_2$ both have the same signature at $\infty$, namely the one for which  $\psi_1$ and $\psi_2$
are trivial relative to the standard 
 real embedding of $K$.

 Because the restrictions of $\psi_1$ and $\psi_2$   (viewed as characters of the id\`eles 
 ${\mathbb A}_K^\times$ of $K$)  to the group 
 ${\mathbb A}_\Q^\times$ of id\`eles of $\Q$ 
 are inverses of each other, it follows that, for all primes $v$ of $K$ dividing $D$ where
  $\psi_1$ and $\psi_2$ 
are possibly ramified,
 $$ \psi_{1,v}|_{\cO_v^\times} = 
 \psi_{2,v}^{-1}|_{\cO_v^\times}.$$
 But the Galois conjugation map $x\mapsto x'$ induces the identity on the residue fields of $K_v$ for such $v$,  and hence  the characters    
\begin{equation} \label{psi12def} \psi_{_{12}} := \psi_1 \psi_2, \qquad \psi_{_{12'}} :=\psi_{1}\psi_2'\end{equation}
appearing in Theorem \ref{thm:waldspurger-garrett-CM-general} are trivial on $\cO_{v}^\times$ 
for all primes $v$, including those dividing 
$D$.
It follows that $\psi_{_{12}}$ and 
$\psi_{_{12'}}$ are everywhere    unramified.
The character  $\psi_{_{12}}$ is  furthermore
 totally even, and  $\psi_{_{12'}}$
 is  totally odd.   
 
 The existence of the  odd unramified character
 $\psi_{_{12'}}$
 implies that the narrow class number of $K$ is twice its class number, and hence,  that all the units of $K$ have positive norm. 
 The fundamental unit $\varepsilon$ is chosen so that $\varepsilon >1$ relative to the fixed standard real embedding 
  $K \hookrightarrow \R$
  of $K$ evoked in the introduction
 
   In fact, it will be shown below that any pair  of unramified  characters of $K$ with  trivial
restrictions to $ {\mathbb A}_\Q^\times$
and opposite pure signatures  can be obtained from a pair $(\psi_1,\psi_2)$ as above,  in an  essentially unique way;
this fact plays a crucial role in the proof of Theorem
\ref{thm:RM-general} 
below, because it eliminates the need for an analogue of Proposition \ref{Shimura}
and thus leads to a more precise result.

  \subsection{Setup on Heegner cycles}\label{Idef}

As above, $K$ is a real quadratic field of odd discriminant $D$, all of whose units have norm $1$. 
Let $\mathfrak{o}$ be the maximal order of $K$.  Let $N\nmid D$ be an odd prime that splits in $K$.

 Choose $\delta_N\in \Z$ 
  satisfying 
  $$\delta_N^2 \equiv D \pmod{N}.$$
  This choice determines an ideal $\mathfrak{N} = (N, \delta_N-\sqrt{D})$ of $\fO$ of norm $N$.
  Also, let
$$ M_0(N) := \left\{ \left( \begin{array}{cc} a & b \\ Nc & d \end{array}\right)
\ \ \mbox{ with } a,b,c,d\in \Z \right\} \subset M_2(\Z)$$
be the standard Eichler order of level $N$ in the matrix ring 
$M_2(\Z)$. 
 This Eichler order is equipped with the standard orientation 
$$\mathbf{o}: M_0(N)  \ra \F_{N} = (\Z/N\Z) $$ 
onto the field of $N$ elements, sending a matrix to the mod $N$ residue class of its upper left hand entry.

Let $I \subset \mathfrak{o}$ be an ideal.
Writing $I \cap \mathbb{Z} = (a)$ with $a > 0$, we can write 
$$I = \left(a, \frac{-b+\sqrt{D}}{2}\right),$$
with $b$ uniquely determined modulo $a$. 
The action of $\mathfrak{o}$ on $I$ with respect to the basis
$(a,(-b+\sqrt{D})/2)$ 
gives a homomorphism  
\begin{equation} \label{alphadef} \alpha: \mathfrak{o} \rightarrow M_2(\mathbb{Z}), \ \  \ \sqrt{D} \mapsto \left[  \begin{array}{cc}  b & -2c\\ 2a & -b \end{array} \right], \end{equation} 
where $c$ is defined by 
stipulating that  the  binary quadratic form
 $ax^2+bxy+cy^2$  has discriminant 
 $D$.

 An eigenvector $v\in K^2$ for the action of $\alpha(K)$ is given by
 $$ v = \left(\begin{array}{c} (b+\sqrt{D})/2 \\
a\end{array}\right).$$
Write $\tau := \frac{b+\sqrt{D}}{2a}$, and  let
$v'$ and $ \tau'$ denote the algebraic conjugates of $v$ and $\tau$ 
respectively over $K$. 
 
 Suppose that $I$ is divisible by $\mathfrak{N}$ but not by $\mathfrak{N}'$. 
 Then  $a$ is divisible by $N$,
 $b$ is congruent to $\delta_N$ modulo $N$,
 and  $\alpha$   is an   embedding of $\mathfrak{o}$
 into $ M_0(N)$.
  Indeed, 
 the basis vector  $a\in I$ belongs to $\mathfrak{N}' I$ 
 since it 
  is  divisible by $N$, and   its image in
 $I/NI$ generates the index N subgroup
  $\mathfrak{N}'I/NI$, which is preserved under multiplication by $\mathfrak{o}$.
Hence multiplication by any element of $\mathfrak{o}$
  is represented by a matrix in $M_0(N)$ relative to the basis 
 $(a, (b+\sqrt{D})/2)$.
 Moreover the composition
$\mathbf{o} \circ \alpha: \mathfrak{o} \rightarrow \mathbb{F}_N$ of  $\alpha$ with the
orientation $\mathbf{o}: M_0(N) \rightarrow \mathbb{F}_N$
is  reduction modulo $\mathfrak{N}$.  

Replacing the basis $(a, (b+\sqrt{D})/2)$ of $I$  by another positively oriented\footnote{Here, a basis $(e_1 e_2)$
 is said to be positively oriented if it is in the $\SL_2(\Z)$-orbit
 of the specified one, or, said more intrinsically, $e_1 \wedge e_2$
 equals the norm of $I$ multiplied by $1 \wedge \sqrt{D}/2$. }
  basis of the same form conjugates the resulting embedding by an 
 element of $\Gamma_0(N)$, 
 hence the   embedding $\alpha$ attached to $I$  is independent of this choice of basis, up to conjugation in $\Gamma_0(N)$.

 The standard real embedding $K\hookrightarrow \R$ 
 that was fixed previously
yields a geodesic
 $ (\tau, \tau') \subset \calH $
 in the upper half-plane. 
Recall the  fundamental unit 
 $\varepsilon\in \mathfrak{o}_1^\times$  of $K$ of norm one, and let
\begin{equation} 
\label{eqn:Heegner-cycles}
\gamma_{I}= 
\alpha(\varepsilon)^\Z  \backslash(\tau, \tau')
  \end{equation}
 denote the closed geodesics on $\Gamma_0(N)\backslash \calH$ attached to
 $I$.  We regard it as oriented {\em from} $\tau$ {\em to} $\tau'$. This depends only on the 
 class of $I$
 in
\begin{equation} \label{CCsign} \mathcal{C} :=  \mbox{ the narrow ideal class group of $K$}, \end{equation}
 and correspondingly we will freely write $\gamma_I$
 for $I \in \mathcal{C}$. 
 
 Note that $\tau' < \tau$ and moreover the derivative 
 of the fractional linear transformation of $\mathbb{R}$
 induced by $\alpha(\varepsilon)$ at $\tau'$ (resp. $\tau$)
 is given by $(\varepsilon')^{-2}$ (resp. $\varepsilon^{-2}$). 
 Since $\varepsilon > 1 > \varepsilon'$,  we conclude that 
 the action of $\alpha(\varepsilon)$ on $(\tau, \tau')$ moves along the direction opposite to the orientation of the geodesic. 
%
%
%

   \subsection{Statement of the trace identity}
\label{sec:trace-indef-statement}

Given two narrow ideal classes, 
 choose representatives $I_1$ and $I_2$  that are 
  divisible by $\mathfrak{N}$ but not by $\mathfrak{N}'$.
Let $\alpha_i$ for $i \in \{1,2\}$
denote the two   embeddings   attached to $I_1$ and $I_2$ as in \S \ref{Idef}, and let  $v_i, v_i' \in K^2$ and $\tau_i, \tau_i' \in K$ be the associated
eigenvectors and fixed points, respectively.
  
 Write $\langle  \gamma_{_{I_1}} \cdot T_m \gamma_{_{I_2}}  \rangle_{_N}$ for the topological intersection pairing of the homology cycles
 $\gamma_{_{I_1}}$ and $ T_m \gamma_{_{I_2}}$ on 
 the Riemann surface $X_0(N)(\C)$.
 The generating series
\begin{equation}
\label{eqn:Jgg}
 \Theta(\gamma_{_{I_1}} \otimes \gamma_{_{I_2}}) := \sum_{m=1}^\infty 
 \langle \gamma_{_{I_1}} \cdot  T_m \gamma_{_{I_2}} \rangle_{_N} q^m 
 \end{equation}
 is a cusp form of weight two and level $N$.
This definition can be extended  by linearity to arbitrary linear combinations of RM geodesics, notably  the  
paths 
\begin{equation} \label{gammapsidef} \gamma_{\psi_{_{12}}}(q) = \sum_{I\in \cC} \psi_{_{12}}(I) \gamma_{_{I}}, \qquad  \gamma_{\psi_{_{12'}}}(q) = \sum_{I\in \cC} \psi_{_{12'}}(I) \gamma_{_{I}}\end{equation}
associated to the unramified characters $\psi_{_{12}}$ and $\psi_{_{12'}}$ respectively.

The  following theorem,
 which is the main result of this chapter, 
 relates the trace of products of binary theta series
 to modular generating series of real quadratic geodesic cycles as in
 \eqref{eqn:Jgg}.
  \begin{theorem}
\label{thm:RM-general} 
For all  theta series  $g = \theta_{\psi_1^{-1}}$ and $h = \theta_{\psi_2^{-1}}$ of
$K$ as above,
$$ 
 \mathsf{Tr}^{N D^2}_{N} (\theta_{\psi_1^{-1}}(N z) \theta_{\psi_2^{-1}}(z))  = 
    C  \cdot \psi_1(\cN') \cdot   \Theta(\gamma_{\psi_{_{12}}} \otimes 
    \gamma_{\psi_{_{12'}}}),$$
 where 
 \begin{equation} \label{Cdef}  C =  D    \sum_{D = D_1 D_2} \mu(D_1) \cdot D_2
  \cdot \psi_{_{12}} \psi_{_{12'}}(j_{_{D_1}})    \ =   \  D \prod_{p|D} (p  - \psi_{_{12}}\psi_{_{12'}}(j_p)),\end{equation} 
  and $j_{_{D_1}}$ is  the  order two  
  element represented by the ideal 
  $(D_1, \sqrt{D})$ in the narrow class group of $K$.
 \end{theorem}

  The reader should compare this theorem 
  to Theorem \ref{thm:waldspurger-garrett-CM-general},
which is less precise. It turns out that allowing  the ray class characters
  $\psi_i$ to be ramified at primes dividing the discriminant simplifies  rather than complicates  the situation.
Transposing the proof of Theorem \ref{thm:RM-general} 
 to the setting of Chapter
 \ref{thetaproducts}
  would presumably lead to a refined and slightly more general variant  of Theorem  \ref{thm:waldspurger-garrett-CM-general}.

\begin{remark}\label{symmetrybreaking}   In the  extension above of the generating series  \eqref{eqn:Jgg} to linear combinations
of geodesics we are always taking representatives of $I \in \cC$    
that are   divisible by $\mathfrak{N}$ but not by $\mathfrak{N}'$.  This choice introduces an asymmetry that reappears
throughout this section, and explains the appearance of the factor $ \psi_1(\cN')$ on the right-hand side of the identity in the Theorem.
Since the right-hand side is invariant under exchange of $\cN$ and $\cN'$, the second factor $\Theta(\gamma_{\psi_{_{12}}} \otimes 
    \gamma_{\psi_{_{12'}}})$ must also depend on the choice of $\mathfrak{N}$.
    \end{remark}

 The proof of Theorem \ref{thm:RM-general} 
is  summarized in \S \ref{summary} below, and the 
 details of this sketch  are fleshed out in  the remainder of the chapter. 
  
  \subsection{Summary of the proof} \label{summary}

%
%

  Let 
 $$ \cC := {\cI}(\mathfrak{o})/ {\rm P}_+(\mathfrak{o}), 
 \qquad
 \cC_D := {\cI}_{\delta}(\mathfrak{o})/  {\rm P}_{\delta,+}(\mathfrak{o}),$$
 be the narrow class group and generalised class group of
 conductor $\delta$, defined by letting
 \begin{itemize}
 \item 
 ${\cI}(\mathfrak{o})$, resp.~$\cI_{\delta}(\mathfrak{o})$,
 be the semi-group of ideals  of $\mathfrak{o}$,  resp. the ideals 
 that are prime  to $\delta$;
 \item
 ${\rm P}_+(\mathfrak{o})$ be the semi-group of principal ideals with a totally positive generator;
 \item
 ${\rm P}_{\delta,+}(\mathfrak{o})$ be the semi-group of principal ideals with a totally positive generator that is congruent to $1$ modulo $\delta$.
 \end{itemize}

Given ideals  $I_1$ and $I_2$, let
\begin{equation}
 \label{hl} 
 \cA := \{ (x,y) \in  (\cN'I_1)_+  
  \times  (I_2)_- 
   \mbox{ satisfying }  x  \equiv y \pmod{\delta} \},
   \end{equation}
and the $+$ and $-$ subscripts mean, respectively, positive and negative norm. 
The group
\begin{equation}
\label{eqn:defU}
\cU := \{   \pm(\varepsilon^a,  \varepsilon^b) \mbox{ satisfying } a\equiv b \pmod 2 \} 
\end{equation}
operates naturally on $\cA$. Let
\begin{eqnarray} 
 \label{Theta'def}
   \Theta^\sharp(I_1, I_2) &:= &
   \sum_{(x,y) \in \cA/(\varepsilon^{2\Z} \times \varepsilon^{2\Z})}
   {\rm sign}(x) \cdot {\rm sign}(y) \cdot q^{\frac{xx'}{D{\rm N}(I_1)} - \frac{yy'}{D{\rm N}(I_2)}} \\
   \nonumber
   &=&
    4\sum_{(x,y) \in \cA/\cU}
   {\rm sign}(x) \cdot {\rm sign}(y) \cdot q^{\frac{xx'}{D{\rm N}(I_1)} - \frac{yy'}{D{\rm N}(I_2)}}.
    \end{eqnarray}

The function $\Theta^\sharp(I_1,I_2)(e^{2\pi i \tau})$
 is a finite sum of suitable pairs of indefinite binary 
   theta series attached to certain cosets in
   $I_1\oplus I_2$, and is
   a modular form of weight two.
 It is  readily verified that it depends only on the classes of $I_1$ and $I_2$ in $\cC_D$.

  The proof of  Theorem
  \ref{thm:RM-general} 
 follows from two key Propositions. The first will be proved in \S \ref{Step1}
 and the second in \S \ref{Step2}.
 
   \begin{proposition}
  \label{prop:prod-HV-bis} (See \S \ref{Step1}, also cf. \eqref{GG}).  
 There is an equality of modular forms on $\Gamma_0(N)$:   \begin{eqnarray*}
  {\rm Tr}^{ND^2}_{N} \left( \theta_{\psi_1^{-1}}(q^N) \cdot \theta_{\psi_2^{-1}}(q) \right) &=&  \psi_1(\cN') \cdot \frac{C}{4}  \cdot
 \sum_{\cC\times \cC }   \psi_{_{12}}(I_1)
\psi_{_{12'}}(I_2)   \cdot \Theta^\sharp(I_1 I_2,  I_1 I_2'),
\end{eqnarray*}
where $C$ is as in \eqref{Cdef}.
\end{proposition}

   \begin{proposition}
 \label{prop:main-vonk} 
 (See \S \ref{Step2}, also cf. Prop. \ref{IJprop}). 
 The generating series of \eqref{eqn:Jgg} is equal to 
 $$ \Theta({\gamma_{_{I_1}}\otimes\gamma_{_{I_2}}})(q) =  \frac{1}{4} \cdot
  \Theta^\sharp(I_1 I_2 , I_1 I_2')(q).$$
   \end{proposition}
  Taken together, these two propositions 
  imply that the trace 
   appearing in Proposition \ref{prop:prod-HV-bis} is
   equal to
   $$    \psi_1(\cN') \cdot C  \cdot
 \sum_{\cC\times \cC }   \psi_{_{12}}(I_1)
\psi_{_{12'}}(I_2) \cdot \Theta(\gamma_{_{I_1}}\otimes \gamma_{_{I_2}}),$$
and the sum appearing here, by definition, equals  
$ \Theta(\gamma_{\psi_{_{12}}}\otimes \gamma_{\psi_{_{21}}}) $. 
That is precisely the statement of  Theorem \ref{thm:RM-general}.
 

%

 \subsection{Setup on class groups}
 
 The running assumption that all units of $K$ have norm one  implies that equivalence of ideals in the narrow sense is strictly finer than equivalence in the wide 
  sense, i.e., that 
  the narrow class number of $K$ is twice its 
class number. 
It also implies, by genus theory, 
 that   the odd discriminant   $D$  is a product of two negative fundamental discriminants, and hence is not prime.
Let  
$a\ge 2$ be the number of prime divisors of $D$.

Although $K$ possesses no 
unramified id\`ele class characters of mixed signature, such characters {\em always appear} in conductor  dividing  the different
   $\delta$   of $K$, since the units of $\mathfrak{o}$ which are $1$ modulo $\delta$ are all totally positive.

Let
   $$ \iota := \mbox{ the  class of  $\varepsilon$  
modulo  $\delta$}.$$
It is  one of the
$2^{a}-2$ possible {\em non-trivial}  ($\ne\pm 1$)
square roots of $1$ in $\mathfrak{o}/\delta = \Z/D\Z$.
For if $\iota=\pm 1$, the fundamental unit 
$\pm \varepsilon$ gives rise to a solution $(x,y)\in \Z^2$
of the Pell's equation
\begin{eqnarray}
\label{eqn:pell2}
 x^2-Dy^2 = 1,  && \qquad x\equiv 1 \pmod{D}, 
\qquad x \mbox{ odd}, \quad y \mbox{ even}, 
\qquad \mbox{or } \\
\label{eqn:pell1}
x^2- D y^2 = 4, &&  \qquad x \equiv 2 \pmod{D},  \qquad x,y \mbox{  odd}.
\end{eqnarray}
In the second case, the factorisation of $Dy^2 = (x-2)(x+2)$ into relatively prime integers implies that $$  x+2 = \pm u^2  \ \  \mbox{ and } \ \     x-2 = \pm Dv^2,$$
for some $(u,v)\in \Z^2$, 
and hence $(u,v)$ is a   solution of 
the equation $u^2 -Dv^2 = \pm 4$ of height strictly smaller than that of $(x,y)$.
Likewise, a solution to \eqref{eqn:pell2}
leads to  a pair $(u,v)$ satisfying
$$  x+1 = \pm 2 u^2 \mbox{ and }  x-1 = \pm 2 D v^2, \quad$$
and hence to  a unit of $\mathfrak{o}$ of smaller height, 
contradicting in both cases  the assumption 
that $\varepsilon$ is a  fundamental unit.

 \medskip
  There is a natural exact sequence
 $$ 0 \lra \langle \iota\rangle \lra (\Z/D\Z)^\times \lra \cC_D \lra \cC \lra 0.$$
where the first inclusion sends $t \in (\Z/D\Z)^{\times}$  to the principal ideal generated by any totally positive integer
congruent to $t$ modulo $\delta$. 
Let 
 $$ Z:= \ker(\cC_D\lra \cC) \simeq (\Z/D\Z)^\times/\iota$$
 be the kernel of the natural projection.
 Next, let $W \subset (\Z/D\Z)^{\times}$ be the 
 index $2$ subgroup  which is the kernel of the quadratic Dirichlet character associated to $K$, and   
 \begin{equation} \label{Ndef} \mathrm{N} : \cC_D \rightarrow  W \subset (\Z/D\Z)^{\times}\end{equation}
 be the norm map sending the class of an ideal to the mod $D$ residue class of its norm. 
The   triviality of the Herbrand quotient of the finite group  $\cC_D$  as 
 a ${\rm Gal}(K/\Q)$-module implies that 
 $$\cC_D^- :=  \mbox{kernel of $\mathrm{N}$} = \left\{ g/g' \mbox{ with } g\in \cC_D \right\},$$
 where $g \mapsto g'$ is induced by the Galois automorphism of $K$ over $\Q$; thus
  $W$ is now identified with $\cC_D/\cC_D^{-}$. 
 
 The groups $Z$ and $W$ have the same cardinality
 $\varphi(D)/2$, but the natural homomorphism $Z\lra W$ obtained by composing the  inclusion $Z\hookrightarrow \cC_D$ with the  
 surjection $\cC_D\rightarrow W$ is not an isomorphism; its kernel is the two-torsion subgroup of $Z$, of cardinality $2^{a-1}$.

Global class field theory identifies  $\cC$ with the Galois group of the Hilbert class field $H$ over $K$,
and $\cC_D$ with the Galois group of $H_D$ over $K$, where $H_D$ is the ray class field of $K$ of conductor $\delta$, 
an extension of $H$ of degree $\varphi(D)/2$. The subgroup $\cC_D^-$ is identified with the Galois group of $H_D$ over the maximal  subfield of $H_D$ which is Galois and  abelian over $\Q$, namely, the cyclotomic field $\Q(\zeta_D)$. The group $W$ is identified with the  Galois group of $\Q(\zeta_D)$ over $K$, an index two subgroup of
$(\Z/D\Z)^\times$.
 The situation is summarised in the  field diagram below.
$$ 
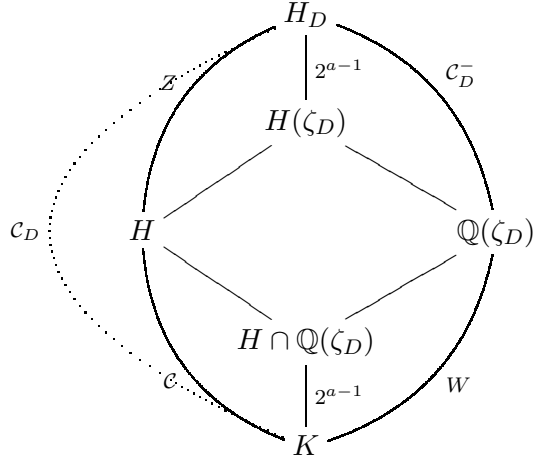
\xymatrix{ 
    & H_D \ar@{-}[d]^{2^{a-1}}  \ar@{-}@/_1.5pc/[ddl]_{Z}       
\ar@{-}@/^1.5pc/[ddr]^{\cC_D^-} \ar@{.}@/_8pc/[dddd]_{{\cC}_D}
    & \\
  &   H(\zeta_D) \ar@{-}[dl] \ar@{-}[dr] &  \\
  H \ar@{-}[dr] \ar@{-}@/_1.5pc/[ddr]_{\cC} &    &  \Q(\zeta_D)  \ar@{-}[dl]  \ar@{-}@/^1.5pc/[ddl]^{W} \\
    &    H\cap \Q(\zeta_D)  \ar@{-}[d]^{2^{a-1}}  & \\
  &  K  &}
    $$

We can now state and prove the crucial
 
\begin{lemma}
\label{lemma:absolutely-crucial}
There is an isomorphism 
\begin{equation}
\label{eqn:themapxi}
 \xi:  \cC^2 \lra ({\cC}_D \times_W {\cC}_D)/Z, \qquad
(I_1, I_2) \mapsto (I_1 I_2, I_1 I_2'),
\end{equation}
where the target is defined after
 choosing  lifts $I_1$ and $I_2$ of the eponymous ideal 
classes $I_1,I_2 \in \cC$  to the ray class group ${\cC}_D$.
\end{lemma}

The validity of this Lemma is the main reason that the current (RM) section obtains a more precise result
than the CM section.

 \proof 
Observe, first, that the map is
  well-defined, since the kernel of $\mathcal{C}_D \rightarrow \mathcal{C}$ 
 is the image of $(\Z/D\Z)^{\times}$, represented by principal ideals $(t)$ for $t \in \Z$, 
 and multiplying $I_1$ or $I_2$ by such a principal ideal of norm prime to $D$ only changes $(I_1 I_2, I_1 I_2')$ by an element of the diagonally embedded $Z$. 
The two groups have the same cardinality, by the discussion above; so  it is enough to prove that $\xi$  is surjective. But clearly
a pair $(J_1, J_2)$ lies in the image if and only if $J_2 J_1^{-1}$ has the form $I_2/I_2'$, i.e., belongs to $\mathcal{C}_D^{-}$. 
%
 \qed


\subsection{Setup on binary $\theta$ series}  \label{bthetadef} 
 For lack of a reference, let us briefly sketch the general situation, before
 specializing to the case of a quadratic space arising from the quadratic field $K$.

Consider a $2n$-dimensional anisotropic quadratic
space $(V, q)$ over $\Q$. 
The space of Schwartz functions
 on $V\otimes \mathbb{A}_f$ is 
 endowed with an action of 
  $\SL_2(\mathbb{A}_f)$   via the Weil representation which at any finite prime of $\Q$ is given by the following
  formulas: 
 
 \begin{equation} \label{popa0} \begin{gathered}
r_{\mu}\left(\begin{array}{cc}
1 & a \\
0 & 1
\end{array}\right) f(x)=\mu(a q(x)) f(x) \\
r_{\mu}\left(\begin{array}{cc}
a & 0 \\
0 & a^{-1}
\end{array}\right) f(x)=|a|\omega(a) f(a x) \\
r_{\mu}\left(\begin{array}{cc}
0 & 1 \\
-1 & 0
\end{array}\right) f(x)=\gamma \widehat{f}(x)
\end{gathered}\end{equation}

Here $\mu$ is a chosen additive character, $\omega$ is the quadratic discriminant character corresponding to the space $V$, 
the function
$\widehat{f}$ is Fourier transform of $f$ relative to $\mu$ and a self-dual Haar measure on $V$, and $\gamma$ is an eighth root of unity.  
We apply this only 
in the case when $V \otimes \Q_{\ell}$ is a {\em split} $4$-dimensional quadratic space;
in this case $\omega=1$ and also $\gamma=1$ (for the latter, see \cite[p. 176]{WeilActa}). 


Now suppose $\dim V = 2$, that  $(V,q)$ has signature $(1,1)$, and suppose that $\Psi_f$ is a Schwartz function on $V \otimes \mathbb{A}_f$
with stabilizer $\Gamma \leqslant \mathrm{SO}_q(\Q)$.

\begin{proposition}
\label{prop:theta-akshay}
Let
\begin{equation}
 \label{Binarythetapsi} 
 \theta_{\Psi_f}(z)  := \sum_{\substack{v \in \Gamma \backslash V, \\ q(v) >0} } \mathrm{sign}(v) e^{2 \pi i q(v) z}  \Psi_f(v),\end{equation}
where $\mathrm{sign}(v)$ is positive on one connected component of $q(v) > 0$ and negative on the other.
Then $\theta_{\Psi_f}(z)$
 is a modular form on $\SL_2$ and the association $\Psi_f \mapsto \theta_{\Psi_f}$
is equivariant for the action of 
$\SL_2(\mathbb{A}_f)$   via the Weil representation.
The same conclusion applies replacing the condition $q(v) > 0$ by $q(v) < 0$ and $e^{2 \pi i q(v) z}$ by $e^{-2\pi i q(v) z}$.  
\end{proposition}

\begin{proof}[Sketch of proof]
To check this we use the dual pair 
 $\mathrm{SO}_q \times \SL_2$. 
Fix an isomorphism $(V \otimes \R,q) \simeq (\R^2, xy)$,
let  $\Psi_{\infty}(x,y) = (x+y) e^{-\pi (x^2+y^2)}$, and let $\Psi = \Psi_{\infty} \otimes \Psi_f$
be the associated Schwartz function on $V \otimes \mathbb{A}$.
The function $\Psi_{\infty}$ is chosen so that
its average $\overline{\Psi}_{\infty}$
 over the connected component of
 $\mathrm{SO}_q(\R)$ is explicitly computable:  

{\small
 $$ \overline{\Psi}_{\infty} (x,y) =    \int_{\lambda \in \mathbb{R}_{+}^\times} (\lambda x + \lambda^{-1} y) e^{-\pi (\lambda^2 x^2 +\lambda^{-2} y^2)} \frac{d\lambda}{\lambda} \\
 = \begin{cases} \mathrm{sign}(x) e^{-2 \pi x y}, \ \ \mbox{ if }  xy > 0, \\ 0,  \ \ \ \ \ \mbox{ otherwise}. \end{cases}
 $$}

  In particular, fixing  $h \in \mathrm{SO}_q(\mathbb{A}),$ 
  the rule 
$$g \mapsto \tilde\theta_\Psi(g,h):= \sum_{x \in V} (g,h) \cdot \Psi (x),$$
where $(g,h) \cdot \Psi$ refers to the  actions of $g \in \SL_2(\mathbb{A})$  on $\Psi$ 
via the Weil representation, and of $h
\in \mathrm{SO}_q(\mathbb{A})$ via 
translation on the arguments,
defines an automorphic form on $\SL_2(\mathbb{A})$.
The rule $\Psi \mapsto \tilde\theta_\Psi$ is  equivariant for the $\SL_2(\mathbb{A})$-actions on both sides. 
 We  now integrate over $h \in \Gamma \backslash \mathrm{SO}_q( \R)$ to check that
 $$ (g_{\infty}, g_f) \in \SL_2(\mathbb{A}) \mapsto \theta  :=  \sum_{x \in \Gamma \backslash V}    \overline{\Psi}_{\infty}^{g_{\infty}}(x)  \cdot  \Psi_f^{g_f}(x)$$
is again a modular form for $\SL_2(\mathbb{A})$. This gives the claimed statement. 
 \end{proof}
 
 Now, we will  explicitly take $V$ to be $K$ together with a suitable rescaling of the norm as quadratic form,
 and explicate the above construction when $\Psi_f$
 is given by suitable characteristic functions. 
 
Given any fractional ideal $I$ of 
 $\mathfrak{o}$ of norm ${\rm N}(I) \in \Q^{>0}$ which is relatively prime to $\delta$, 
  the group $\varepsilon^{2\Z}$ preserves 
  the
  intersection $I^+$ (resp.~$I^-$)
   of  $I$ with the cone of  elements of positive (resp.~negative) norm
 in $K\otimes \R$, as well as the
  subsets
 \begin{eqnarray*}
  I_1^{+} & := &  \{ x\in I^+  \mbox{  with } x\equiv 1 \pmod{\delta}  \}, \\
    I_1^{-} & := &  \{ x\in I^-  \mbox{  with } x\equiv 1 \pmod{\delta}  \}.
  \end{eqnarray*}
  Taking $\Psi_f$ to be the characteristic function
  of $\{x \in I \otimes \hat{\Z}: x \equiv  1(\delta)\}$, 
  we recover Hecke's partial  theta series 
 \begin{eqnarray*}
  \vartheta^+(I)(q) &:=&  \sum_{x \in I_1^+/\varepsilon^{2\Z}} 
  {\rm sign}(x) \cdot  q^{xx'/D{\rm N}(I)},  
  \\
     \vartheta^-(I)(q) &:=&  \sum_{x \in I_1^-/\varepsilon^{2\Z}} 
  {\rm sign}(x) \cdot  q^{-xx'/D{\rm N}(I)}.   
   \end{eqnarray*}
 These theta-series depend only on the image of $I$ in the ray class group $\cC_D$, and 
  are modular forms of weight one on a suitable congruence subgroup. More precisely, by
\eqref{popa0}  or ~\cite[\S 1]{him}  we have:
 \begin{lemma}
 \label{lemma:HIM}
 For all $\left(\begin{array}{cc} a & b \\ c & d \end{array}\right)\in \Gamma_0(D),$
 \begin{eqnarray*}
 \vartheta^+(I)\left(\frac{a\tau+b}{c\tau+d}\right) &=&
 \left(\frac{D}{|d|}\right) 
 e^{\frac{-2\pi i ab}{D {\rm N}(I)}}   (c\tau+d) \cdot \vartheta^+(aI)(\tau), \\ 
 \nonumber
  \vartheta^-(I)\left(\frac{a\tau+b}{c\tau+d}\right) &=&
 \left(\frac{D}{|d|}\right) 
 e^{\frac{2\pi i ab}{D {\rm N}(I)}} (c\tau+d) \cdot
 \vartheta^-(aI)(\tau).
 \end{eqnarray*}
 \end{lemma}

%
%

\begin{lemma} \label{heckeandnewform}
We have
\begin{eqnarray}
\label{psi1eqn0}
 2 \theta_{\psi_1^{-1}}= \sum_{I\in \cC_D} \psi_1(I) \vartheta^+(I)(q^D),\\
 \label{psi2eqn0}
 2 \theta_{\psi_2^{-1}} =  \sum_{I\in \cC_D} \psi_2(I) \vartheta^-(I)(q^D).
\end{eqnarray}
 Here, by convention, 
  $\psi_1(J)$ simply means the value of $\psi_1$ applied to the image of $J$
 in the ray class group $\mathcal{C}_D$.
\end{lemma} 
\proof
Rewrite the right hand side of \eqref{psi1eqn0}   as  
\begin{equation}
\label{eqn:substack}
\theta_{\psi_1^{-1}}^+ :=   \sum_{\substack{I\in \cC_D, \\ x \in I_1^+/\varepsilon^{2\Z}}}    \psi_1(I)   \cdot  {\rm sign}(x) \cdot  q^{xx'/{\rm N}(I)},
\end{equation}
where we have made the slight abuse of notation of choosing a representative $I$ for each class in $\cC_D$, 
and  $I_1^+$ consists of elements in $I$ of positive norm and 
congruent to $1$ modulo $\delta$. 
The set $I_1^+$ is the union of its totally 
positive and totally negative elements. 
Sending a pair $(I,x)$ in the range of summation of the right-hand side of \eqref{eqn:substack}, where $x$ is totally positive (resp.~totally negative) to the integral ideal $I^{-1} x$
determines two
 bijections 
 \begin{equation} \label{oo} 
 (I, x \in I_1^{+} \mbox{ totally positive}) \mapsto  I^{-1} x, \end{equation} 
 \begin{equation} \label{oo2} (I, x \in I_1^{+} \mbox{ totally negative}) \mapsto  I^{-1} x, \end{equation}
to  the set of integral prime-to-$\delta$ ideals.
These two bijections are interchanged by precomposing with the 
involution
 $(I,x)  \mapsto ((z) I,  z \cdot x)$, 
where $z$ is any totally negative element congruent to $1$ modulo $\delta$. Therefore, for a given 
  integral prime-to-$\delta$ ideal
$J$, the preimages $(I, x)$ and $(I', x')$ under these two bijections
do not coincide; rather, the classes of $I$ and $I'$ in  $\mathcal{C}_D$
differ  
 by the image of $(-1) \in (\Z/D\Z)^{\times}$ in  $\mathcal{C}_D$. 
%
  Being of mixed signature, the character $\psi_1$ sends this element to $-1$, and 
  reindexing via $J = I^{-1} x$ allows us to 
  rewrite \eqref{eqn:substack}
   as  
 $ 
 2 \sum_{J} \psi_1^{-1}(J) q^{\mathrm{N}(J)},
 $
 which is (up to the factor of $2$) the standard expression for the $\theta$-series  $\theta_{\psi_1^{-1}}(q)$ 
 attached to $\psi_1^{-1}$.   This proves 
 \eqref{psi1eqn0}, and the  
 proof of 
 \eqref{psi2eqn0} is essentially the same.
 
 \qed

\subsection{Proof of Proposition   \ref{prop:prod-HV-bis}} \label{Step1} 
With preliminaries on $\theta$-series in hand, we proceed the proof of the first key step, Proposition  \ref{prop:prod-HV-bis}.
  
Recall that $N$ is a prime that splits in $K$ as a product 
   $\cN \cN'$ of two prime ideals of norm $N$. 
    If $I_1$ and $I_2$ are (representatives of) elements of $\cC_D$, thus, fractional ideals of $K$,  
the modular form
\begin{equation}
 \label{originalThetadef} 
 \Theta(I_1, I_2) = \vartheta^+(\cN' I_1)(q^N) \cdot \vartheta^-(I_2)(q) 
\end{equation}
 is of weight two on $\Gamma(D) \cap \Gamma_0(N)$.  
 Define
\begin{eqnarray*}
 \Theta^{(1)}(I_1, I_2) &=& \mbox{trace of $\Theta(I_1, I_2)$ to level $\Gamma_0(N) \cap \Gamma_1(D)$}, \\
 \Theta^{(0)}(I_1, I_2) &=& \mbox{trace of $\Theta(I_1, I_2)$ to level $\Gamma_0(N) \cap \Gamma_0(D)$}, \\
 \Theta^{(\emptyset)}(I_1, I_2) &=& \mbox{trace of $\Theta(I_1, I_2)$ to level $\Gamma_0(N).   $}
 \end{eqnarray*}
The superscripts here are intended to remind the reader of the level structure at $D$.  
 %

 \begin{lemma}
\label{lemma:Theta-zero}
For all ideals $I_1$ and $I_2$ of $\cC_D$,  
\begin{equation} \label{Theta1form}  \Theta^{(1)}(I_1,I_2)  = \left\{
\begin{array}{cl} 
 D \cdot \vartheta^+(\cN' I_1)(q^N) \cdot \vartheta^-(I_2)(q),   & \mbox{ if } {\rm N}(I_1) = {\rm N}(I_2), \\
0  & \mbox{ otherwise}. 
\end{array}
\right.\end{equation}
Here $\rm N$ is the norm of \eqref{Ndef}. 
Moreover $\Theta^{(0)}(I_1, I_2)$, which therefore vanishes unless
$(I_1, I_2)$ belongs to the fiber product
$$ \cC_D\times_W \cC_D := \{
(I_1, I_2) \in \cC_D \times \cC_D \mbox{ satisfying }
{\rm N}(I_1) = {\rm N}(I_2) \},$$
 depends only on the image of 
    $(I_1, I_2)$ in the quotient
    $ (\cC_D \times_W \cC_D)/Z$. 

\end{lemma}
\noindent
{\em Proof}.
The non-zero terms in the Fourier expansion of $ \vartheta^+(\cN' I_1)(q^N) \cdot \vartheta^-(I_2)(q)$ are concentrated at powers  of the form $q^{m/D}$, where 
$$m\equiv 1/{\rm N}(I_1) - 1/{\rm N}(I_2) \pmod{D},$$
and the result follows, since the trace from $\Gamma(D)$ to 
$\Gamma_1(D)$ annihilates any term of the form $q^{m/D}$ with
$D$ not dividing $m$, and multiplies the others by a factor of $D$.
The final assertion follows from the explicit formula
\begin{equation} \label{Theta0form}  \Theta^{(0)}(I_1, I_2) =  \sum_{a\in (\Z/D\Z)^\times}
     \Theta^{(1)}(a I_1 , a I_2) \end{equation}
     which is an immediate consequence of Lemma \ref{lemma:HIM}
      \qed

%
%

Note that if $(I_1,I_2)$ 
    belongs to $(\cC_D \times_W \cC_D)$, then the same is true of $(I_1, e I_2)$, where      $e$ is any element of $K^\times$ whose associated fractional ideal  is prime to $\delta$, and satisfies $e^2 =1 \pmod{\delta}$.
   \begin{proposition} 
   \label{prop:trace-theta} (cf. Prop. \ref{traceNDN0}).   
   For all $(I_1, I_2) \in (\cC_D \times_{W} \cC_D)/Z$
we have
\begin{equation} \label{prop-eqn}  \Theta^{(\emptyset)}(I_1, I_2) = D \cdot  \sum_{ D = D_1 D_2 } \mu(D_1) \cdot D_2  \cdot\Theta^\sharp(I_1, \varepsilon_{D_1} I_2),\end{equation}
   where  $\mu$ is the M\"obius function, and 
   \begin{enumerate}
     \item   $
   \Theta^\sharp(I_1, I_2) $ is the modular form 
     defined   in \eqref{Theta'def};  
    \item 
   the sum on the right is taken over all factorisations of $D$ into (relatively prime) fundamental discriminants $D_1, D_2$;
   \item  
  $\varepsilon_{D_1}$ is a totally positive element which  is congruent to $-1$ (resp $1$) modulo the primes dividing $D_1$ (resp.~$D_2$).
   \end{enumerate}
   \end{proposition}

\medskip\noindent
{\em Proof.}
 By Lemma \ref{lemma:Theta-zero}, 
it  may be assumed that $I_1$ and $I_2$ have the same norm, and  are represented by ideals that are relatively prime to $\delta$.
%
%
We must prove an equality of the form
$$ \mbox{Trace of $\Theta(I_1, I_2)$ from $\Gamma(D) \cap \Gamma_0(N)$ to $\Gamma_0(N)$} = \mbox{sum of $\Theta'$s}.$$

We will do this in a fashion very similar to the proof of Proposition \ref{traceNDN}, i.e.
by reducing it to a local question about Weil representations. 
Both $\Theta(I_1, I_2)$ and $\Theta^\sharp(I_1,I_2)$ have the   general form
\begin{equation} \label{tPsidef} \Theta_\Psi(q) :=  \sum_{(x,y) \in V_{\pm}} \Psi(x,y)  {\rm sign}(x) \cdot {\rm sign}(y) \cdot q^{Q(x,y)}, \end{equation}
 where:

 \begin{itemize}
 \item 
 $V = K \oplus K$ considered as a quadratic space over $\Q$:
 we consider it as a
 $\Q$-vector space and endow it with the  quadratic form 
 $$ Q(x,y) = \frac{xx'}{D{\rm N}(I_1)} - \frac{yy'}{D{\rm N}(I_2)}.$$
 \item $V_{\pm}$ are elements $(x,y)$  with $xx'> 0$ and $yy'<0$. 

\item $\Psi$ is   a Schwartz function on $V \otimes {\mathbb{A}}_f$
(with ${\mathbb{A}}_f$ the ring of finite adeles), invariant by the action of the subgroup 
$\cU$
of the 
of the unit group $\mathfrak{o}_1^\times \mathfrak{o}_1^\times$.
\end{itemize}
In the situation of \eqref{tPsidef} the map $\Psi \mapsto \Theta_{\Psi}$
 is equivariant for the   Weil representation action of $\SL_2({\mathbb{A}}_f)$
 on Schwartz functions on $V \otimes {\mathbb{A}}_f$;
 this action preserves the invariance condition on $\Psi$. 
Indeed this is a a product of two copies of the situation already discussed in \S \ref{bthetadef}
 and the Weil representation for a direct sum of quadratic spaces
 is simply the tensor product of the individual factors. 
 
%
%

 The action of $\SL_2({\mathbb A}_f)$ on Schwartz functions just mentioned
  factors as a (restricted) tensor product of actions of $\SL_2(\Q_p)$
 on the space of Schwartz functions on $V \otimes \Q_p$. The factor at $p$ is  the Weil representation
 of $\SL_2(\Q_p)$ on the Schwartz functions on   
 the quadratic space $(V_p, Q_p)$, where:
   $$ V_p = (K \oplus K) \otimes \Q_p,  \qquad  Q_p(x,y) =  \frac{xx'}{D{\rm N}(I_1)} - \frac{yy'}{D{\rm N}(I_2)}. $$
   In this way, we are reduced to a problem in explicitly computing with this Weil representation:
  the question of computing the trace
 of $\Theta^{\sharp}(I_1, I_2)$ from $\Gamma_0(N) \cap \Gamma(D)$
 to $\Gamma_0(N)$ reduces, thereby, to a product
 of local computations over $p$ dividing $D$, which 
   which 
   we will spell out below.

\begin{lemma}\label{weiltrace2} (cf. Lemma \ref{weiltrace}). 
Let $\ell$ divide $D$. 

Let $(W,  L, \langle \, ,\rangle)$ be the quadratic space over $\Q_{\ell}$ given by $K\otimes \Q_{\ell}$ equipped with the norm form,
multiplied by $(D N(I_1))^{-1}$ and $L$ be the ring of integers. 
Let $(W',  L',\langle \, ,\rangle')$ be similarly defined
but multiplying the form by $-(D N(I_2))^{-1}$ and taking $L'$ to be the ring of integers. 

Call $\mathbf{e}_1$ the characteristic function of
$$ \{(x \in L, x' \in L'): x \equiv  x'  \equiv 1 \in (\mathbb{Z}/\ell)\},$$
considered as a Schwartz function on $W \oplus W'$.  (Here the map from $L$
to $\Z/\ell$ is given by reduction at the maximal ideal.)

Then, for the Weil representation action of $\SL_2(\Q_{\ell})$ on Schwartz functions on $W \oplus W'$, 
the trace
  $$ \mathsf{Tr}^{\SL_2(\Z_{\ell})}_{\Gamma(\ell)}  \mathbf{e}_1 =  \ell \left( \ell 1_{\cM_+} -  1_{\cM_-} \right),$$
 where $\cM_{\pm}$ are the two self-dual integral lattices contained in $(L \oplus L')$,
 defined in \eqref{cmdef}. 
\end{lemma}

  Proposition 
\ref{prop:trace-theta}  follows readily from this Lemma.
     Indeed, from \eqref{originalThetadef}
 we can write 
  $\Theta(I_1, I_2)$ in the notation of \eqref{tPsidef} as the series $\Theta_{\Psi}$  with  $\Psi = \bigotimes \Psi_{\ell}$
  and $\Psi_{\ell}$ simply the characteristic function
  of $\mathfrak{N}' I_1 \oplus  I_2$
  for $\ell$ not dividing $D$, and
  $\Psi_{\ell} = \mathbf{e}_1$ for $\ell$ dividing $D$. 
We must only observe that, given a factorization $D=D_1 D_2$, 
the value of the corresponding $\Theta$ series
where we replace the role of $\mathbf{e}_1$ by $\cM_+$
for $\ell|D_1$ and by $\cM_-$
for $\ell|D_2$
is  exactly $\Theta^\sharp(I_1, I_2)$ but replacing $x \equiv y (\delta)$
by $x \equiv  \varepsilon_{D_1} y \ (\delta)$,
and this in turn coincides with $\Theta^\sharp(I_1, \varepsilon_{D_1} I_2)$
by means of the substitution 
$y \leftarrow \varepsilon_{D_1} y$.  

%
 \proof (of Lemma \ref{weiltrace2}). 
First we define $\cM_{\pm}$. 
Let $(L \oplus L')^*$ be the dual lattice
 with respect to the quadratic form $Q$ on $W \oplus W'$
 and similarly define $L^*, (L')^*$. 
Then $L^*$ corresponds simply to the maximal ideal inside
$L$, and similarly for $L'$, so there are canonical identifications
\begin{equation}
 \label{smp} 
 (L/L^*) \simeq (\Z/\ell) \simeq L'/(L')^*, \qquad \frac{(L \oplus L')}{(L \oplus L')^*} \simeq (\Z/\ell)^2.\end{equation}
We let
\begin{equation}  \label{cmdef} \cM_{\pm} = \mbox{ preimages of the lines $x_1 \equiv \pm x_2$ in $(\Z/\ell\Z)^2$}.
\end{equation}

The function $\mathbf{e}_1$ is readily verified, using  the formulas in \eqref{popa0},
to   be invariant by the principal congruence subgroup $\Gamma(\ell)$ of level
$\ell$ inside $\SL_2(\Z_{\ell})$.   Indeed, using the Iwahori factorization of $\Gamma(\ell)$ it suffices 
to prove this for upper triangular unipotent elements, diagonal elements, and lower triangular unipotent elements congruent to the identity modulo $\ell$.  For the first two this is obvious from the first two lines of \eqref{popa0}; to conclude we write the lower triangular unipotent subgroup with the conjugate of the upper triangular subgroup by the element 
$$w = \left(\begin{array}{cc}
0 & 1 \\
-1 & 0
\end{array}\right)$$
that appears on the last line of \eqref{popa0}.  Since the Weil constant $\gamma = 1$, it suffices to observe that $w^{-1}$ acts as the inverse of the Fourier transform.

We must compute its trace to $\SL_2(\Z_{\ell})$-invariants. 
Clearly, this projection is the same as if we first average over the diagonal subgroup, 
which has the effect of replacing $\mathbf{e}_1$ by
$\frac{\sum_{j \neq 0} \mathbf{e}_j}{\ell-1}$ with 
 where $\mathbf{e}_j$ be the Schwartz function defined  similarly to $\mathbf{e}_j$
but now considering $x_1 \equiv x_2 \equiv j \mbox{ modulo $\ell$}$.
 Now $1_{\cM_+} = \sum_j \mathbf{e}_j$ and  so
  $$\sum_{j \neq 0} \mathbf{e}_j = 1_{\cM_+} - \mathbf{e}_0.$$
  
Now this is in fact invariant by $ K_0(\ell)  \subset 
 \SL_2(\Z_{\ell})$. Indeed $1_{\cM_+}$ is already invariant by $\SL_2(\Z_{\ell})$, 
  since it is self-dual and integral for the quadratic form,
  and $\mathbf{e}_0$ is the characteristic function of $(L \oplus L')^{*}$,
  on which the quadratic form is integral. 
  We will prove that (cf. \eqref{def_trace})
\begin{equation} \label{indef_trace}  {\mathrm{trace}}^{\SL_2(\Z_\ell)}_{K_0(\ell)} \mathbf{e}_0 = 1_{\cM_+} + 1_{\cM_-}.\end{equation}
From this it follows that  the corresponding trace of $\sum_{j \neq 0} \mathbf{e}_j$ 
 equals $\ell 1_{\cM_+} - 1_{\cM_-}$ and 
the Lemma follows from this, taking into account
the index $[K_0(\ell):K(\ell)] = (\ell-1) \ell$.

The proof of of \eqref{indef_trace} is   very similar to the computation  carried out in Proposition  \ref{traceNDN}
 of the previous chapter, and, more specifically, to \eqref{trace1}.   The role of $\cM_{\pm}$ arises from 
 the fact that  
$$\{x \in L \oplus L': Q(x) \in \Z_{\ell}\} = \cM_+ \cup \cM_-, $$
and indeed the function induced by the quadratic form upon the
right-hand group of \eqref{smp} is proportional to $(x_1, x_2) \in (\Z/\ell \Z)^2 \mapsto \ell^{-1}(x_1^2-x_2^2) \in \ell^{-1} \Z/\Z$. 
  Let notation be as in \eqref{Weil}; as discussed there, 
   a system of coset representatives for $\SL_2(\Z_\ell)/K_0(\ell)$ is given by $w$ together with
 $w m(t) w$, 
where $1 \leq t \leq \ell$. 
We get $w \mathbf{e}_0 = \ell^{-1} 1_{L \oplus L'}$ 
and thus  $$   \sum_{t}   m(t)  w \mathbf{e}_0
= 1_{\cM_+} + 1_{\cM_-}-\mathbf{e}_0.$$
Therefore, 
$ \left( w \sum_{t}  m(t)    w \right)  \mathbf{e}_0
= 1_{\cM_+} + 1_{\cM_-} - w  \mathbf{e}_0$,
and so the trace of $\mathbf{e}_0$ is $1_{\cM_+} + 1_{\cM_-}$ as desired.
 \qed

\medskip
We will now parlay   Prop.~
    \ref{prop:trace-theta} into an expression for the trace of the product  
    $\theta_{\psi_1^{-1}}(q^N) \theta_{\psi_2^{-1}}(q)$
     of  weight one theta series. 
      The following result immediately implies the desired 
 Proposition \ref{prop:prod-HV-bis},
after performing a change of variables via the isomorphism
$ \cC^2 \lra ({\cC}_D \times_W {\cC}_D)/Z$
 of 
Lemma  \ref{lemma:absolutely-crucial}, given 
explicitly by
$(I_1, I_2) \mapsto (I_1 I_2, I_1 I_2')$.
      \begin{proposition}
  \label{prop:prod-HV}
  Let
   $$ G_{ND^2}(q) :=  \theta_{\psi_1^{-1}}(q^N) \cdot \theta_{\psi_2^{-1}}(q),  $$
which belongs to the space $M_2(\Gamma_0(ND^2))$ of modular forms of level $ND^2$ with trivial nebentypus character.  Then
\begin{eqnarray*}
 {\rm Tr}^{ND^2}_{ND} \left(G_{ND^2}\right) &=&  \psi_1(\cN') \frac{1}{4} \sum_{ (\cC_D\times_W \cC_D)/Z }  \psi_1(I_1) 
\psi_2(I_2) \cdot \Theta^{(0)}(I_1, I_2), \\ 
 {\rm Tr}^{ND^2}_{N} \left(G_{ND^2}\right) &=&  \psi_1(\cN') \cdot \frac{C }{4} \cdot
 \!\!\!\!\!\!\!
 \sum_{ (\cC_D\times_W \cC_D)/Z }   \psi_1(I_1) 
\psi_2(I_2) \cdot \Theta^{\sharp}(I_1, I_2),
\end{eqnarray*}
where $$
C := 
D \sum_{D=D_1 D_2}  \mu(D_1) \cdot D_2 \cdot \psi_1(\varepsilon_{_{D_1}}) = D \prod_{p|D} (p - \psi_1(\varepsilon_p))
$$
is a constant that depends on $(\psi_1,\psi_2)$ and $D$ but not on $N$.
\end{proposition}

\noindent
{\em Proof}. 
By Lemma \ref{heckeandnewform} 
$$
G_{ND^2}(q) = \frac{1}{4}  
\sum_{(I_1, I_2)\in \cC_D^2} \psi_1(\cN' I_1) \vartheta^+(\cN' I_1)(q^{ND}) \cdot
\psi_2(I_2) 
 \vartheta^-(I_2)(q^D),
$$  
where we re-indexed the sum for $\theta_{\psi_1^{-1}}$ via $I \leftarrow \mathfrak{N}' I$. 

Because the restrictions to $Z$ of the 
 characters $\psi_1$ and $\psi_2$ are inverses of each other, the right hand side can be rewritten as
 $$\frac{1}{4} \sum_{\cC_D^2/Z}  \psi_1(\cN' I_1) \psi_2(I_2) \sum_{j\in Z} \vartheta^+(j\cN' I_1)(q^{ND})  
 \vartheta^-(jI_2)(q^D),$$
 where $Z\subset \cC_D^2$ is embedded diagonally.
 It follows  from \eqref{Theta0form} and \eqref{Theta1form} that 
 \begin{equation}
\label{eqn:theta-psi12-Theta-bis}
G_{ND^2}(q) = \frac{1}{4D}
\sum_{\cC_D^2/Z} \psi_1(\cN' I_1) \psi_2(I_2) 
\cdot \Theta^{(0)}(I_1,I_2)(q^D).
\end{equation}
Both the left and right hand sides in this identity 
are modular forms on $\Gamma_0(ND^2)$.
Let   $U_D$ be the Hecke operator  which on $q$-expansions is given by
$$  
U_D(\sum a_n q^n) = \sum_{n\in \Z} a_{_{nD}} q^n.$$
The trace from level $D^2$ to level $D$ amounts to an application of $D\cdot U_D$, and
by the same reasoning as in 
Lemma 
\ref{lemma:Theta-zero},
we have
$$ U_D (\Theta^{(0)}(I_1,I_2)(q^D)) =
\left\{ \begin{array}{cl}
 \Theta^{(0)}(I_1,I_2)(q) & \mbox { if } {\rm N}(I_1) \equiv {\rm N}(I_2), \\
  0 & \mbox{ otherwise}. 
  \end{array} \right.
  $$
 Applying the trace to level $ND$ to both sides of
\eqref{eqn:theta-psi12-Theta-bis} therefore gives
 \begin{eqnarray*}
 {\rm Tr}^{ND^2}_{ND}(G_{ND^2}) &=& 
   \psi_1(\cN')  \cdot  \frac{1}{4} \cdot \sum_{(\cC_D\times_W \cC_D)/Z} \psi_1(I_1) \psi_2(I_2) 
\Theta^{(0)}(I_1,I_2)(q),
\end{eqnarray*}
and the first equation in Proposition 
  \ref{prop:prod-HV} follows directly.
  The second follows from this and
\eqref{prop-eqn}, taking into account that $\psi_2$ and $\psi_1$ agree on $\varepsilon_p$. 
  \qed

 \medskip

\subsection{Proof of Proposition \ref{prop:main-vonk}}
 \label{Step2} 
Recall now the setup
of \S \ref{sec:trace-indef-statement}.
 We choose narrow ideal classes $I_1$
 and $I_2$, 
and, by choosing representatives by ideals that are divisible by $\mathfrak{N}$ but not $\mathfrak{N}'$,
obtain  a pair of real quadratic geodesics
$\gamma_1:= \gamma_{_{I_1}}$ and $\gamma_2:= \gamma_{_{I_2}}$  in $\Gamma_0(N)\backslash \calH$ with the same discriminant $D$. We also obtain    embeddings $\alpha_i$ for $i \in \{1,2\}$
  attached to $I_1$ and $I_2$;  similarly we get  eigenvectors $v_i, v_i' \in K^2$  and fixed points $\tau_i, \tau_i' \in K$
  for the action of
  $\alpha_i(K^\times)$. 

Proposition  \ref{prop:main-vonk} asserts that
the generating series of \eqref{eqn:Jgg} is equal to 
 $$ \Theta({\gamma_1\otimes\gamma_2})(q) =  \frac{1}{4} \cdot
  \Theta^\sharp(I_1 I_2 , I_1 I_2')(q).$$

%

The proof proceeds, much as in the proof of the  Gross-Zagier formula, by the most powerful 
technique known to number theory -- {\em compute and compare}. 

Examining the definition of $\Theta^\sharp(I_1,I_2)$ from
\eqref{Theta'def}, we see  
that
\begin{equation} 
\label{bmdef}
 \mbox{$m$th Fourier coefficient of $\Theta^\sharp$} = \  4\!\!\!\sum_{(x,y)\in \cA_m/\cU} {\rm sign}(xy),
 \end{equation}
where
 $ \cA_m $
consists of the pairs $(x,y) \in \mathfrak{N}' I_1 I_2 \times I_1 I_2'$ satisfying
 \begin{equation} \label{Amdef} 
 xx'>0,  \ \  yy'<0,\ 
\qquad 
  \frac{xx'-yy'}{a_1 a_2} = Dm,
     \end{equation}
where $a_1 = N(I_1)$,  $a_2 = N(I_2)$ and  $\cU$ is the 
subgroup of $\mathfrak{o}_1^\times {\times} \mathfrak{o}_1^\times$    introduced in
\eqref{eqn:defU}.



%
%

Now we turn to the left hand side, which is more involved and will take up the remainder of the subsection. 
We must compute the $m$-th Fourier coefficient
$$ a_m := \langle \gamma_1 \cdot T_m \gamma_2 \rangle_{_N}.$$
Letting $M_0(N)_m$ be the set of elements of $M_0(N)$ of determinant $m$, and letting 
$$ \Gamma_1 : = \alpha_1(\mathfrak{o}_1^\times), 
\quad \Gamma_2 : = \alpha_2(\mathfrak{o}_1^\times), $$
 this intersection number can be rewritten as
\begin{equation}
\label{eqn:sum-am}
 a_m = \sum_{  A\in \Gamma_1 \backslash M_0(N)_m/\Gamma_2}  \langle (\tau_1,\tau_1') \cdot  (A\tau_2,A\tau_2') \rangle.
 \end{equation}
(Note that in \eqref{eqn:sum-am} the intersection numbers are now being computed on the upper half-plane and not on the modular curve.)
%
The  calculation  proceeds by rewriting the  coefficient  $a_m$ of 
\eqref{eqn:sum-am} 
as a sum over certain ideals  of $K$,  by exploiting the map   $$ \eta: M_2(\Q) \hookrightarrow   K \oplus K, \qquad
\eta(A) := (\det(v_1, A v_2), \det(v_1, A v_2')).
$$
The map $\eta$ sets up a $K\otimes K$-module
 isomorphism from $M_2(\Q)$ to $K \oplus K$, the module
 structures being given by
\begin{equation}\label{kkmodules} (a\otimes b) M := \alpha_1(a') M \alpha_2(b) 
 \mbox{ and }(a \otimes b) (x,y) = (abx, ab'y)
 \end{equation}
  respectively.

 It is also an isomorphism of quadratic spaces, after equipping $K\oplus K$ with the quadratic form 
 $Q(x,y) = \frac{xx'-yy'}{D  a_1 a_2}$:
 \begin{lemma}
 \label{lemma:eta-1}
If $\eta(A)  = (x,y)$, then 
$$ \det(A) = \frac{xx' - yy'}{D  a_1 a_2}.$$
\end{lemma}
\proof
The source and the target of $\eta$ are both cyclic $(K \otimes K)$-modules, as in \eqref{kkmodules}, 
and both sides transform the same way, which reduces us
to verifying the assertion for a single generator; taking $A$ to be the identity 
and using 
$D a_1 a_2  = \det(v_1,v_1')\det(v_2,v_2')$
this follows from the identity
$$ \det(v_1, v_1') \det(v_2, v_2') - \det(v_1, v_2) \det(v_1', v_2') + \det(v_1, v_2') \det(v_1', v_2) = 0.$$
 which can be derived by considering the determinant of the $4 \times 4$ matrix whose
rows are two copies of $[v_1, v_1', v_2, v_2']$. \qed
%

\begin{proposition}
\label{lemma:eta-2}
The image of $M_0(N)$ under $\eta$ is equal to
$$ \eta(M_0(N)) =  \left\{ (x,y) \in  \cN'I_1 I_2 \times I_1 I_2' \ \ \mbox{ with }
x\equiv y \pmod{\delta} \right\},$$
and $\eta$ induces a bijection between $\Gamma_1\backslash M_0(N)_m/\Gamma_2$ and $\cA_m/\cU$.  
 \end{proposition}
{\bf Proof}. 
 Note that $K\otimes K$ is naturally identified with $K\oplus K$ 
via the map $\varrho$ sending $a\otimes b$ to
 $$\varrho(a\otimes b) = (ab, ab').$$
 For $1\le i,j \le 2$, let
 $E_{ij}$  
be the elementary matrix having a $1$ in the $ij$ entry and $0$'s elsewhere, and set $\beta_j = (-b_j+\sqrt{D})/2$.
By the definition of $\eta$, 
\begin{eqnarray*} 
 \eta(E_{11}) = \varrho(-a_1\otimes \beta_2)  && \eta(E_{12}) = \varrho(-a_1 \otimes a_2),   \\
\eta(E_{21}) = \varrho(\beta_1 \otimes \beta_2),  && \eta(E_{22}) = \varrho(\beta_1\otimes a_2).
\end{eqnarray*}
It follows that   $\eta(M_2(\Z))$ is contained in the index $D$ subgroup of 
$ I_1 I_2 \times I_1 I_2'$ consisting  of pairs that are congruent modulo $\delta$. 
The fact that this containment is an equality follows by comparing the determinants of the pairing matrices for the two lattices, relative to the quadratic forms $\det(A)$ and
$ \frac{xx'-yy'}{D a_1 a_2}$ respectively.
Furthermore, the lattice $\eta(M_0(N))$ is obtained by replacing the $\Z$-module generator 
$\eta(E_{21}) $ by $N \eta(E_{21})$.  
A local analysis at $N$ shows that
$$ \eta(M_0(N)) \subset \cN' I_1 I_2 \times I_1 I_2'.$$
Since it is of index at most $N$ in $\eta(M_2(\Z))$, it must be equal to
$$   \{(x,y) \in \cN' I_1 I_2 \times I_1 I_2' \ \mbox{ with } 
x\equiv y \pmod{\delta} \},$$
as claimed. 
In particular, the map $\eta$ identifies $M_0(N)_m$ with 
$\cA_m$, and the last assertion follows from the fact
that  $\eta$ tranforms the left action of $\varepsilon\in \Gamma_1$ (resp.~the right action of $\varepsilon \in \Gamma_2$) into  mutliplication by  $(\varepsilon, \varepsilon)$
(resp.~by $(\varepsilon,\varepsilon^{-1}))$, which together generate $\cU$.  $\square$
\bigskip

It is also crucial to interpret the intersection pairing
$\langle \gamma_1 \cdot A \gamma_2\rangle\in \{-1,0,1\}$ in terms of $\eta(A)$.
\begin{lemma}
\label{lemma:eta-3}
 If $\det(A) = m>0$, then 
the intersection $\langle \gamma_1 \cdot A\gamma_2\rangle$ is non-zero if and only if $xx' >0$ and $yy'<0$,
where $\eta(A)=(x,y)$. 
In that case, it is given (after suitable choice of orientation conventions for the intersection)
by ${\rm sign}(xy)$. 
\end{lemma}
{\bf Proof}.
Given any four distinct  elements $t_1,t_1', t_2, t_2'$  of $ \PP_1(\R)$, the hyperbolic geodesics $(t_1,t_1')$ and $(t_2,t_2')$ intersect non 
trivially if and only if the cross-ratios $[t_1,t_2'; t_2,t_1']$ and
$[t_1,t_2; t_2',t_1']$ belong to the open interval $(0,1)\subset \R$.
This can be seen by  exploiting the invariance of
 the cross ratio  under M\"obius transformations to reduce this statement to the special case in which $(t_1,t_1', t_2,t_2')=(0,\infty,1,t)$, where it can be verified directly.
In particular, 
the geodesics $(\gamma_1, A \gamma_2)$ intersect precisely when the following cross ratios   
belong to  $(0,1)\subset \R$: 
  \begin{equation}
  \label{eqn:cross-ratio-appears-1}
  \frac{xx'}{mDa_1 a_2} =  \frac{\det(v_1,Av_2) \det(v_1',A v_2')}{\det(v_1,v_1')\det(Av_2,Av_2')} = [\tau_1, A\tau_2'; A\tau_2,\tau_1'],
   \end{equation}
 \begin{equation}
  \label{eqn:cross-ratio-appears-2}
   \frac{-yy'}{D a_1 a_2 } = \det(A) \frac{\det(\tau_1, A \tau_2') \det(\tau_1', A \tau_2)}{\det(\tau_1, \tau_1') \det(A \tau_2', A \tau_2)} =  \det(A) [\tau_1, A\tau_2; A\tau_2',\tau_1'].
  \end{equation}

  The first assertion follows.
  As to the second,   the sign of   
      $$xy = \det(v_1, Av_2) \det(v_1', Av_2)$$ 
 determines
  whether or not $\tau_1$ lands inside or outside of the geodesic from $A\tau_2$ to $A\tau_2'$,
   and hence determines the sign of the non-zero intersection, given a suitable choice of orientation
  on $\calH$. 
  $\square$

\bigskip
Recall that $\cU$
acts naturally on the set $\cA_m$ from \eqref{Amdef}. 
   By combining Lemmas \ref{lemma:eta-1}, 
   \ref{lemma:eta-2}, and \ref{lemma:eta-3}, we obtain:  
  \begin{proposition}
  \label{prop:formula-for-am}
For all $m\ge 1$, 
$$ a_m = \sum_{ (x,y) \in \cA_m/\cU}{\rm sign}(xy). $$
\end{proposition}
Comparing  this proposition  with \eqref{bmdef} 
shows that
$$\Theta(\gamma_{_{I_1}}, \gamma_{_{I_2}}) = \frac{1}{4} \Theta^\sharp(I_1I_2, I_1 I_2'),$$
and Proposition   \ref{prop:main-vonk}
 follows.

 \newcommand{\bM}{{\mathbb M}}
\newcommand{\fM}{{\mathfrak{M}}}
\newcommand{\fN}{{\mathfrak{N}}}
\newcommand{\bX}{{\mathbb X}}
\newcommand{\bS}{{\mathbb S}}
\newcommand{\bD}{{\mathbb D}}
\newcommand{\bH}{{\mathbb H}}
\newcommand{\fS}{{\mathfrak{S}}}

\section{Higher Eisenstein elements }
\label{sec:hee}

  This chapter is devoted to a review  of ``higher Eisenstein elements"   in the sense of Merel and Lecouturier  \cite{merel}, \cite{Lec1},
  i.e. elements in suitable spaces of modular forms that are not killed by the Eisenstein ideal but by its square, see
Definition \ref{first-hee}.
We will provide explicit formulas for Eisenstein and higher Eisenstein elements in
\begin{itemize}
\item  the space  $\mathbb{M}$ of modular forms  (Proposition
\ref{prop:mazur});
\item   the dual space $\mathbb{M}^*$ to modular forms  (Theorem    \ref{thm:shimura-hes}); 
\item  the positive part of cohomology $\mathbb{H}^+$ of the modular curve (Theorem \ref{thm:eisenstein-relative-H1});
\item  the negative part of cohomology $\mathbb{H}^-$ of the modular curve (\S \ref{sec:he-betti}; here we don't need higher elements), and finally
\item 
 the supersingular module  $\mathbb{D}$ (Theorem \ref{thm:sigma-g}). 
\end{itemize}
Each of these spaces $\mathbb{M}, \mathbb{H}^+, \mathbb{H}^-, \mathbb{D}$ is the completion
of a suitable Hecke module at Mazur's Eisenstein ideal in the Hecke algebra.

 \subsection{Higher Eisenstein series}
 \label{sec:hes}
 
As in \S \ref{notn}, let $N>3$ be a prime,   let $M_2(N)$ be the module of   weight two modular forms
    with Fourier coefficients in $Z = \Z[\frac{1}{6N}]$ for the Hecke congruence group $\Gamma_0(N)$,
and let $S_2(N)\subset M_2(N)$ denote the submodule of cusp forms. 
  Denote by 
$\TT(N)$     the  ring  generated by the   Hecke operators 
$T_n$ (with $N\nmid n$) together with $T_N:= U_N$,
acting 
faithfully on  $M_2(N)$.

 The vector space $M_2(N)\otimes\Q$ is generated by $S_2(N)\otimes \Q$ along with 
 the weight two Eisenstein series   whose $q$-expansion is given by
\begin{equation}
\label{eqn:def-EN}
  \EN(q) =    \frac{N-1}{24}  +  \sum_{n=1}^\infty \sigmaN(n) q^n, \qquad \mbox{ where }
  \sigmaN(n) = \sum_{\substack{d|n, \\   N\nmid d}} d.
  \end{equation}
The homomorphism 
$$\varphi_{\rm Eis}: \TT(N)\lra  \Z, \qquad \varphi_{\rm Eis}(T_n):= \sigmaN(n)  $$ 
by which $\TT(N)$ acts on $\EN$
is called the 
{\em Eisenstein homomorphism}, and its kernel $I_{\rm Eis}$  is called the {\em Eisenstein ideal}.

 For any maximal ideal $\mathfrak{m}$ of $\mathbb{T}(N)$ and
    any $\mathbb{T}(N)$-module $M$, let $ M_\mathfrak{m}$  denote
     the completion of $M $ at $\mathfrak{m}$.  
    The  maximal ideal $\mathfrak{m}$ is said to be {\em
  Gorenstein} if   
  $\TT:= \mathbb{T}(N)_{\mathfrak{m}}$
    is a Gorenstein ring.
It is known that all maximal ideals of $\TT(N)$ 
containing a prime $p>3$ are Gorenstein, by a result of Mazur \cite[cor. II.16.3]{Maz}. 

Let $p>3$ be a prime divisor of $N-1$. The  maximal ideal  $\fm := (p, I_{\rm Eis})$
 of $\TT(N)$ is called the $p$-Eisenstein ideal. 
Let 
 $$\TT := \TT(N)_\fm, \qquad \bM := M_2(N)_{\fm}$$ 
   denote the completions of 
$\TT(N)$   and $M_2(N)$  relative to  this ideal.  The ring $\TT$  is a complete local ring which is 
   free of finite rank as a  $\Z_p$-module. The module $\bM$ 
  is canonically dual to $\TT$ via the pairing $\bM\times \TT\lra \Z_p$ given by
    $\langle f, T\rangle = a_1(Tf)$, and hence  $\bM$
    is free of rank one as a $\TT$-module, since $\TT$ is Gorenstein.
 The $\Z_p$-rank of $\TT$ is strictly greater than one because
  $p$ divides $N-1$. We fix a discrete log $(\Z/N\Z)^{\times} \rightarrow \Z/p^t \Z$
  as in \S \ref{notn}.

The following proposition is due to Lecouturier
\cite{Lec1},  but the  details of the proof have been provided for the sake of being self-contained. 
\begin{proposition}
\label{prop:mazur}
There is  a modular form $E' \in M_2(N)\otimes(\Z/p^t\Z)$  
having Fourier expansion  of the form
$$ 
E' =  \cM  -  \sum_{n=1}^\infty \left(\sum_{d|n'}  \log(d^2/n') d\right) q^n  
$$
for some $\cM\in \Z/p^t\Z$,
where $n'$ denotes the prime-to-$N$ part of $n$.  It satisfies
 $(U_N-1) E' = 0$, and,  for all primes $\ell \ne N$,
$$ (T_\ell -(\ell+1)) E' =  (\ell-1)   \log(\ell) \EN.$$
\end{proposition}

The modular form $E'$ mod $p^t$  is called the {\em higher Eisenstein series} of  weight $2$ and level $N$.
We will discuss abstractly such elements in other Hecke modules in \S \ref{sec:gen-he}.

\begin{proof} 
Recall that $Z:=\Z[1/6N]$ and let $I$ denote the augmentation ideal in the group ring $Z[G_N]$, where
$G_N$ is as in \eqref{GNdef}. For $d$ an integer prime to $N$ we shall denote by $\sigma_d$
the corresponding element of $G_N$, arising from $d$ by means of the homomorphism
$\Z \rightarrow (\Z/N)^{\times} \rightarrow G_N$.  Let
  $\bE$ and $\bF$ be the formal $q$-expansions with coefficients in  
 $Z[G_N]$ given by
\begin{equation}
\label{eqn:group-ring-es}
 \bE :=   \fM  -  \sum_{n=1}^\infty \left(\sum_{\substack{d|n, \\ N \nmid d}} d \sigma_d\right) q^n, 
\qquad\qquad \bF :=  - \sum_{n=1}^\infty \left(\sum_{\substack{d|n, \\ N   \nmid \ n/d}} d \sigma_{n/d}\right) q^n,
\end{equation}
where  
$$  \fM := \frac{1}{2} \sum_{j=1} ^{N-1}   \theta_j \cdot \sigma_j, \qquad \mbox{ with }
\theta_j :=  \frac{N}{2}  B_2(j/N), \quad B_2(x):=x^2-x+1/6.  
$$ 
These formal $q$-expansions 
 satisfy, for every Dirichlet character $\chi$ of modulus $N$,
$$  \chi(\bE) = \left\{ 
\begin{array}{ll}
\EN = E_2(1,1_N) & \mbox{ if } \chi=1; \\
E_2(1,\chi) & \mbox{ otherwise}, 
\end{array}\right.    \qquad
 \chi(\bF) = \left\{ 
\begin{array}{ll}
E_2(1_N,1) & \mbox{ if } \chi=1; \\
E_2(\chi,1) & \mbox{ otherwise},
\end{array}\right.
$$
where  $1_N$ denotes the trivial character, but viewed as having modulus $N$,
and $E_2(1,\chi)$ and $E_2(\chi,1)$ are the usual Eisenstein series  associated 
to the Galois representations $\chi\omega \oplus 1$ and $ \omega \oplus \chi$ respectively with $\omega$ the cyclotomic character, 
 whose Fourier expansions 
are given  by
$$
E_2(1,\chi)(q) = -L(-1,\chi)/2 - \sum_{n=1}^\infty (\sum_{d|n} \chi(d) d ) q^n,
$$
$$
E_2(\chi,1)(q) =  -\sum_{n=1}^\infty (\sum_{d|n} \chi(n/d) d ) q^n.
$$
These Eisenstein series are    classical modular forms of weight two on the congruence group
 $\Gamma_1(N)$, 
with the exception of $E_2(1_N,1)$. The latter is (the holomorphic part of) a  {\em nearly holomorphic} form  in the sense of Shimura, as we see via
 $$E_2(1_N,1) = E_2(q) - E_2(Nq) + \frac{\mathrm{const}}{y}, \quad \mbox{ 
with } E_2 =(8 \pi y)^{-1}-\frac{1}{24}+ \sum_{n} (\sum_{d|n} d) q^n. $$
Denote by
     $M_2^{\rm nh}(\Gamma_1(N);Z)$ the abelian group of $q$-expansions of such nearly holomorphic forms, so that $E_2(1_N, 1) \in M_2^{\rm nh}$. 
   
It follows that $\bE$ is a classical modular form with coefficients in $Z[G_N]$. As for $\bF:= \sum_{\sigma\in G_N} F_\sigma \cdot \sigma$,    although the individual coefficients $F_\sigma \in M_2^{\rm nh}(\Gamma_1(N);Z)$ are merely nearly holomorphic, their pairwise differences
$F_{\sigma_1}-F_{\sigma_2}$ are in fact holomorphic, since they lie in the linear span of the 
$E_2(\chi,1)$ with $\chi$ non-trivial.
It follows that one can write
$$ \bF = \bF_0 +  \eta \cdot {\bf N}, $$
where 
$$ \bF_0 \in M_2(\Gamma_1(N);Z[G_N]), \qquad
\eta \in M_2^{\rm nh}(\Gamma_1(N);Z), \qquad {\bf N} = \sum_{\sigma\in G_N} \sigma.$$
Since the  
 $q$-series $E_2(1, 1_N)$ and $E_2(1_N,1)$ agree modulo $p^t$,  and the image of the norm element ${\bf N}$  in $\Z/p^t\Z[G_N]$ belongs to $I^2$, 
the mod $p^t$ reduction of the  difference $\bE-\bF$ belongs to $M_2^{\rm nh}(\Gamma_1(N);\Z/p^t\Z) \otimes I$.
 It follows that its natural image, denoted $\overline{\bE-\bF}$, in $M^{\rm nh}_2(\Gamma_1(N);\Z/p^t\Z)\otimes (I/I^2)$ 
 gives rise to an element
$$\overline{\bE-\bF} \in M_2(\Gamma_1(N);\Z/p^t\Z) \otimes (I/I^2) = 
M_2(\Gamma_1(N);\Z/p^t\Z)\otimes G_N, $$
which is   invariant under the diamond operators. 
At the last stage we have used the isomorphism
$(I/I^2)  \simeq G_N \otimes Z$
uniquely characterized by the fact that $\sum a_j \sigma_j \mapsto \prod j^{a_j} \otimes 1$
when $a_j \in \mathbb{Z}$.  
Consequently, 
%
%
%
$\overline{\bE-\bF}$ 
arises from a unique element of $M_2(\Gamma_0(N);\Z/p^t\Z) \otimes G_N$, to be denoted by the same letter. One then readily checks that the 
modular form $E'$ given 
by 
$$ E'  := \log(\overline{\bE-\bF}) $$ has
all the properties claimed in the proposition. For instance, 
since $\bE$ and $\bF$ are eigenvectors for $T_\ell$ with
eigenvalue $(1+\ell\sigma_\ell)$ and $(\sigma_\ell+\ell)$ respectively,
$$ (T_\ell-(\ell+1))(\bE-\bF) = (\ell \sigma_\ell  -\ell) \bE - (\sigma_\ell-1) \bF = (\ell-1)(\sigma_\ell-1) \EN \pmod{I^2[[q]]},$$
and therefore, after reducing modulo $I^2$ and taking the discrete logarithms on both sides,
$$ (T_\ell-(\ell+1))E' =  (\ell-1)  \log(\ell)    \EN,$$
as claimed.
\end{proof}

\begin{remark}
The proof of Proposition \ref{prop:mazur} yields an explicit formula for the constant term  
$\mathcal{M}$ of  
 $E'$. It is attached to the 
  Mazur-Tate,   or Stickelberger element  $\fM$, which  is characterised as  
  the unique element of $Z[G_N]$
satisfying 
$$ \chi(\fM) = \left\{ 
\begin{array}{ll}
(1-N)/24  & \mbox{ if } \chi=1; \\
-L(-1,\chi)/2 & \mbox{ otherwise}, 
\end{array}\right. \qquad \mbox{ for all   } \chi: G_N \lra \C^\times.
$$
This Mazur-Tate element  belongs to the augmentation ideal $I$ of the   group ring   $(\Z/p^t\Z)[G_N]$, and its 
natural
 image in $I/I^2 = G_N \otimes (\Z/p^t\Z)$,  denoted $\fM'$, 
  is called the {\em ``Mazur-Tate derivative"} of $\fM$.
The constant term $\cM$ is the discrete logarithm
of this Mazur-Tate derivative:
\begin{equation}
\label{eqn:merel-constant}
 \cM  =  \log(\fM') 
 \end{equation}
This explicit formula for $\cM$, which was first obtained (under a slightly different guise)
by Loic Merel \cite{merel}, will play no role in   the
 argument.
 \end{remark}

\subsection{General Higher Eisenstein elements}
\label{sec:gen-he}

From now on, the symbol $I_{\rm Eis}$ shall also be used to denote  the Eisenstein ideal
in the completed Hecke algebra $\TT$, whose associated 
 quotient $\TT/I_{\rm Eis}$ is isomorphic to $\Z_p$.
 
 Mazur has proved that  $\TT$ is generated by a single element as a $\Z_p$-algebra, i.e. 
 $\TT =\Z_p[x]$ for suitable $x \in \TT$. Indeed, one may take $x=T_{\ell}-\ell-1$
for suitable $\ell$ and  $x$ may be taken to generate $I_{\rm Eis}$. See \cite[\S II, Prop. 18.10]{Maz},
as well as the discussion at the start of \S 19 therein. 
 The following result is also proved by Mazur ({\em loc. cit.} Proposition 18.8); 
 we sketch a direct proof. 
\begin{corollary} 
\label{cor:mazur-eisenstein}
There   is an isomorphism
$$  \eta: I_{\rm Eis}/I_{\rm Eis}^2 = \Z_p\otimes (\Z/N\Z)^\times    \simeq (\Z/p^t\Z) $$
sending the element $(T_\ell-(\ell+1))$ to $(\ell-1) \otimes  \ell $, for all primes $\ell\ne N$, 
and sending $U_N$ to $1$.
\end{corollary}
\begin{proof}[Sketch of proof]
The modular form $\EN + \varepsilon E'$ with coefficients in the ring $\Z/p^t\Z[\varepsilon]$ of dual numbers is a
Hecke eigenform on $\Gamma_0(N)$,  
and gives rise to a surjective homomorphism with kernel $I_{\rm Eis}^2$
\begin{equation}
\label{eqn:phi-eis}
\tilde\varphi_{\rm Eis}: \TT \lra  \Z/p^t\Z[\varepsilon], \qquad \tilde\varphi(U_N) = 1, \qquad
\tilde\varphi(T_\ell) = (\ell+1)+ (\ell-1) \log(\ell) \varepsilon.
\end{equation}
The quantity $\tilde\varphi(T_\ell - (\ell+1))$  is equal to  $\log\circ\eta(T_\ell-(\ell+1))$, and the corollary follows. 
\end{proof}

Let   $\bX$ be  a  free  $\TT$-module  of rank one.

\begin{lemma}
The module $\bX[I_{\rm Eis}]$ of   elements $m\in \bX$ 
   satisfying 
   $$ (T_\ell-(\ell+1))m  = 0  \quad \mbox{ for all  primes } \ell\ne N, \qquad U_N m = m,$$
   is free of rank one over $\Z_p$.
   \end{lemma}
   \begin{proof}
   Since the localisation of $\TT$ at $I_{\rm Eis}$ is Gorenstein, the $I_{\rm Eis}$-torsion submodule of $X$ is isomorphic to $\bX/I_{\rm Eis} \bX$, and the result therefore follows from the fact that $\TT/I_{\rm Eis}$ is isomorphic to $\Z_p$.
   \end{proof}

 A generator of the $\Z_p$-module $\bX[I_{\rm Eis}]$ 
 is called an {\em Eisenstein element} in $\bX$. 
  Although such generators  are only well  defined up to 
  scaling by $\Z_p^\times$, the concrete Hecke modules that arise in practice are
   frequently equipped with a distinguished choice of  Eisenstein element $m_0$. 
Corollary \ref{cor:mazur-eisenstein} implies the following lemma:
\begin{lemma}
There is an element
   $m_1\in \bX/p^t\bX$ satisfying $U_N  m_1 =m_1$ and 
\begin{equation}
\label{eqn:higher-eisenstein-relation}
 (T_\ell-(\ell+1)) m_1 = (\ell-1) \log(\ell) m_0 \pmod{p^t}, \qquad  \mbox{ for every prime } \ell\ne N,
    \end{equation}
    and the choice of $m_0$ uniquely specifies $m_1$ up to the addition of a multiple of $m_0$. 
    \end{lemma}  
   The element $m_1$ depends linearly on the choice of discrete logarithm, namely, 
    replacing $\log$ by $a\cdot \log$  with $a\in (\Z/p^t\Z)^\times$
    has the effect of replacing $m_1$ by 
    $a m_1$.
    \begin{definition}\label{first-hee}
    The element $m_1$     
   is called  the {\em higher Eisenstein element }  in $\bX/p^t$
   (associated to $m_0$ and to the choice of discrete logarithm).
   \end{definition}
   
   For example, (the Eisenstein completion) $\bM = M_2(N)_\fm$  of 
    the module  of modular forms 
    has 
 a distinguished Eisenstein element $m_0 = \EN$.
 Proposition \ref{prop:mazur}  
 supplies an
 explicit description of the higher Eisenstein element $m_1 = E'$ in  
 $ \bM\otimes (\Z/p^t\Z)$.
 The proof of Conjecture \ref{conj:HV} for dihedral forms rests crucially on similar explicit 
expressions of the   higher Eisenstein element  in various other
Hecke modules, which will be described in  the forthcoming sections.

\begin{remark}
When $\cM\equiv 0 \pmod{p^u}$ with $u\le t$, 
there is  also a {\em second higher Eisenstein element} $m_2\in \bX\otimes (\Z/p^u\Z)$ 
satisfying, for all primes $\ell\ne N$,
$$ (T_\ell-(\ell+1)) m_2  = (\ell-1) \log(\ell) m_1 \pmod{m_0 \bX}.$$
In fact, in $\bX\otimes(\Z/p\Z)$   there is an entire  sequence $m_0, m_1, \ldots, m_r \in \bX\otimes (\Z/p\Z)$ of higher
 Eisenstein elements obeying similar inductive relations, 
where $r+1$ is the  $\Z_p$-rank of $\TT$.  
These higher Eisenstein elements have been studied systematically in 
\cite{Lec1}, but only the first higher Eisenstein elements will play a role in this work. Henceforth, the terminology
``higher Eisenstein series" or ``higher Eisenstein element" shall always refer to what might
 be called the ``first higher
Eisenstein element" in \cite{Lec1}. 
\end{remark}

  \subsection{The Betti cohomology relative to the cusps}
  \label{sec:he-betti-rel}
  
  One of the  settings which turns out to be 
  relevant to the proof of Conjecture \ref{conj:HV} for RM dihedral forms occurs when
  $\bX := \bH^+$ is the $p$-Eisenstein completion of the relative cohomology
  $H^1_{\Bet}(X_0(N); \{0,\infty\}; Z)^+$ with coefficients in the ring $Z:= \Z[1/6N]$, where the superscript $+$ denotes the subspace which is fixed by complex conjugation. 
  As discussed in \S \ref{notn}, the subscript $\Bet$
  means that we take the singular cohomology of the {\em complex points} of $X_0(N)$.  This relative cohomology is dual to $H^1_{\Bet}(Y_0(N),Z)^-$, which is isomorphic, after tensoring with $\C$, to the space of weight two modular forms on $\Gamma_0(N)$, via integration.
  In particular, the ring generated by the Hecke operators acting on   $H^1_{\Bet}(X_0(N); \{0,\infty\}; Z)^+$
  is naturally identified with $\TT(N)$.
   
 The  module $H^1_{\Bet}(X_0(N); \{0,\infty\}; Z)^+$
 fits into the short exact sequence
\begin{equation}
\label{eqn:ses-betti-rel}
 0 \lra Z \stackrel{\partial^*} {\lra}
  H^1_{\Bet}(X_0(N); \{0,\infty\}; Z)^+ \stackrel{i^*}{\lra} H^1_{\Bet}(X_0(N),Z)^+ \lra 0 
  \end{equation}
  of $\TT(N)$-modules,
  where $\partial^*$ is dual to the boundary homomorphism
\begin{equation}
\label{eqn:boundary-hom}
\partial: H_{1,\Bet}(X_0(N); \{0,\infty\}; Z) \lra Z \cdot(0-\infty) = Z. 
\end{equation}
The relative cohomology group $H^1_{\Bet}(X_0(N); \{0,\infty\};Z)$ can be described concretely in terms of $Z$-valued modular symbols:  $\Gamma_0(N)$-invariant functions $m$  from 
 $\PP_1(\Q)\times \PP_1(\Q)$ to $Z$  which are {\em additive } in the sense that they satisfy
 $$ m\{a,b\} + m\{b, c\} = m\{a,c\} \qquad \mbox{ for all } a,b,c \in \PP_1(\Q).$$
 The image of the class 
  $\partial^*(1)$ in $\bH^+$,   denoted $\kappa_0^+$,
     is a distinguished Eisenstein element in $\bH^+$, which corresponds to the {\em boundary symbol}
  sending $(a,b)$ to $f_\infty(b) - f_\infty(a)$, where $f_\infty$ is the unique $\Gamma_0(N)$-invariant function on $\PP_1(\Q)$ which sends $\infty$ to $1$ and $0$ to $0$. Let 
   $$\bar\kappa_1^+: \Gamma_0(N) \lra (\Z/p^t\Z), \qquad
 \left(\begin{matrix}  a& b \\ c & d\end{matrix}\right) \mapsto \log(a).$$

Since it is trivial on  parabolic elements,
   it can be viewed as an element of $H^1_{\Bet}(X_0(N),\Z/p^t\Z)^+$. 
  Let $\kappa_1^+ \in \mathbb{H}^+ \otimes (\Z/p^t \Z)$ be the class obtained by
  choosing a   
  preimage of
  $\bar\kappa_1^+$ under $i^*$, in the exact sequence obtained from
  \eqref{eqn:ses-betti-rel} by replacing $Z$ with $(\Z/p^t\Z)$, 
and projecting it to
 $\bH^+$.
   This class depends on the choice of preimage, but only up to the addition of a 
   multiple of $\kappa_0^+$.  
   Furthermore, it is annihilated by   $I_{\rm Eis}^2$,
   since $\bar \kappa_1^+$ is annihilated by $I_{\rm Eis}$, and therefore, 
    for all rational primes $\ell\ne N$,
 the class
$(T_\ell-(\ell+1)) \kappa_1^+$ is a multiple of the boundary symbol $\kappa_0^+$. 
   \begin{theorem}
  \label{thm:eisenstein-relative-H1}
  The class $\kappa_1^+$ is  the higher Eisenstein element in $\mathbb{H}^+ \otimes (\Z/p^t\Z)$ attached to
  $\kappa_0^+$. 
  \end{theorem}
 \begin{proof}
The  modular symbol attached to $\kappa_1^+$ admits an explicit description when 
restricted to $\Gamma_0(N) 0 \times \Gamma_0(N) 0$.  Namely, 
if  $r/s$ and $t/u$  (viewed as fractions in lowest terms, with the convention that $\infty =1/0$, 
so that, in particular, $s$ and $ u$ belong to $ (\Z/N\Z)^\times$) are
 elements of this $\Gamma_0(N)$-orbit,
 we have
 $$ \kappa_1^+(\{r/s,t/u\}) =  
 \log(s/u).
 $$ 
 This fact is   proved by observing that the matrix
$$ \gamma:=  \mat{u'}{t}{\ast}{u} \mat {s}{-r}{\ast}{s'} \in  \Gamma_0(N), \qquad uu' \equiv ss' \equiv 1 \pmod{N}$$ 
 sends $r/s$ to $t/u$,  and hence  $\kappa_1^+(\{r/s,t/u\}) = \bar\kappa_1^+(\gamma) = \log(su')$. 
 To calculate the constant of proportionality relating $(T_\ell-(\ell+1))\kappa_1^+$ and 
$\kappa_0^+$,  
we exploit the usual formula for the action of the Hecke operators on modular symbols
(cf. \cite[Prop 18.9]{Maz}): 
\begin{eqnarray*}
 (T_\ell-(\ell+1)) \kappa_1^+ (\{0, \infty\}) &=&  \kappa_1^+\left( \{0, \infty\} + \sum_{i=0}^{\ell-1} \{i/\ell, \infty\} - (\ell+1)\{0, \infty\}\right) \\
 &=& \sum_{i=1}^{\ell-1}\kappa_1^+(\{i/\ell, 0\})  = (\ell-1) \log(\ell)   \\
 &=& (\ell-1) \log(\ell)  \cdot \kappa_0^+(\{0,\infty\}).
 \end{eqnarray*}
 The result follows.
 \end{proof}
     
  \subsection{The Betti cohomology of the open modular curve}
  \label{sec:he-betti}
  
  Consider now the case where 
  $$ \bX = \bH^-  = H^1_{\Bet}(Y_0(N),Z)^-_\fm.$$
  The exact sequence 
  $$ 0 \lra  
  H^1_{\Bet}(X_0(N),Z)^- {\lra} H^1_{\Bet}(Y_0(N),\Z)^- \lra  Z \lra 0$$
  produces an explicit rank one quotient of $H^1_{\Bet}(Y_0(N),\Z)^-$ which is 
  Eisenstein. The   Eisenstein element $\kappa_0^- $ in $\bH^-$ 
  is described by the Dedekind-Rademacher homomorphism on $\Gamma_0(N)$ described in 
  \cite[\S II.2]{Mazur-special}:
  $$ 
  \kappa_0^-(\gamma) = \frac{1}{2\pi i}\left(\log(\Delta_N)(\gamma z) - \log(\Delta_N)(z)\right),
   \qquad \Delta_N(z) := \Delta(Nz)/\Delta(z),
   $$
  which  encodes the periods of the modular unit $\Delta_N \in \cO_{Y_0(N)}^\times$. 
  It is given by the formula
  $$ 
  \kappa_0^-\left(\begin{array}{cc} a& b \\ Nc & d 
  \end{array}\right) = \begin{cases} 
  (N-1)b/d & \mbox{ if } c=0;\\
  \frac{(N-1)(a+d)}{cN} + 12 {\rm sign}(c) D^N\left(\frac{a}{N|c|}\right) & \mbox{ if } c\ne 0,
  \end{cases}
  $$
where $\DD^N(x) = \DD(x) - \DD(Nx)$ and $\DD$ is the Dedekind sum
$$ \DD(a/m) = \sum_{j=1}^{m-1} B_1(j/m) B_1(aj/m), \qquad \mbox{ for } m>0,  \quad \gcd(a,m)=1.$$
The homomorphism $\kappa_0^-$ can also be written as 
\begin{equation}
\label{eqn:DR-on-N}
 \kappa_0^-\left(\begin{array}{cc} a& b \\ Nc & d 
  \end{array}\right) = \varphi\left(\begin{array}{cc} a& b \\ Nc & d 
  \end{array}\right) - \varphi\left(\begin{array}{cc} a& Nb \\ c & d 
  \end{array}\right),
  \end{equation} 
  where $\varphi: \SL_2(\Z) \lra \Z$ is the Rademacher $\varphi$-function  given by  
 \begin{equation}
 \label{eqn:zagier-rho}
   \varphi\left(\begin{array}{cc} a& b \\ c & d 
  \end{array}\right) =     \begin{cases} 
  -b/d & \mbox{ if } c=0;\\
  \frac{-(a+d)}{c} + 12 {\rm sign}(c) \ \DD\left(\frac{a}{|c|}\right) & \mbox{ if } c\ne 0.
  \end{cases}
\end{equation}
  
In \cite{Lec1},  a formula for the higher Eisenstein element attached to $\kappa_0^-$ 
is given, which we omit because it shall not be needed in this work.

\subsection{The dual of the modular forms}
\label{sec:he-Mdual}
This section considers the case where 
 $\bX := \bM^*$ is the completion of   
 $$M_2(N)^\vee = \hom(M_2(N),Z) $$
  at the 
$p$-Eisenstein ideal. 
It is  a free $\TT$-module of rank one, and is also
equipped with an Eisenstein element $\fS_0$ defined by
$$ \fS_0(f) = a_0(f),$$
where $a_0(f)$ denotes the constant term of the modular form $f$ at the cusp $\infty\in X_0(N)$.
Let $\fS_1$ denote the higher Eisenstein element in $\bar\bM^* := \bM^*\otimes (\Z/p^t\Z)$ attached to $\fS_0$. 
It  turns out to be related   to the {\em Shimura class} $\fS$   described in the introduction.
   More precisely, the inclusion $S_2(N)\hookrightarrow M_2(N)$ induces a surjection $\bM^* \rightarrow \bS^*$. Fix any lift of 
   $\fS$ to  $\bar \bM^* = \bM^* \otimes (\Z/p^t\Z)$ via this surjection, denoted $\fS_1$.  Note that $ \fS_1$ is not completely well-defined, but that
   any two choices of lift differ by a multiple of $\fS_0$ (mod $p^t$).
   \begin{theorem}
   \label{thm:shimura-hes}
   The class
    $\fS_1$  is the  Higher 
Eisenstein element in $  \bM^* \otimes (\Z/p^t \Z)$ attached to  the Eisenstein   class $\fS_0$.
\end{theorem}
\newcommand{\cusps}{\mathrm{cusps}}
\begin{proof}
The class $\fS_1$ arises from the  $\bar\kappa_1^+$ described in the discussion  
preceeding Theorem
  \ref{thm:eisenstein-relative-H1}
  by means of the ``{\'e}tale to coherent'' morphism
$  H^1_{\rm et}(X_0(N), \Z/p^t\Z)   \lra 
H^1(X_0(N)_{/\Z/p^t\Z}, {\mathbb G}_a)
$. 

For reasons that will become clear below,
instead of working with $X_0(N)$ over the spectrum of $\Z_p$,
we will use instead  an unramified extension $W$ of  $\Z_p$ containing the $N$th roots of unity. 
Clearly it is enough to prove the claimed statement in $\bM^* \otimes (W/p^t W)$
instead of $\bM^* \otimes \Z/p^t$ since $\Z/p^t \hookrightarrow W/p^t W$. 

Let   $\iota: \cusps \hookrightarrow X_0(N)$,
be the inclusion of the cuspidal divisor, a relative divisor over $Z$. Let $j: Y_0(N) \rightarrow X_0(N)$
be the complementary open immersion.
Now, there are compatible short exact sequences of {\'e}tale sheaves on $X_0(N)_{W/p^t}$,
the base change of $X_0(N)$ along $Z \rightarrow W/p^t$:
\begin{equation} \label{sheafdiagram}
\xymatrix{
j_{!}(\Z/p^t)   \ar[d]  \ar[r]  & \mathcal{O}(-\cusps) \ar[d]  \\
 (\Z/p^t)   \ar[r]\ar[d] & \mathcal{O} \ar[d] \\
 i_*(\Z/p^t) \ar[r]   & \mathcal{O}_{\cusps}
 }
\end{equation}
Note that
we are dealing here with {\'e}tale sheaves whose order is not prime to the
residual degrees, but all we are using is the existence of this diagram. 
Taking cohomology now gives 
  the following commutative diagram which is
   compatible with Hecke operators: 
   \medskip
\begin{equation} 
\label{BigDiagram} 
{\tiny
\xymatrix{
   H^0_{\rm et}({\rm cusps}_{W}; \Z/p^t \Z)  \ar[d] \ar[r]& H^0(\mathrm{cusps}_{W/p^t}, \mathcal{O}) \ar[d] & \\
     H^1_{\rm et}(X_0(N)_{W}, {\rm cusps}_{W}; \Z/p^t\Z)  \ar@{->>}[d]  \ar[r]  &   
     H^1(X_0(N)_{W/p^t}, \mathcal{O}(-\cusps))   \ar@{->>}[d] \ar[r]^{ } &
 \Hom(M_2(N), W/p^t)  \ar[d] \\
H^1_{\rm et}(X_0(N)_{W}, \Z/p^t\Z)^{(0)}  \ar[r] &   
   H^1(X_0(N)_{W/p^t}, \cO)  \ar[r]^{ } & \Hom(S_2(N),  W/p^t).   } }
\end{equation}

\medskip\noindent
Here the groups in the  middle column are Zariski cohomology groups; 
  the map from left to middle column arises from, first of all,
  restricting to $W/p^t$, then using \eqref{sheafdiagram}
  and the fact that coherent sheaves have the same cohomology in Zariski and {\'e}tale topology. 
 The zero superscript in the  bottom left of
 \eqref{BigDiagram} refers to classes that
are trivial when pulled back to the cusps. 
The maps from middle to right are induced by 
the Serre duality pairings
as in \eqref{gdp}.

The Shimura class  $\mathfrak{S} 
\in H^1_{\et}(X_0(N)_{\Z_p}, \Z/p^t)$  
  gives rise to a class in the group
   $H^1_{\et}(X_0(N)_W, \Z/p^t)^{(0)}$
   in the lower left of  \eqref{BigDiagram} 
(which we also denote by $\mathfrak{S}$), i.e.
$\mathfrak{S}$ becomes trivial when pulled back to the cusps -- 
because the cusps are defined over $W$.
Fix a lift 
$$\tilde{\mathfrak{S}} \in  H^1_{\rm et}(X_0(N)_{W}, {\rm cusps}_{W}; \Z/p^t\Z)$$
 to the middle left group in \eqref{BigDiagram}.
 Now this left hand  term can be compared 
  with \eqref{eqn:ses-betti-rel}
 via restriction to the  geometric generic fiber, i.e., the fiber over $\overline{\mathbb{Q}}_p$, and it follows from
  Theorem \ref{thm:eisenstein-relative-H1} that 
  $$(T_{\ell}-\ell-1) \tilde{\mathfrak{S}} = (\ell-1) \log(\ell) \mathfrak{S}$$
  holds after restriction to this geometric generic fiber. 
  
  We claim that ``restriction to the geometric generic fiber''
  is injective on the group $H^1_{\et}(X_0(N)_{W}, \mbox{cusps}_{W};\Z/p^t)$. To see this,
 let $E =W \otimes \Q_p$ be the quotient field of $W$.
 In view of the diagram
 \eqref{BigDiagram},   it is enough to check that 
  the  kernel of the map
  $$  q: H^1_{\et}(X_0(N)_{W}, \Z/p^t) \rightarrow H^1_{\et}(X_0(N)_{\overline{\Q}_p}, \Z/p^t)$$
  is precisely the image of  $H^1_{\et}(\Spec W, \Z/p^t)$ on the left. 
  
A class in the kernel of $q$ amounts to an {\'e}tale $\Z/p^t$-cover of $X_0(N)_{W}$ which becomes
   trivial on the geometric generic fiber. This cover is uniquely determined by its restriction to $X_0(N)_{E}$
 (see \cite[Th\'eor\`eme 3.8, Expos\'e X]{SGA1})  where it becomes trivial on passage to a finite field extension of $E$, 
 i.e. the cover on $X_0(N)_{E}$ necessarily arises from a character $\Gal(\overline{\Q}_p/E) \rightarrow \Z/p^t$. 
 For such a cover to extend over $X_0(N)_W$ the character $\chi$ must be unramified.
 (For instance,  this can be seen by restricting to the cuspidal sections.) 
This implies the claim  regarding $\mathrm{ker}(q)$ and concludes the proof. 
 \end{proof}
 
    \begin{remark}
 Theorem    \ref{thm:shimura-hes} implies Merel's theorem that $\langle \fS_1, E_2^{(N)}\rangle = \cM$,
 where $\cM$ is the Merel constant of 
\eqref{eqn:merel-constant}, since, letting
 $E'$ be the mod $p^t$ modular form defined in Proposition 
\ref{prop:mazur}, 
 $ \langle \fS_1, E_2^{(N)}\rangle = \langle \fS_0, E'\rangle = a_0(E') = \cM.$
 \end{remark}

 \subsection{Supersingular divisors and modular units}
 \label{sec:supersingular}
 

Recall from \S \ref{tracedefstatement} the module $\Div(\cE)$ of  
 $\Z$-linear combinations of isomorphism classes of supersingular elliptic curves over $\overline{\F}_N$.

 The Jacquet-Langlands correspondence shows that $\Div(\cE)\otimes \C$ is abstractly isomorphic to
  $M_2(N;\C)$  as a module over the ring
of Hecke operators,   and in particular  
 the Hecke ring for $\Div(\cE)$ can be identified with $\mathbb{T}(N)$. 
    In this section we consider the case
where  $\bX:= \bD$ is the $p$-Eisenstein completion 
 of $\Div(\cE)$. 
 
 The vector (in the notation of \S  \ref{tracedefstatement})
\begin{equation}
\label{eqn:eis-ss}
 \Sigma_0 := \sum_{i=1}^n \frac{e_i}{w_i} \in \bD
 \end{equation}
satisfies   $T_{\ell} \Sigma_0 = (\ell+1)\Sigma_0$ for all $\ell \ne N$, 
and is  thus an
 Eisenstein element in $\bD$. 


Let $\Sigma_1 \in \bD \otimes (\Z/p^t\Z)$ denote the Higher Eisenstein element associated
to $\Sigma_0$, as specified in Definition \ref{first-hee}. 
The main goal of this section is to give an explicit construction of
$\Sigma_1$ in terms of the restrictions of certain 
modular units to the supersingular locus. 
 This construction   is inspired from  
  \cite{Lec1} and involves  
  the Eisenstein  series $E_{N+1}$ of weight $N+1$,
   and  the cusp form $\Delta$ of weight $12$, viewed as modular forms mod $N$ of 
   level $1$.

 Let $\cO_N$ denote the ring of (meromorphic) modular functions  on the modular curve 
of level one  over $\Spec(\Z/N\Z)$ 
that are regular at its supersingular points.
Since  $p$ is odd and 
$p\nmid N+1$, 
the discrete logarithm  
$\log$ 
extends uniquely 
to the multiplicative group $\F_{N^2}^\times$, and  can therefore be used to define
a  homomorphism 
\begin{equation}\label{Log-def}
 \Log: \cO_{N}^\times \lra \Div(\cE) \otimes (\Z/p^t\Z), \qquad \Log(U) :=  
  \sum_{i=1}^n \log(U(e_i)) \cdot   \frac{e_i}{w_i}.
  \end{equation}

It shall be useful to introduce {\em multiplicative Hecke operators}  acting on the multiplicative 
monoid  in  the graded ring
of modular forms mod $N$.
To describe these operators, we 
 shall adopt  Katz's point of view to describe modular forms over a ring.
 Recall that a {\em Katz test object} over $\F_N=\Z/N\Z$ is a pair 
 $(A, \omega)_{/R}$,   where
 \begin{enumerate}
 \item[(i)]  $A$ is   an elliptic curve over  a $\F_N$-algebra $R$;
 \item[(ii)]  $\omega$ in an $R$-module generator 
  of $H^0(A,\Omega^1_A)$.
  \end{enumerate}
  A {\em weakly holomorphic modular form  of weight $k$ and level $1$ over $\F_N$ }
   is a rule $f$ which to any such test  object
associates
 an invariant $f(A,\omega)\in R$,
 satisfying
  \begin{enumerate}
 \item $f(A,\omega)$ depends only on the  $R$-isomorphism class of $(A,\omega)$;
 \item $f$ commutes with base change with respect to any homomorphism $R \ra R'$ of $\F_N$-algebras, in the obvious sense;
 \item $f(A,u \omega) = u^{-k} f(A, \omega)$, for any $u\in R^\times$.
 \end{enumerate}
 Let  $(A_{q}, \omega_{\rm can})$  denote the ``Tate test object" over $\F_N((q))$, whose   points over this local field are identified with $\F_N((q))^\times/q^\Z$, 
   equipped with its canonical differential
    $\omega_{\rm can} = dt/t$.
 If 
  $f(A_{q}, \omega_{\rm can})$ 
 lies in   $\F_N[[q]]$ (resp.~$q\F_N[[q]]$),
then $f$ is called a modular form (resp.~a cusp form). The space of modular forms and cusp forms of weight $k$ and level $1$
   over $\F_N$ shall simply
   be denoted $M_k$ and $S_k$ respectively.

Let $\ell \ne N$ be a   prime. 
The {\em multiplicative Hecke operator} $$T_\ell^\times: M_k  \lra M_{k(\ell+1)}  $$ 
is defined  by setting
\begin{equation}
 \label{XXX}
 (T_{\ell}^\times f)(A, \omega) = \prod_{\varphi} f(A',  \omega'),
 \end{equation}
 where the product is taken over the distinct isogenies $\varphi:A \rightarrow A'$ of degree $\ell$,
 with  $\omega'$   determined by   $\omega := \varphi^* \omega'$.  Up to language
 this is already in  \cite{Hurwitz}.
 One readily checks that $T_\ell^\times$ maps $M_k$ to $M_{(\ell+1)k}$, as claimed.
Of course, $T_\ell^\times$ is not additive but it  is compatible
 with multiplication on the graded ring of
modular forms over $\F_N$:
$$ T_\ell^\times(fg) = T_\ell^\times(f) T_\ell^\times(g).$$
In particular, it induces   homomorphisms $T_\ell^\times: \cO_{N}^\times \lra \cO_{N}^\times$ for which 
 the diagram 
\begin{equation}
\label{eqn:diagrams} 
\xymatrix{  \cO_{N}^\times \ar[r]^{T_\ell^\times} \ar[d]^{\Log} & \cO_{N}^\times  \ar[d]^{\Log}\\
\Div(\cE)\otimes (\Z/p^t\Z) \ar[r]^{T_\ell} & \Div(\cE) \otimes (\Z/p^t\Z) 
 }
\end{equation}
 commutes.

Consider the meromorphic modular function
\begin{equation}
\label{eqn:defg-ratio}
 \Sigma^\times := \frac{E_{N+1}^{12}}{\Delta^{N+1}}
 \end{equation}
of  level one.  
By a result of Katz (\cite[Theorem 3.1]{Katz-Zannier}), 
 $E_{N+1}$ has no common zero with the Hasse invariant. 
Since the Hasse invariant has simple zeroes at the supersingular points, it follows that
$\Sigma^\times$ belongs to $\cO_N^\times$ and therefore that 
 the vector 
\begin{equation} \label{Sigma1def}
 \Sigma_1 := \frac{1}{12} \Log(\Sigma^\times) \in (\Z/p^t\Z) \otimes \Div(\cE) 
 \end{equation}
 is well-defined. Note that the class of $ \Sigma_1$ mod $\Z/p^t\Z\cdot \Sigma_0$ does not depend on the way one normalizes the constant term of $E_{N+1}$.

\begin{theorem}
\label{thm:sigma-g}
For 
all primes $\ell\ne N$,
$$ (T_\ell - (\ell+1)) \Sigma_1 =   (\ell-1)  \log(\ell) \Sigma_0,$$
and $\Sigma_1$  is therefore  equal to the higher Eisenstein element attached to $\Sigma_0\in \bD$.
\end{theorem}
\begin{proof}
While the   Eisenstein series  $E_{N+1}$ presumably exhibits a complicated behavior 
under the multiplicative Hecke operators, 
 a result of G. Robert (\cite[Th\'eor\`{e}me B]{Rob})  asserts that if  $(A,\omega)$ and $(A',\omega')$ 
 are marked {\em supersingular} elliptic curves 
and $\varphi:A\lra A'$ is an isogeny of degree $\ell$
   satisfying $\varphi^*(\omega') = \omega$, then
  \begin{equation}
  \label{eqn:robert}
  E_{N+1}(A', \omega') = \ell E_{N+1}(A,\omega), \qquad \mbox{ for all } A \in {\cE}.
  \end{equation}
  It follows that
\begin{equation}
\label{eqn:mult-E2}
 T^\times_{\ell} E_{N+1} = \ell^{\ell+1} E^{\ell+1}_{N+1}.
\end{equation}
  In addition, for every prime $\ell \ne N$,
\begin{equation}
\label{eqn:mult-Delta}
T_{\ell}^\times \Delta   = {\ell}^{12} \Delta^{\ell+1}.
\end{equation}
This follows by noting that 
$$ T_\ell^\times(\Delta)(E_q,   \omega_{\rm can}) = \Delta(E_{q^\ell}, \ell^{-1}\omega_{\rm can})
\times \prod_{\zeta\in\mu_\ell} \Delta(E_{\zeta q^{1/\ell}},  \omega_{\rm can}) = 
\ell^{12} \Delta(q)^{\ell+1}.$$
 Combining 
  \eqref{eqn:mult-E2} and \eqref{eqn:mult-Delta}, we obtain
 $$ T_\ell^\times( \Sigma^\times )= \ell^{12(\ell-1)}  (\Sigma^{\times})^{\ell+1}.$$
It   follows that
 $$ T_\ell(\Log(\Sigma^\times)) =   12(\ell-1) \log(\ell) \Sigma_0 +  (\ell+1) \Log(\Sigma^\times),$$
as claimed.
  \end{proof}
  
\begin{example}
Take $N=23$ and $p=11$. The supersingular $j$-invariants mod $N$ are $\{ 1728,19, 0\}$ and we have $\Sigma_0=(6,1,4)$ with respect to this basis. Normalize $\log: (\Z/N\Z)^\times \lra \Z/p\Z$ by setting $\log(5)=1$. Vector $\Sigma_1=\frac{1}{12} \Log(\Sigma^*) = \Log(\frac{E_{24}}{\Delta^{2}}) \in \F_p e_{1728} \oplus \F_p e_{19} \oplus \F_p e_0$ is then computed to be
$$
\Sigma_1 = (-1,-1,-3).
$$
This can readily be checked for instance by means of the identity 
$$
\frac{E_{24}}{\Delta^{2}} = (aj^2+b(j^2-1728j) +c(j-1728)^2)/d,
$$
where 
$$ 
a=49679091, \ \ b=176400000, \ \ c=10285000, \ \ d=236364091,
$$
which follows by comparing the $q$-expansions of $E_4$, $E_6$ and $E_{24}$. A computation with Brandt matrices allows to verify numerically the identity of Theorem \ref{thm:sigma-g}.
\end{example}

The   description of $\Sigma_1$ given in Theorem \ref{thm:sigma-g} 
makes it possible to relate some  of its pullbacks to modular units.
More precisely, let $q\ne N$ be an auxiliary prime,  let $\cE^{(q)}$  denote the set of 
of supersingular points of the modular curve
$X_0(q)$ in characteristic $N$ (i.e. over $\overline{\mathbb{F}_N}$) and let 
$\Div(\cE^{(q)})$ and $\bD^{(q)}$ denote (respectively) the space of  $Z$- and
$(\Z/p^t\Z)$-linear combinations of elements of $\cE^{(q)}$.  Note that 
in carrying over constructions from $\cE$ to $\cE^{(q)}$ we must
take account of the fact that the weights $w_x$ for $x \in \cE^{(q)}$
take into account the level structure and thus will not in general coincide
with the weight $w_{\bar{x}}$ of the image $\bar{x} \in \cE$.

The two degeneracy maps
\begin{equation}
\label{eqn:degeneracy-maps}
 \pi_1,\pi_2: X_0(q) \lra X(1), \qquad \pi_1(A,C) = A, \quad \pi_2(A) = A/C
 \end{equation}
induce maps $\pi_1,\pi_2: \cE^{(q)} \lra \cE$, 
and correspondingly push-forward maps 
$$ \pi_{1*}, \pi_{2*} : \mathrm{Div}(\cE^{(q)})  \lra \mathrm{Div}(\cE).$$
The dual of these maps are 
  pullback maps $\pi_1^*, \pi_2^*$, defined so as to satisfy
$$ \langle \pi_j^* a, b\rangle_q = \langle a, \pi_{j\ast} b \rangle, \qquad \mbox{ for all } a\in \mathrm{Div}(\cE), \ b\in \mathrm{Div}(\cE^{(q)}).$$
where $\langle - , - \rangle_q$ and $\langle -, - \rangle$ are the natural pairings (cf. \eqref{divpair}). 
In particular we get
$$ \pi_1^\ast, \pi_2^\ast: \bD  \lra \bD^{(q)},$$
which is, now, compatible with  the corresponding pullback of functions
on the ambient modular curves by means of the map \eqref{Log-def}.

Just as in \eqref{Log-def} we have
a  homomorphism  
\begin{equation}\label{Log-def2}
 \Log: \cO_{q,N}^\times \lra \Div(\cE^{(q)}) \otimes (\Z/p^t\Z), 
  \end{equation}
  where now $\cO_{q,N}^{\times}$ denotes the multiplicative group
  of (meromorphic) modular functions on $X_0(q)_{\mathbb{F}_N}$
  regular at the supersingular points. 
%
This applies to the case where $f=  \pi_1^*(\Delta)/\pi_2^*(\Delta) = \Delta(z)/\Delta(qz)$, which is a modular unit of level $q$. 
 \begin{theorem}
 \label{thm:modular-unit-q}
 For any  auxiliary prime  $q \ne N$, denote by 
\begin{equation}
\label{eqn:modular-unit-q}
u_q := \Delta(z)/\Delta(qz)
\end{equation}
the modular unit of level $q$, considered as an element of $\mathcal{O}_{q,N}^{\times}$ (see \eqref{Log-def2}). Then
 $$ \pi_1^*(\Sigma_1) - \pi_2^*(\Sigma_1) = -\frac{1}{6} \Log(u_q) \pmod{\Sigma_0^{(q)}},$$
where $\Sigma_0^{(q)} = \pi_1^*(\Sigma_0) = \pi_2^*(\Sigma_0)$  is an Eisenstein eigenvector on $\bD^{(q)}$. 
 \end{theorem}

The use of the auxiliary prime $q$
simplifies the situation: the
map $(\pi_1^* -\pi_2^*)$ kills $\Sigma_0$; thus
$(\pi_1^* -\pi_2^*) \Sigma_1$ is  independent of the choice of $\Sigma_1$
and is {\em strictly} Eisenstein, rather than higher Eisenstein.   In fact,
in the case $q=2$, this general idea appears in the work of Lecouturier;
 the role of the modular unit \eqref{eqn:modular-unit-q} is replaced in his work by the $\lambda$-invariant, cf. \cite[Prop 3.25]{Lec}.

 \begin{proof}
Equation \eqref{eqn:robert} shows that $\pi_1^*(E_{N+1})/\pi_2^*(E_{N+1})$ is constant on $\cE^{(q)}$, and hence 
$$ \Log( \pi_1^*(E_{N+1})/\pi_2^*(E_{N+1})) \sim \Sigma_0^{(q)},$$
where $\sim$ indicates that the two vectors are proportional to each other. 
It follows from the definition \eqref{eqn:defg-ratio} of $\Sigma^{\times}$  and $N \equiv 1$ modulo $p^t$ that
$$ \Log(\pi_1^*(\Sigma^\times)/\pi_2^*(\Sigma^\times)) =  2 \Log(\Delta(qz)/\Delta(z)) \pmod{\Sigma_0^{(q)}},$$
and the claim follows from the definition \eqref{Sigma1def} of $\Sigma_1$. 
\end{proof}

  \subsection{Tensor products}
  \label{sec:he-tensor}
  
  Let $M$ and $N$ be any two free modules of rank one over $\TT$. 
  The tensor product $M\otimes_{\TT} N$ is still free of rank one.
  If  $m_0$ and $m_1$  (resp.~$n_0$ and $n_1$) 
  are   Eisenstein and higher Eisenstein  elements  in $M$ (resp.~ $N$),
   there seems to be no simple expression for the higher Eisenstein element
   in $M\otimes_{\TT} N$ in terms of these elements. (For instance, the vector $m_0\otimes n_0$ fails to generate the Eisenstein subspace in $M\otimes_{\TT} N$ in general.)  
   
   Since $\TT$ is Gorenstein, the $\Z_p$-dual $M^*=\Hom(M,\Z_p)$ is again a free $\TT$-module of rank $1$ and hence it makes sense to consider (higher) Eisenstein elements on it.

   \begin{proposition}
   \label{prop:he-tensor}
   If $m_0^*$ and $m_1^*$  (resp.~$n_0^*$ and $n_1^*$)  are the Eisenstein and higher Eisenstein elements of 
   $M^*$ and $N^*$ respectively, then
   \begin{enumerate}
   \item The element $m_0^* \otimes n_0^*$ is an Eisenstein element of $(M\otimes_{\TT} N)^*$. 
   \item The element $m_0^*\otimes n_1^*+ m_1^* \otimes n_0^*$ is the higher Eisenstein element of $(M\otimes_{\TT} N)^*/p^t$ associated to $m_0^*\otimes n_0^*$. 
   \end{enumerate}
   \end{proposition}
      Note that there is a natural module homomorphism $M^*\otimes_{\Z_p}  N^* \lra (M\otimes_{\Z_p}  N)^*$ sending $m^*\otimes n^*$ to the functional defined by $(m^*\otimes n^*)(a\otimes b) =
   m^*(a) n^*(b)$.  The meaning of the first statement above is, then, that the displayed expressions
   in fact belong to $(M\otimes_{\TT} N)^* \subset (M \otimes_{\Z_p} N)^*$, and moreover are 
   Eisenstein/higher Eisenstein considered in the former group.     Similarly for the second statement (see below for details). 
   
   \begin{proof} As for (1), we first check that $m_0^* \otimes n_0^*$ belongs to the submodule $(M\otimes_{\TT} N)^*$ of $(M\otimes_{\Z_p} N)^*$. The kernel of the surjection $M\otimes_{\Z_p} N \lra M\otimes_{\TT} N$  is generated by $(T\otimes 1 -1\otimes T)(M\otimes N)$ for $T\in \TT$. Hence it suffices to verify that $(T\otimes 1 -1\otimes T)(m_0^* \otimes n_0^*)=0$ for all $T\in \TT$, and this follows because $\TT$ is a simple algebra over $\Z_p$, generated by an element of $I_{\rm Eis}$. Now (1) follows, as it is obvious that  $m_0^* \otimes n_0^*$ is a generator of the $\Z_p$-module $(M\otimes_{\TT} N)^*[I_{\rm Eis}]$.
   
   As for (2), write $\bar M := M/p^tM$ 
   and  $\bar N := N/p^tN$.  
   Note that
   $M^*/p^t \simeq \bar M^*$ where, on the right, $*$ denotes $\Hom(-, \Z/p^t)$. 
   The expression $m_0^*\otimes n_1^*+ m_1^* \otimes n_0^*$ lies in 
   $$ (M^* \otimes_{\Z_p} N^*)/p^t =  \bar{M}^* \otimes_{\Z/p^t} \bar{N}^*.$$
    We argue as before that $m_0^*\otimes n_1^*+ m_1^* \otimes n_0^*$ lies in $(\bar M\otimes_{\TT} \bar N)^*$.
    The $\TT$-module structure is given by applying $T \in \mathbb{T}$
    to either the first or second argument.  Applying $T_\ell-\ell-1$ to the first argument gives:
     \begin{eqnarray*}
    ( T_{\ell} - \ell-1) \left[ m_0^*\otimes n_1^*+ m_1^* \otimes n_0^* \right] &=& (T_{\ell}-\ell-1) m_0^* \otimes n_1^* +  (T_{\ell} - \ell -1 ) m_1^* \otimes n_0^* \\ 
     &=& (\ell-1) \log(\ell) m_0^* \otimes n_0^*,
      \end{eqnarray*}
 as desired.

\end{proof}

\section{Proof of the main theorem}
\label{sec:cm-forms}
This chapter proves  Conjecture \ref{conj:HV} for dihedral modular forms.


\subsection{Elliptic units}
\label{sec:elliptic-units}
We put ourselves in the situation of \S \ref{Heegnersetup} and
\S \ref{tracedefstatement} with $\psi_2=\psi_1^{-1}$ and $\psi_2 \neq \psi_1$. 
In particular: 
  $K$ is an imaginary quadratic field of odd discriminant $D<0$
and ring of integers $\mathfrak{o}$; the level
$N$ is   prime, $p >3$ a prime dividing $N-1$, and 
 $\psi_1: \mathcal{C} \lra L^{\times}$
a class group character into 
some cyclotomic field $L$.
Let $R$ be the ring of  integers of $L$. 

Finally put
\begin{equation} \label{psitwo} \psi = \psi_1/\psi_1' = \psi_1^2: \mathcal{C} \rightarrow L^{\times}.\end{equation}

When $N$ splits in $K$, 
Conjecture  \ref{conj:HV}
reduces to the equality $0=0$,
as explained  in
  \S   \ref{sec:trivial-cases}
of  the introduction.
Hence it shall be assumed throughout that $N$ is inert in $K$.

  Let $\mathbb{F}_{N^2}$ be the quotient $\fO/N$, a finite field of size $N^2$, and fix an algebraic 
  closure $\overline{\mathbb{F}}_N$ of $\mathbb{F}_{N^2}$. 
  
%
%
 
  In the current section only $\psi$ will be relevant (and the discussion
 would be valid for an arbitrary character $\psi$, not just one of the form \eqref{psitwo}). 
 We will construct an elliptic unit $u_{\psi}$ associated to $\psi$
 and explain how its discrete logarithm at various primes is related
 to the geometry of supersingular points. 
We will use the setup of \S \ref{Heegnersetup} regarding double coset spaces attached to
definite quaternion algebras,
but  will now use the incarnation of these spaces 
in terms of supersingular elliptic curves.

More precisely, global class field theory identifies $\mathcal{C}$ with the Galois group of an
abelian extension $H$ of $K$: the  Hilbert class field of
$K$, generated over $K$ by the $j$-invariants of 
elliptic curves over $\bar K$  with endomorphism ring equal to $\fO$.
The set of all such elliptic curves up to ${\bar K}$-isomorphism,
denoted $\cE_{\fO}$,
 is a principal transitive
$\mathcal{C}$-set and the choice of a base point $A\in \cE_\fO$ identifies 
the two sets
 $$ \mathfrak{a} \in \mathcal{C} \mapsto A_{\mathfrak{a}} \in \cE_{\fO}.$$
via tensoring with the inverse of $\mathfrak{a}$.  


 The prime $N$, which is inert in $K/\Q$, splits completely in $H/K$, and the choice of a prime 
 $\mathfrak{N}$ of $H$ above $N$ determines 
 reduction maps
 $$ \iota:  \cE_{\fO} \lra \cE,  \qquad \iota: \Pic(\fO) \lra \cE,$$
 where $\cE$ is the set of isomorphism classes of supersingular curves
 over $\overline{\mathbb{F}}_N$. 
 Since the end result we are proving is independent of the choice of $\mathfrak{N}$, 
 we can and will choose $\mathfrak{N}$ in such a way that 
 the reduction $\iota(A) \in \cE$ matches
 with one of the basepoints for $\cE$ chosen
 before \eqref{iotadef0}, i.e., to reprise,  the endomorphism ring of the reduction
 of $A$ at $\mathfrak{N}$ should contain an order of the form $\mathfrak{o} \oplus \mathfrak{o} j$.  
 
The map $\iota$ coincides with the map \eqref{iotadef0} after identifying $\cE$ with
maximal orders in the associated quaternion algebras, as specified prior to \eqref{iotadef0}.
%
As  in \eqref{psibrackdef}, the  image of $\psi$ under the pushforward map  
 $ \iota_\ast: R[\Pic(\mathfrak{o})] \longrightarrow \Div(\cE) \otimes R$
 is denoted by
  $   [\psi] := \iota_* (\psi)    \in \Div(\cE) \otimes R.$
  
  Let $q$ be an auxiliary rational prime which does not divide $D N$. 
A {\em Heegner point} on $X_0(q)(\bar{K})$ attached to  $\fO$ is a pair   $(A,C)$  where $A$ an elliptic curve over $\bar{K}$ equipped with a cyclic subgroup $C \subset A$ of  order $q$, 
 for which both $A$ and $A/C$ belong to $\cE_{\fO}$. 
The set $\cE_{\fO}^{(q)}$ of Heegner points on $X_0(q)(\bar{K})$ is non-empty precisely when
the prime $q\nmid D$ is split   in $K/\Q$, i.e.,  when $q = \fq \bar\fq$. It is then
 contained in  $X_0(q)(H)$. 
Just as  above, the choice of a prime $\mathfrak{N}$ of 
$\cO_H$ 
induces  reduction maps $\cE_{\fO}^{(q)} \rightarrow \cE^{(q)}$.

The set $\cE_{\fO}^{(q)}$ is equipped with the two degeneracy maps
$$ \pi_1,\pi_2: \cE_{\fO}^{(q)} \lra  \cE_{\fO}; \qquad
 \pi_1(A,C) = A, \quad \pi_2(A,C) = A/C,$$
obtained by restricting the corresponding degeneracy maps $X_0(q) \rightarrow X(1)$. 
The choice of  a prime divisor $\fq$ of $q$  determines a section  $\eta_\fq: \cE_{\fO} \lra \cE_{\fO}^{(q)}$ 
of $\pi_1$ by setting
 $$ \eta_\fq(A) =  \tilde{A} := (A, A[\fq]).$$
 Observe that the action of $\Pic(\fO)$ on $\cE_{\fO}$ 
satisfies
\begin{equation}
\label{eqn:action-compatibility}
A_{\fa\fq} = \pi_2(\eta_\fq(A_{\fa})).
 \end{equation}

  \begin{definition}\label{def-unit}
The {\em elliptic unit} attached to $\psi$ and $\fq$ is the element
\begin{equation} \label{upsiqdef} u_{\psi,q}  = \sum_{\fa \in {\rm Pic}(\fO)} u_q(\eta_{\fq}(A_\fa)) \otimes \psi(\fa)  \in H^{\times} \otimes R,\end{equation}
with $u_{q}$ the modular unit defined in 
 \eqref{eqn:modular-unit-q}.
\end{definition}

If $\psi$ is non-trivial, then $u_{\psi,q}$ belongs to $\cO_H^\times \otimes R$  and more precisely to its $\psi$-isotypical component, that is to say: 
\begin{equation}
\label{eqn:kl}
g \cdot u_{\psi,q} = \psi^{-1}(g) u_{\psi, q}, \qquad \mbox{ for  all } g \in \mathrm{Gal}(H/K). 
\end{equation}
(Cf.~\cite[\S 11, Thms.~1.1.~and~1.2]{kubert-lang}.) 
Note that on the  left hand side of \eqref{eqn:kl},
 $g$ acts on $H$ in the natural way. On the right-hand side,  $\psi$
is understood as a character of $\mathrm{Gal}(H/K)$ through the isomorphism
$\mathrm{Gal}(H/K) \simeq \mathcal{C}$  through which this Galois group acts on $\cE_{\fO}$, and
$\psi^{-1}(g)\in R^\times$ acts by multiplication on the second factor in the tensor 
product $\cO_H^\times \otimes R$.
 If $\psi=1$ then $u_{\psi,q}$ may fail to be a unit at the primes above $q$
but this case will not arise.

 The following proposition plays a key role in the proof of Conjecture \ref{conj:HV} for CM forms described in the next section,
since it is via this result that the relevant Stark unit makes its appearance.
\begin{proposition}
\label{prop:elliptic-units-mt}

For all characters $\psi: \mathcal{C} \rightarrow R^{\times}$, and all split primes $q = \fq\bar\fq$ as above, 
we have an equality in $R/p^t$: 
$$(1-\psi(\bar{\fq}))   \times  \langle \Sigma_1, [\psi] \rangle =    -\frac{1}{6} \log(u_{\psi,\fq}), $$
where  $\Sigma_1 \in \Div(\cE) \otimes \Z/p^t\Z$ is the higher Eisenstein element of Theorem \ref{thm:sigma-g},
and we wrote
$
\log: \cO_H^\times \otimes R \,\lra \, R/p^t
$
for the composition of the reduction map $\cO_H^\times \ra (\cO_H/\mathfrak{N})^\times \simeq \F_{N^2}^\times$ with discrete logarithm fixed at the outset.\footnote{This
discrete logarithm was defined on $(\Z/N\Z)^{\times}$ but uniquely extends to $\F_{N^2}^{\times}$. }

\end{proposition}

\begin{proof}
  Recall that $A$ is a fixed basepoint for $\cE_{\mathfrak{o}}$ and $[\psi] = \sum_{I \in \mathcal{C}} \psi(I) A_{I}$. We may write:
  \begin{eqnarray*}
  (1-\psi(\bar{\fq}))  \langle \Sigma_1, [\psi] \rangle 
   &=& \sum_{I \in \Pic(\fO)} (\psi(I) - \psi(I\bar{\fq}))  \langle \Sigma_1, A_{I} \rangle  \\
  &=&  \sum_{I \in \Pic(\fO)} \psi(I)  \langle \Sigma_1, A_{I} -   A_{I \fq} \rangle. 
  \end{eqnarray*}
   Letting ${\tilde A_I} := \eta_\fq(A_{I})$, we have, by  \eqref{eqn:action-compatibility},
   $$ A_I - A_{I\fq} = (\pi_1-\pi_2)_\ast(\tilde A_I),$$ and hence, 
   by invoking Theorem  \ref{thm:modular-unit-q},
 \begin{eqnarray*}
  (1-\psi(\bar{\fq}))  \langle \Sigma_1, [\psi] \rangle &=&   \sum_{I \in \Pic(\fO)} \psi(I) \langle (\pi_1^\ast - \pi_2^\ast) \Sigma_1, {\tilde A}_{I}  \rangle \\
  &=&  -\frac{1}{6}\cdot  \sum_{I \in \Pic(\fO)} \psi(I) \langle \Log(u_q), {\tilde A}_{I}  \rangle,
  \end{eqnarray*}
    with the pairings the natural ones on $\Div(\cE^{(q)})$. 
   The latter expression is equal to  $-\frac{1}{6} \langle \Log(u_q),[\psi]\rangle =   - \frac{1}{6} \log(u_{\psi,\fq})$, 
   the equality taking place in $R/p^t$: 
    \begin{eqnarray*} 
    \langle \Log(u_q),[\psi]\rangle & \stackrel{\eqref{Log-def2}}{=}&
     \sum_{I \in \mathcal{C}}   \log u_q(\iota \circ \eta_{\fq}(A_I) )  \psi(I)   \\
   &=&  \log \left(  \sum_{I \in \mathcal{C}}   u_q(\iota \circ \eta_{\fq}(A_I) ) \otimes  \psi(I)\right) \\ 
   &=&  \log   \mathrm{red}_{\mathfrak{N}} \sum_{I \in \mathcal{C}} u_q(\eta_{\fq}(A_I)) \otimes \psi(I) =  \log(u_{\psi,\fq}).
\end{eqnarray*}
           \end{proof}
           
%

\begin{remark}
One can replace the algebra $M_2(\Q)$ in the above considerations by a non-split, indefinite
quaternion algebra $D_M$ over $\Q$, of discriminant $M>1$ say, 
which is associated to a Shimura curve $X_M$  arising from a co-compact subgroup
of $\SL_2(\R)$. Given a prime $N\nmid M$, the module $\cE_{M,N}$ of supersingular points of $X_M$
in characteristic $N$ is identified with the space of functions on a finite double coset space attached
to the definite quaternion algebra $D_{MN}$ of discriminant $MN$. If $\psi$ is a character of the class group of
a quadratic imaginary field $K$ in which all the primes dividing $MN$ are inert, one can define an associated
vector $[\psi]\in \cX_{M,N}$ much as in the case where $M=1$.  The space $\cX_{M,N}$ contains an Eisenstein
eigenvector $\Sigma_0$,  whose value on a double coset is equal to the cardinality of its stabiliser subgroup.
Theorems 1.2 and 1.3 of \cite{hwajong-yoo}  show  that the Hecke algebra $\TT_{MN}$ 
acting on $\cX_{M,N}$ is equipped with an Eisenstein homomorphism $\tilde\varphi_{\rm Eis}$ as
in 
\eqref{eqn:phi-eis}
with $\TT$ replaced by $\TT_{MN}$, 
and
 suggest that,  if $p>3$ is a prime 
with $p^t || N-1$,  the module 
 $\cX_{M,N} \otimes (\Z/p^t\Z)$ contains a  generalised Eisenstein eigenvector 
$\Sigma_1$ attached to a choice of discrete logarithm $\log:\F_N^\times \lra \Z/p^t\Z$, satisfying
$$ (T_\ell - (\ell+1)) \Sigma_1 = (\ell-1) \log(\ell) \Sigma_0.$$
Does such a $\Sigma_1$,  when it  exists, satisfy an analogue of  Proposition
\ref{prop:elliptic-units-mt} relating $\langle \Sigma_1, [\psi]\rangle$ to the discrete logarithm of the elliptic unit $u_\psi$, which does not depend on $N$? Such a relationship would be intriguing in light of the fact that the arithmetic subgroup of
$\SL_2(\R)$ defining $X_M$ has no parabolic elements and hence there are no modular units
on  $X_M$ that 
could be parlayed into a direct construction of $\Sigma_1$. 
\end{remark}
 
 \subsection{Proof of Conjecture \ref{conj:HV} for definite theta series}
  
 We now restrict to the case $D$ prime; however, 
 as we comment in the statements,    the proofs verbatim give results for $D$ odd under further restrictions on $N$. 
 
 \label{sec:proof-cm2}  
 We   let $g=\theta_{\psi_1}$ be the associated $\theta$ series.  It is a cusp form by virtue of the assumption that $\psi_1 \neq \psi_1^{-1}$. 
 The Galois representation $\rho_g$ is the induction to $G_{\Q}$ of the finite order
 character $\psi_1$. 
Let $$G\in M_2(\Gamma_0(N)) = \mathsf{Tr}^{ND }_{N} g(z)g^*(Nz)$$ denote the modular form defined as the trace to the space of modular forms of weight $2$ and level $N$  of the product $g(z)g^*(Nz)= \theta_{\psi_1}(z) \theta_{\psi_1^{-1}}(Nz)$. 

%
Recall from \eqref{Theta-cor} the $\Theta$-correspondence
$$  \Theta: \Div(\cE) \otimes_{\mathbb{T}(N)} \Div(\cE) \rightarrow M_2(\Gamma_0(N)).$$
In particular, this induces a map on localizations at the Eisenstein ideal $\mathfrak{m}$, and it follows from 
   \cite[Theorem 0.5]{Em}  that the resulting map  is an isomorphism   of free $\mathbb{T}=\mathbb{T}(N)_{\mathfrak{m}}$-modules of rank one. 
Write
$$  \Tp = \mathbb{T}/p^t, \quad \bar \bD = \Div(\cE)_{\mathfrak{m}}/p^t, \quad \bar\bM = M_2(\Gamma_0(N)_{\mathfrak{m}}/p^t, \quad
\bar\bS = S_2(\Gamma_0(N))_{\mathfrak{m}}/p^t.$$
 Then, reducing $\Theta$ modulo $p^t$, we obtain an isomorphism 
\begin{equation} \label{Thetadef0} \Theta:  \bar\bD  \otimes_{\Tp} \bar \bD \simeq  \bar{\bM},\end{equation}
with associated adjoint
\begin{equation} \label{Thetadef1}  \Theta^*: \bar{\bM}^* \simeq (\bar\bD  \otimes_{\Tp} \bar \bD )^*.\end{equation}
Here $*$ denotes $\Hom(-, \Z/p^t)$. 
%
%
%

The strategy of the proof of Conjecture \ref{conj:HV}, as outlined in \S \ref{outline},  is to express the inner product
 $\langle G,  \fS \rangle$    as an inner product on $\bar \bD \otimes \bar \bD$ via 
$\Theta$.  
It follows from Theorem \ref{thm:waldspurger-garrett-CM-general} that
\begin{equation}
\label{eq:waldspurger-garrett-CM}
\langle G,\mathfrak{S}\rangle = 4 \cdot \langle \Theta([1] \otimes [\psi]), \mathfrak{S} \rangle =
4 \langle [1]\otimes [\psi], \Theta^*(\mathfrak{S})\rangle.
\end{equation}
Here we regard the equality as occurring inside $R/p^t$,
and we regard $\mathfrak{S} \in \bar{\bM}^*$ and $[1] \otimes [\psi] \in \bar{\bD} \otimes_{\Tp} \bar{\bD}$. 
 We now need:

\begin{theorem}
\label{thm:shimura-class}
Let $\fS_0$ and $\fS_1 \in \bar\bM^*$ denote the Eisenstein and higher classes
described in Section \ref{sec:he-Mdual}
and let 
$\Sigma_0$ and $\Sigma_1$ denote the analogous classes in $\bar \bD$ described in 
Section \ref{sec:supersingular}.
Then
\begin{enumerate}
\item
$\Theta^*(\fS_0) = \frac{1}{2} \Sigma_0 \otimes \Sigma_0$;
\item  $
 \Theta^* (\fS_1) \equiv    \frac{1}{2} (\Sigma_1 \otimes \Sigma_0 + \Sigma_0 \otimes \Sigma_1) \mbox{ modulo } \Sigma_0 \otimes \Sigma_0.$    \end{enumerate}
\end{theorem}

Here we used the pairing $\langle \ , \ \rangle$ given in \eqref{divpair}
to identify $\bar \bD \simeq (\bar \bD)^*$;
we also used the inclusion $\left( \bar{\bD}\otimes_{\Tp} {\bar \bD} \right)^*
\subset \left( \bar{\bD}^* \otimes_{\Z/p^t} {\bar \bD}^* \right)$ to describe
elements of the left hand group, just as was done in 
Proposition  
\ref{prop:elliptic-units-mt}.

\begin{proof}
The first part of the theorem follows directly from the definition of
$\Theta$ given in \eqref{Theta-cor}. The second follows from the Hecke equivariance of $\Theta^*$, in light of the fact that $\Sigma_1\otimes \Sigma_0 + \Sigma_1\otimes \Sigma_0$ is the higher Eisenstein element in
$(\bar\bD\otimes_{\TT} \bar\bD)^*$ attached to $\Sigma_0\otimes \Sigma_0$,
by Proposition \ref{prop:he-tensor}.
\end{proof}

 We now choose an auxiliary prime ideal $\fq$ so that $\psi(\fq)$ is a primitive
root of unity of order equal to the order of $\psi$.


 \begin{proposition}
 \label{prop:almost-final-cm}
There is an equality inside $R/p^t$ 
\begin{equation}
\label{eqn:almost-final-cm}
(1-\psi(\bar{\fq})) \langle G, \mathfrak{S} \rangle = \frac{ -h(\fO)}{3}  \log(u_{\psi,\fq}),
\end{equation}
where $u_{\psi, \fq}$ is the elliptic unit defined in 
  \eqref{def-unit}, and $h(\fO)$ is the order of the class group $\mathcal{C}$. 
\end{proposition}

Note that, for $D$ odd but not assumed prime,
the same conclusion holds true with the following caveats: 
we suppose not merely that $N$ is inert in $\Q(\sqrt{-D})$, but
that $-N$ is a square modulo $D$; and, owing to the denominators
potentially introduced in \S \ref{Shimuraproof},  it 
is only valid for $p$ sufficiently large, in the sense of Theorem 
\ref{thm:waldspurger-garrett-CM-general}.

\begin{proof}
By \eqref{eq:waldspurger-garrett-CM} and part (2) of Theorem \ref{thm:shimura-class},
     $$ \langle G, \mathfrak{S}\rangle = 
2 \langle   [1]\otimes [\psi],  \Sigma_0\otimes \Sigma_1 + \Sigma_1\otimes \Sigma_0 \rangle
    =2 \langle \Sigma_0, [1]\rangle \langle \Sigma_1, [\psi] \rangle $$
   where we have used the fact that   $\langle \Sigma_0, [\psi]\rangle =0$ since $\psi$ is non-trivial. Since $\langle \Sigma_0, [1]\rangle = h(\fO)$ by definition,
the theorem now follows from  
Proposition \ref{prop:elliptic-units-mt}
\end{proof}   
 
 To prove Theorem \ref{thm:main} of the introduction for CM weight one forms, it remains to relate the right-hand side of  
\eqref{eqn:almost-final-cm} to the expression $\red_N(u_g)$ occuring in this theorem;
this is done by the following Lemma.

\begin{lemma}
\label{lemma:regulator-cm}
Let $U_g  := ( \cO_H^\times \otimes \Ad^*(\rho_g)^\circ)^{G_\Q}$. 
There exists $u_g \in U_g$ with the property that, for all $N$ as above, 
$$ \log(\red_N(u_g)) = 2 \log(u_{\psi,\fq}).$$
\end{lemma}

This lemma concludes  the proof of Theorem \ref{thm:main},
after multiplying equality \eqref{eqn:almost-final-cm} by  $\frac{-6n}{1-\psi(\bar{\fq})} \in R$
with $n$ the norm of $1-\psi(\bar{\fq})$: 
$$ (-6n)  \langle G, \mathfrak{S} \rangle = \log(u_g'), \qquad
u_g' :=  \frac{-h(\mathfrak{o}) n}{ (1-\psi(\mathfrak{q}))}  \cdot u_g.
$$
where, in the last equality, we are implicitly using the $R$-module structure
on $U_g$ to form the product. 

\begin{proof}
For typographical simplicity we write just $u_{\psi}$ instead of $u_{\psi, \fq}$. 

Let $e_1$ be an eigenvector in $V_g$ for the action of $G_K$, on which $G_K$ acts via the character 
$\psi_1$. Since $N$ is inert in $K$, the associated Frobenius automorphism $\sigma_N\in G_{\Q}$ 
sends $e_1$ to a complementary vector $e_2 = \sigma_N(e_1)$, on which $G_K$ acts via the character 
$\psi_1'$.  Since $\sigma_N$ has determinant $-1$,  it then sends $e_2$ to $ e_1$. 
Representing  elements of ${\rm Ad}(V_g)$ as matrices relative to the basis $(e_1,e_2)$, 
so that 
$$  \rho_g(x) =\left(\begin{array}{cc} \psi_1(x) & 0   \\ 0 & \psi_1(x)^{-1} \end{array}\right), \ \ \mbox{ for } x \in G_K, \qquad \mbox{ and } \  \rho_g(\sigma_N) = \left(\begin{array}{cc} 0 & 1 \\ 1 & 0 \end{array}\right),$$
and using the  trace form to identify $\Ad(V_g)$ with its dual $\Ad^*(V_g)$, 
 we define
$u_g \in  U_g$  via
\begin{equation} \label{ugCMdef}  u_g := 
u_\psi\otimes  \left(\begin{array}{cc} 0 & 1 \\ 0 & 0 \end{array}\right) + \sigma_N(u_\psi) \otimes \left(\begin{array}{cc} 0 & 0 \\ 1 & 0 \end{array}\right).\end{equation}
Note that the matrices above do indeed define functionals on $\Ad(V_g)$ that  
send the image of $R[G_{\Q}]$ to $R$. 
We readily see that $u_g$ is in fact $G_{\Q}$-invariant,
  e.g. $\rho_g(G_K)$ acts on the $u_{\psi}$ through $\psi^{-1}$
and on the upper nilpotent matrix through $\psi=\psi_1^2$. 
 We then compute an equality inside $\mathcal{O}_H^{\times} \otimes R$:
 $$ \red_N(u_g) = \left\langle u_g, \rho_g(\sigma_N)\right\rangle = {\rm Trace} \left( \begin{array}{cc} u_\psi & 0 \\ 0 & \sigma_N(u_\psi) \end{array}\right) = 
 u_\psi + \sigma_N (u_\psi).$$
 (The reader is cautioned that  additive notation for the group law in $\cO_H^\times\otimes R/p^t$ has been used in this last equation.) 
 Since the discrete logarithm mod $N$ is equivariant for the action of $\sigma_N$, which acts trivially on $(\Z/N\Z)^\times$, we obtain the desired equality
 $$  \log(\red_N(u_g)) = \log( u_\psi + \sigma_N u_\psi)  = 2 \log(u_\psi).$$
 \end{proof}

\subsection{Proof of Conjecture \ref{conj:HV} for indefinite theta series}
 \label{sec:proof-rm2}

We now turn to prove  Conjecture \ref{conj:HV} when $g$ is an RM form.
We will be in the situation of \S \ref{main-id-RM} with $\psi_2 =\psi_1^{-1}$. 
More precisely,  let  
$
\psi_1: G_K \, \lra \, R^\times
$
be the finite order character of mixed signature as in the beginning of \S \ref{main-id-RM},
with values in the ring of integers of a finite extension $L$ of $\Q$, such that $g=\theta_{\psi_1}$ is the theta series associated to $\psi_1$ as described in \eqref{classicalO2}.
 Let $N$ be an odd prime and define $G\in S_2(\Gamma_0(N))$ as the
 trace to the space of modular forms of level $N$ of $ \theta_{\psi_1}(z) \theta_{\psi_1^{-1}}(Nz)$.
   As explained in the introduction, {\em the conjecture we address in this note becomes trivial when $N$ remains inert, and hence we assume throughout that it splits in $K$ as $N=\mathfrak{N} \cdot \bar{\mathfrak{N}}$.}

 The proof of Conjecture \ref{conj:HV},  which computes the pairing of $G$ with the Shimura class, again relies
crucially on the  $\Theta$-correspondence, namely the Hecke-equivariant map  
$$  \Theta: H_{1,\Bet}(X_0(N),{\rm cusps}; Z)^+ \otimes_{\mathbb{T}(N)}  H_{1,\Bet}(Y_0(N),Z)^- 
 \rightarrow M_2(N)$$
given by \newcommand{\realthetaconstant}{\frac{-1}{24}}
 \begin{equation}
\label{eqn:def-theta-rm}
 \Theta(\gamma^+ \otimes \gamma^-) = \realthetaconstant  \kappa_0^+(\gamma^+) \cdot
  \kappa_0^-(\gamma^-)   + \sum_{m \geq 1} \langle T_m \gamma^+,  \gamma^- \rangle q^m.
 \end{equation}
 Here $\kappa_0^{\pm}$ are as defined in   \S \ref{sec:he-betti-rel}
 and \S \ref{sec:he-betti}.  Note that the sign of $\realthetaconstant$ depends on orientation conventions implicit
 in the definition of the intersection pairing. 
 
 For lack of a suitable reference we sketch a proof. 
We identify the relative homology group
$H_{1,\Bet}(X_0(N), {\rm cusps}; Z)^+$ with $H^1_{\rm{B}}(Y_0(N), Z)^-$,
 a free $\mathbb{T}(N)$-module of rank one, and thus with $\mathbb{T}(N)$ itself. We can similarly identify $H_{1,\Bet}(Y_0(N), Z)^-$
 with its dual $M_2(N;Z)$. 
Adjusting these identifications if necessary, we can suppose that the Poincar{\'e} pairing $\langle - , - \rangle$
 corresponds to the pairing on $\mathbb{T}(N)\times M_2(N;Z)$ given by
  $(T, f ) \mapsto a_1(Tf)$,
 and $\Theta$ corresponds to $(T, f) \mapsto Tf$.   The formula
 \eqref{eqn:def-theta-rm} follows from this, up to the identification of the constant $\realthetaconstant$. 
 To compute the constant we take $\gamma^+$ the element represented by the geodesic from $0$ to $\infty$,
 and $\gamma^{-}$  a small loop around $\infty$
 and we fix orientations so that $\langle \gamma^+,  \gamma^- \rangle = 1$. 
 In particular,
 $$\langle T_m \gamma^+, \gamma^- \rangle= \sum_{d|m, (d,N)=1} d, \qquad 
\kappa_0^+(\gamma^+) = 1, \qquad \kappa^{-}_0(\gamma^{-}) =  N-1.$$
 The expansion on the right-hand side of \eqref{eqn:def-theta-rm} must represent
 $E_2^{(N)}$, and therefore this fixes the constant as $\realthetaconstant$. 
 
 We note in particular that $\kappa^+_0$ vanishes on the image of $H_{1,\Bet}(X_0(N))$ and so the formula above in
 fact matches with \eqref{eqn:Jgg} used in an earlier section.

As in the CM setting, we can observe that  -- with $\mathfrak{m}$ the Eisenstein ideal as before --
      \begin{itemize}
      \item   the modules $\bH_+ := H_{1,\Bet}(X_0(N),{\rm cusps}; \Z)^+_{\mathfrak{m}}$ and 
      $\bH_-= H_{1,\Bet}(Y_0(N),\Z)^- _{\mathfrak{m}}$, obtained from completing the singular homology of the complex modular curves,  are again   
       free $\TT$-modules of rank $1$.\footnote{Note that we get, by duality,
       an isomorphism of these with the (sign-altered) cohomological analogues: $\bH_+ \simeq \bH^-$ and $\bH_- \simeq \bH^+$, 
       so this result follows from its cohomological analogue.}
       
\item the map $\Theta_{\mathfrak{m}}$ is an isomorphism:
$ (\bH_+ \otimes_{\mathbb{T}} \bH_-) \longrightarrow \mathbb{M}$ and
so (cf. \eqref{Thetadef0}, \eqref{Thetadef1}) we have adjoint maps
\begin{equation} \label{Thetadef2} \Theta:  \bar\bH_+  \otimes_{\Tp} \bar \bH_- \simeq  \bar{\bM}, \ \  \Theta^*: \bar{\bM}^* \simeq (\bar\bH_+  \otimes_{\Tp} \bar \bH_- )^*.\end{equation}
Here bars denote tensoring with $\Z/p^t$ and $*$ denotes $\Hom(-, \Z/p^t)$. 
 
\end{itemize}
  The strategy of the proof of Conjecture \ref{conj:HV} is, much as in the case of CM theta series,
 to express the inner product
 $\langle G,  \mathfrak{S}\rangle$    as an inner product on $\bH_+ \otimes_{\TT}\bH_-$ via 
$\Theta$. 

We will follow the notation of
\S \ref{sec:trace-indef-statement}; in particular $\mathcal{C}$
is the narrow class group of $K$, and 
we have introduced  Heegner cycles $\gamma_I$ attached
to $I \in \mathcal{C}$, as well as weighted combinations
  $\gamma_{\psi}$  in \eqref{gammapsidef}.
  The following proposition plays a key role in the proof of Conjecture \ref{conj:HV} for RM forms, 
 since it is via this result that the relevant Stark unit -- in this case, a fundamental
 unit of the real quadratic field --  makes its appearance.
\begin{proposition}
\label{prop:heegner-cycles-mt}
For all even characters $\psi$ of the narrow Picard group $\mathcal{C}$, 
$$  \kappa_0^+(\gamma_\psi) =  0, \qquad \mbox{and } \quad
 \kappa_1^+(\gamma_\psi) = \begin{cases} 
  - h \log(u_K) & \mbox{ if } \psi =1, \\
0 & \mbox{ if } \psi \ne 1,
\end{cases} 
 $$
where $\kappa_0^+$ and 
  $\kappa_1^+ \in H^1_{\Bet}(X_0(N),{\rm cusps};\Z)^+$ are the Eisenstein and higher Eisenstein elements described in \S \ref{sec:he-betti-rel},
$h$ is the order of the narrow class group $\mathcal{C}$, 
  and $\log(u_K)$ refers to the logarithm of the reduction of $u_K$
  at the chosen divisor $\mathfrak{N}$ of $N$.\footnote{ The definition of $\gamma_\psi$ also depends on the choice of divisor of $N$, although this is not indicated in the notation.   One checks that the identity
  remains valid upon replacing $\mathfrak{N}$ by $\mathfrak{N}'$ on both sides. }
\end{proposition} 
\begin{proof}
The  assertion about $\kappa_0^+$ 
 follows from the fact that the Heegner cycles 
$\gamma_I$, viewed as cycles in the integral homology of $X_0(N)$ relative to the cusps, 
are in the kernel of the
boundary map $\partial$ of \eqref{eqn:boundary-hom}, 
 and hence are orthogonal to   $\kappa_0$. 

To show the second assertion, recall that the class $\kappa_1^+$ was defined 
 modulo $p^t$ by  choosing a discrete logarithm $\log: (\Z/N\Z)^\times \lra \Z/p^t\Z$, and setting
 $$ \kappa_1^+ \left(\begin{matrix} a & b \\ c & d \end{matrix}\right)  = \log(a).$$
With this choice we have
\begin{equation}
\label{eqn:log-of-ek}
 \kappa_1^+(\gamma_I) =  - \log(u_K), 
 \end{equation}
 where $u_K$ is a
  fundamental unit of norm $1$ of the real quadratic field $K$,
  $\log(u_K)$ refers to the logarithm of the reduction of $u_K$
  at $\mathfrak{N}$. .
  This is because (notation of \S \ref{sec:trace-indef-statement})
  the cycle $\gamma_I$ arises from an embedding $\mathfrak{o} \rightarrow M_0(N)$,
 with respect to which
   the ring homomorphism  sending a matrix in $M_0(N)$
   to the mod $N$ reduction of upper left hand entry  restricts to reduction modulo $\mathfrak{N}$ on $\fO$ (see discussion above \eqref{eqn:Heegner-cycles});
   the sign arises for the orientation reason noted below \eqref{CCsign}.
      Equation  \eqref{eqn:log-of-ek}
  therefore  implies that $\kappa_1^+(\gamma_\psi) =  -(\sum_\fa \psi(\fa)) \log(u_K)$,
and the 
   result follows.
  \end{proof}

%
  \begin{proposition}
 \label{prop:mild-dependence}
 For all totally odd ring class characters $\psi$, 
  $$ \kappa_0^-(\gamma_\psi) =   (1-\psi(\fN))  L_{\mathrm{alg}}(\psi),$$
  where    $\kappa_0^-\in H^1_{\Bet}(Y_0(N),\Z)^-$ is as defined in Section \ref{sec:he-betti}, and 
  $L_{\mathrm{alg}}(\psi) \in R$ will be defined in \eqref{Lalgdef} and is in particular independent of $\mathfrak{N}$.
%
 \end{proposition}
 Recall that $\kappa_0^{-}$ arises 
 from the Dedekind-Rademacher function $\varphi$ of \eqref{eqn:zagier-rho} which encodes the periods of the (complex!)
 logarithm of the modular unit $\Delta(Nz)/\Delta(z)$. 
The proposition shows that $\kappa_0^-(\gamma_\psi)$ exhibits a mild dependence on $N$ through the 
  factor $(1-\psi(\fN))$.
  
 \begin{proof} 
 The issue to be dealt with here is, essentially, passage from level $1$ to level $N$. 
 Let $I \in \mathcal{C}$.  Choose a representative that
  is relatively prime to $N$ and an oriented basis $(e_1,e_2)$.
  The element 
  $$ \eta_I = \left(\begin{array}{cc} a & b \\ c & d \end{array}\right) \in \SL_2(\Z), \qquad 
  \mbox{ where } \quad \begin{array}{l} u_K e_1 = a e_1 + c e_2 \\ u_K e_2 = b e_1 + d e_2\end{array}$$
  has conjugacy class in $\SL_2(\Z)$ that does not depend on the choice of oriented basis,  
  and in particular $\varphi(\eta_{\fa})$ is well-defined.   
Choose $(e_1, e_2)$ so that  $e_2$ belongs to $I \cap \fN$, and observe then that $(e_1',e_2') := (N e_1,e_2)$ is an oriented basis for $I\fN$ and that $u_K$ acts on this basis according to the rule
  $u_K e_1' = a e_1' + (cN) e_2'$ and $u_ke_2' = (b/N) e_1' + d e_2'$.
By \eqref{eqn:DR-on-N} as well as the definition \eqref{alphadef} of cycles $\gamma_J$, we get
   $\kappa_0^-(\gamma_{I \fN}) = \varphi(\eta_{I \fN}) - \varphi(\eta_{I})$ and it follows that
  \begin{eqnarray*}
   \kappa_0^-(\gamma_\psi) &=&
  \sum \psi(I \fN)  \left( \varphi(\eta_{I \fN}) -  \varphi(\eta_{I}) \right) \\
  &=& 
(1-\psi(\fN))    \sum_{I} \psi(I) \varphi(\eta_{I}), 
\end{eqnarray*}
  and we obtain the result upon defining 
  \begin{equation}
   \label{Lalgdef}
   L_{\mathrm{alg}}(\psi) :=  \sum_{I} \psi(I)^{-1} \varphi(\eta_{I}).\end{equation} 
  \end{proof}

  \begin{remark}  \label{Lpsi}  As is implicit in the notation, $L_{\rm alg}(\psi)$ is closely related to the ``algebraic part" of   the $L$-series
 $L(\psi,s)=
 \sum_{\fa \lhd \fO_K} \psi(\fa) (N\fa)^{-s}$
  attached to $\psi$, at $s=1$. 
The justification for this is given by Meyer's analogue of the Kronecker limit formula for real quadratic fields  (cf.~\cite[\S 4]{zagier}) which asserts that,  at least for all   unramified, totally odd characters $\psi$ of the narrow Hilbert
  class field of $K$,
  $ L_{\rm alg}(\psi)   =   \frac{12\sqrt{D}}{\pi^2} L(\psi^{-1}, 1).$
 \end{remark}
\vspace{0.2cm}

Note  that if $x^2-a_N(g) + \chi_K(N) = (x-\alpha_N)(x-\beta_N)$   is
 the $N$-th Hecke polynomial attached to $g$, then we may order $\alpha_N$ and $\beta_N$ in such a way that
  $$ \alpha_N = \psi_1(\fN), \quad \beta_N = \psi_1(\fN'), \mbox{ and so }
 \psi(\fN) = \psi_1(\fN) / \psi_1'(\fN) = \alpha_N/\beta_N,$$
 where we use the definition \eqref{psidef}.   
 Proposition   \ref{prop:mild-dependence}
 can then be rewritten as 
\begin{equation}
\label{eqn:eval2}
  \kappa_0^-(\gamma_{\psi})  =  \left(1-\alpha_N/\beta_N\right) \times  L_{\rm alg}(\psi).
\end{equation}

\vspace{0.3cm}

%

 
 Let $\fS_0$ and $\fS = \fS_1 \in \bar\bM^*$ denote the Eisenstein and higher classes
described in \S \ref{sec:he-Mdual}.
It follows from  Theorem \ref{thm:RM-general}   applied to the pair $(\psi_1, \psi_1^{-1})$
-- so  by \eqref{psi12def}    $\psi_{12} =1$ and $\psi_{12'} = \psi_1/\psi_1' =\psi$
 -- 
 that there exists $C_g \in R$ independent of $N$ such that    
\begin{eqnarray}
\label{eq:waldspurger-garrett-RM}
\langle G, \mathfrak{S}\rangle &=&   \beta_N \, C_g \langle \Theta([\gamma_{1}] \otimes [\gamma_{\psi}]), \mathfrak{S}\rangle \\
\nonumber
 &=&  \beta_N  C_g \cdot \langle[ \gamma_1] \otimes [\gamma_{\psi}], \Theta^*(\mathfrak{S})\rangle, 
\end{eqnarray}
where  we understand $[ \gamma_1] \otimes [\gamma_{\psi}]$ as an element of
$ (\bar\bH_+  \otimes_{\Tp} \bar \bH_- )$, and $\Theta^* (\mathfrak{S})$
as an element of the $\Z/p^t$-dual, see
 \eqref{Thetadef2}. 
 Here we regard $\Theta$ as normalized as in  
\eqref{eqn:def-theta-rm}; the $C_g$ that appears in the above equation only agrees with that constant
appearing in Theorem \ref{thm:RM-general} up to sign, 
 arising from the fact that the choice of orientation convention for \eqref{eqn:def-theta-rm}
was not compared with the choice of orientation convention used in Theorem \ref{thm:RM-general}.
This sign may be computed by the enthusiastic reader.

The next theorem below, which determines the image of $\mathfrak{S}$ under $\Theta^*$, 
plays exactly the same role in the RM proof as  
Theorem \ref{thm:shimura-class} in the CM setting.

 \begin{theorem}
\label{thm:shimura-class-bis}
We have 
\begin{enumerate}
\item
$\Theta^*(\fS_0) = \realthetaconstant \kappa_0^+ \otimes \kappa_0^-$;
\item  $
 \Theta^* (\mathfrak{S}) \equiv  \realthetaconstant (\kappa_1^+ \otimes \kappa_0^- + \kappa_0^+ \otimes \kappa_1^-)$  modulo $\kappa_0^+ \otimes \kappa_0^-$. 
  \end{enumerate}
where $\kappa^+$ are the Eisenstein classes of  
 \S \ref{sec:he-betti-rel}, or rather their image in $( \mathbb{H}_+)^*$
 or $(\overline{\mathbb{H}}_+)^*$, and similarly
 $\kappa^-$ are similarly defined from the Eisenstein classes of 
\S \ref{sec:he-betti}. 
 \end{theorem}
 
The statements should be interpreted just as in Theorem \ref{thm:shimura-class}:
we use 
$$  \left( \overline{\mathbb{H}}_+ \otimes_{\Tp} \overline{\mathbb{H}}_- \right)^*
\subset \left( \overline{\mathbb{H}}_+ \otimes_{\Z/p^t} \overline{\mathbb{H}}_- \right)^*  =
 (\overline{\mathbb{H}}_+)^* \otimes_{\Z/p^t} (\overline{\mathbb{H}}_-)^*  ,$$
where $*$ means $\Hom(-, \Z/p^t)$.

\begin{proof}
The first part of the theorem follows directly from the definition of
$\Theta$ given in \eqref{eqn:def-theta-rm}. The second follows from the Hecke equivariance of $\Theta^*$, in light of the fact that $\kappa_1^+\otimes \kappa_0^- + \kappa_1^+\otimes \kappa_0^-$ is the higher Eisenstein element in
$(\bH^+\otimes_{\TT} \bH^-)^\vee$ attached to $\kappa_0^+\otimes \kappa_0^-$,
by Proposition \ref{prop:he-tensor}.
\end{proof}

We can now prove Conjecture \ref{conj:HV} in the RM setting.

 \begin{proposition}
 \label{prop:almost-final-rm} 
We have
\begin{equation}
\label{eqn:almost-final-rm} 
 \langle G, \mathfrak{S} \rangle = \frac{1}{24} h(\mathfrak{o}) C_g
  \cdot L_{\rm alg}(\psi) \cdot (\beta_N-\alpha_N)  \cdot \log(u_K).
  \end{equation}
\end{proposition}
\begin{proof}
Applying  \eqref{eq:waldspurger-garrett-RM} 
and part (2) of Theorem \ref{thm:shimura-class-bis},
     $$ \langle G, \mathfrak{S}\rangle =  \frac{-\beta_N  C_g}{24} \cdot
  \langle  \gamma_1\otimes \gamma_{\psi},  \kappa_0^+\otimes \kappa_1^- + \kappa_1^+\otimes \kappa_0^- \rangle
    =  \frac{-\beta_N  C_g}{24} \cdot \kappa_1^+(\gamma_1) \cdot \kappa_0^-(\gamma_{\psi}),
    $$
   where we have used the fact that   $\kappa_0^+(\gamma_1) =0$  
   to ignore the term arising from $\langle \gamma_1\otimes \gamma_\psi, \kappa_0^+\otimes \kappa_1^-\rangle$.
The theorem now follows from Proposition
\ref{prop:heegner-cycles-mt} and \eqref{eqn:eval2}, which imply that
$$
   \kappa_1^+(\gamma_1) =  - h \log(u_K), \qquad 
  \kappa_0^-(\gamma_\psi) =  (1-\alpha_N/\beta_N) \cdot L_{\rm alg}(\psi).$$
\end{proof}   
 To prove Theorem \ref{thm:main} of the introduction
when $K$ is a real 
quadratic field, it remains, as before,  to relate the right-hand side of  
\eqref{eqn:almost-final-rm} to the expression $\red_N(u_g)$ occuring in this theorem.


\begin{lemma}
\label{lemma:regulator-multiverse} (cf. Lemma 
\ref{lemma:regulator-cm}).
Let $U_g   := ( \cO_K^\times \otimes \Ad^*(\rho_g)^{\circ})^{G_\Q}$. 
There exists $u_g \in U_g $ with the property that, for all $N$ as above, 
$$ \log(\red_N(u_g)) = (\alpha_N-\beta_N) \log(u_K).$$
\end{lemma}

As before,   Theorem  \ref{thm:main}
will follow from this:
we have
$$24 \langle G, \mathfrak{S} \rangle = \log(\mathrm{red}_N(u_g')),$$
with $u_g' = -h(\mathfrak{o}) L_{\mathrm{alg}}(\psi) C_g \cdot u_g$. 

%

\begin{proof}
Let $e_1$ and $e_2$ be  eigenvectors in $V_g$ for the action of $G_K$, on which $G_K$ acts via the characters
$\psi_1$ and $\psi_1'$ respectively. 
Since $N$ is split in $K$, the associated Frobenius automorphism $\sigma_N\in G_{\Q}$ 
is a diagonal matrix, with entries $\alpha_N$ and $\beta_N$.   
Representing  elements of ${\rm Ad}(V_g)$ as matrices relative to the basis $(e_1,e_2)$, 
so that $ \rho_g(\sigma_N) = \left(\begin{array}{cc} \alpha_N & 0 \\ 0 & \beta_N \end{array}\right)$, and using the   trace form to identify $\Ad(V_g)$ with its dual $\Ad^*(V_g)$, we define
$$ u_g := u_K\otimes  \left(\begin{array}{cc} 1 &0  \\ 0 & -1 \end{array}\right),$$ 
which is clearly $G_{\Q}$-invariant: it is fixed by $G_K$, and the nontrivial automorphism of $K$
negates both factors.  As after \eqref{ugCMdef} this indeed defines an element of $U_g$.
 One then finds
\begin{eqnarray*}
 \red_N(u_g)&=&  \left\langle u_g, \rho_g(\sigma_N)\right\rangle \ \  = \ \ {\rm Trace} \left( \begin{array}{cc} 
 u_K \otimes \alpha_N & 0 \\ 0 & u_K \otimes (-\beta_N)
  \end{array}\right) \\
  &=& 
u_K \otimes (\alpha_N-\beta_N).
\end{eqnarray*}
Therefore,   
 $$  \log(\red_N(u_g)) =  (\alpha_N-\beta_N) \log(u_K).$$
The lemma follows.
\end{proof}





\end{document}